\documentclass[reqno,12pt]{amsart}
\usepackage{a4wide}
\usepackage{amsmath,amssymb,latexsym,soul,cite,mathrsfs}
\usepackage{color,graphicx}
\usepackage[colorlinks=true,
urlcolor=blue,
citecolor=red,
linkcolor=red,linktocpage,pdfpagelabels,
bookmarksnumbered,bookmarksopen]{hyperref}
\usepackage[english]{babel}
\usepackage{color,enumitem,graphicx}
\usepackage[norefs,nocites]{refcheck}

\numberwithin{equation}{section}
\newtheorem{theorem}{Theorem}[section]
\newtheorem{proposition}[theorem]{Proposition}

\newtheorem{lemma}[theorem]{Lemma}
\newtheorem{thmA}{Theorem A}

\theoremstyle{definition}
\numberwithin{equation}{section}

\begin{document}
\title[Non-degeneracy of solutions for Lane-Emden systems]
{Non-degeneracy of  solution for critical Lane-Emden
systems with linear perturbation}

 \author{Yuxia Guo,  Yichen Hu  and  Shaolong Peng}

\address{  Department of Mathematical Science, Tsinghua University, Beijing, P.R.China}
\email{yguo@tsinghua.edu.cn}

\address{  Department of Mathematical Science, Tsinghua University, Beijing, P.R.China}
\email{hu-yc19@mails.tsinghua.edu.cn}

\address{Academy of Mathematics and Systems Science, Chinese Academy of Sciences, Beijing, 100190,  P.R.China}
\email{slpeng@amss.ac.cn}

\thanks{Y. Guo was partially supported by National Key R\&D Program (2023YFA1010002) and NNSF of China (No.12271283 \& 12031015).}

\begin{abstract}
 In this paper, we consider the following elliptic system
\begin{equation*}
\begin{cases}
-\Delta u = |v|^{p-1}v +\epsilon(\alpha u + \beta_1 v), &\hbox{ in }\Omega,
\\-\Delta v = |u|^{q-1}u+\epsilon(\beta_2 u +\alpha v), &\hbox{ in }\Omega,
\\u=v=0,&\hbox{ on }\partial\Omega,
\end{cases}
\end{equation*}
where $\Omega$ is a smooth bounded domain in $\mathbb{R}^{N}$, $N\geq 3$, $\epsilon$ is a small parameter, $\alpha$, $ \beta_1$ and $ \beta_2$ are real numbers,  $(p,q)$ is a pair of positive numbers lying on the critical hyperbola
\begin{equation*}
\begin{split}
\frac{1}{p+1}+\frac{1}{q+1} =\frac{N-2}{N}.
\end{split}
\end{equation*}
We first revisited the  blowing-up solutions  constructed in \cite{Kim-Pis} and then we proved its non-degeneracy. We believe that the various new ideas and technique computations that we used in this paper  would be very useful to deal with other related problems involving critical Halmitonian system and the construction of new solutions.
\end{abstract}

\maketitle


{\bf AMS} Subject Classification: 35B33; 35J47; 35J67.

\section{Introduction}

We consider the following elliptic system
\begin{equation}\label{inteq1}
\begin{cases}
-\Delta u = |v|^{p-1}v +\epsilon(\alpha u + \beta_1 v), &\hbox{ in }\Omega,
\\-\Delta v = |u|^{q-1}u+\epsilon(\beta_2 u +\alpha v), &\hbox{ in }\Omega,
\\u=v=0, &\hbox{ on }\partial\Omega,
\end{cases}
\end{equation}
where $\Omega$ is a smooth bounded domain in $\mathbb{R}^{N}$, $N\geq 3$, $\epsilon$ is a small parameter, $\alpha$, $ \beta_1$ and $ \beta_2$ are real numbers,  $(p,q)$ is a pair of positive numbers lying on the critical hyperbola
\begin{equation}\label{inteq4}
\begin{split}
\frac{1}{p+1}+\frac{1}{q+1} =\frac{N-2}{N}.
\end{split}
\end{equation}

If $u=v$, system \eqref{inteq1} is reduced to the following  Brezis-Nirenberg problem \cite{Bre-Nir}
\begin{equation}\label{inteq2}
 \begin{cases}
-\Delta w = |w|^{\frac{4}{N-2}}w +\lambda w, &\hbox{ in }\Omega,
\\w=0, &\hbox{ on }\partial\Omega.
\end{cases}
\end{equation}
It is known that if  $ \lambda\leq 0$ and $\Omega$ is a star-shaped domain,  the classical Pohozaev identity  \cite{P} implies that problem \eqref{inteq2} has no solution. On the other hand, if $N\geq 4$ and $0<\lambda<\lambda_1$, \eqref{inteq2}  has a solution (see \cite{Bre-Nir}, \cite{Cap-For-Pal}). Here and after, $ \lambda_{n}(\Omega )$
is the n-th eigenvalue of the Laplacian $-\Delta$ with Dirichlet boundary condition.

System \eqref{inteq1} has attracted a lot of interest in the past few decades. For example, Mitidieri (\cite{Mit}) and Van der Vorst (\cite{Van} \cite{Hul-Mit-Van}) proved  that there is no  solution in a star-shaped domain $ \Omega$ provided that the matrix
$
 \begin{pmatrix}
-\frac{\beta_2(q-1)}{2(q+1)} & -\frac{\alpha}{N} \\ -\frac{\alpha}{N} & -\frac{\beta_2(p-1)}{2(p+1)}
\end{pmatrix}
$
is  semi-positive definite. In particular, we have the non-existence result if $p, q > 1 $, $ \alpha =0$
and $\beta_1, \beta_2 \leq 0$. And  Hulshof, Mitidieri and Van der Vorst \cite{Hul-Mit-Van} proved
that if $p, q > 1 $, $ \alpha \geq 0$ and either $ \beta_1>0$ or $ \beta_1=0$ and $ \beta_2>0$, then \eqref{inteq1} has a solution
provided that $ \epsilon^{2}\beta_1\beta_2\neq \lambda_{n}^{2}$ for all $n\in \mathbb{N}$ and $N$ is sufficiently large. They also showed
that the solution is positive if $\beta_1, \beta_2 > 0$ and $ \epsilon^{2}\beta_1\beta_2\neq \lambda_{1}^{2}$. Their approach relies on a
dual formulation due to Clarke and Ekeland \cite{Cla-Eke}.

Very recently, for $N \geq 8$, $p \in (1, \frac{N-1}{N-2} )$, and $(p, q)$ satisfies \eqref{inteq4}, if one of the following conditions is satisfied:
\begin{equation}\label{condition}
\hbox{(i) }\beta_1>0, \ \ \hbox{(ii) }\beta_1=0 \hbox{ and }\alpha>0, \ \ \hbox{(iii) }\beta_1=\alpha=0 \hbox{ and }\beta_2>0,
\end{equation}
in their elegant paper, Kim and Pistoia \cite{Kim-Pis} proved the new existence of blowing-up solutions for system \eqref{inteq1}.

The aim of the present  paper is to discuss the non-degeneracy of the  solutions
constructed in \cite{Kim-Pis}. We would like to point out that the non-degeneracy is very important in the construction of new solutions and in the computation of Morse index of the solution. Before the statement of the main results, let us briefly introduce the blowing-up
solutions constructed in \cite{Kim-Pis}.

Let $N\geq 3$, $(p, q)$ be a pair of positive numbers such that $p \in ( 1 , \frac{N}{N-2} )$ and satisfy
\eqref{inteq4}, and let $(U, V )\in \dot{W}^{2,\frac{p+1}{p}}(\mathbb{R}^{N})\times \dot{W}^{2,\frac{q+1}{q}}(\mathbb{R}^{N})$ be  a positive ground state solution to
\begin{equation}\label{inteq5}
 \begin{cases}
-\Delta U  = |V|^{p-1}V, \hbox{ in }\mathbb{R}^{N},
\\-\Delta V  = |U|^{q-1}U, \hbox{ in }\mathbb{R}^{N}.
\end{cases}
\end{equation}
It is known that $(U, V)$  is radially symmetric and decreasing (see
\cite{Lions}). And, the ground state solution is unique up to translation and scaling. That is
 the family of functions
$\{(U_{\mu,P}(x), V_{\mu,P}(x))\}$ given by
\begin{equation*}
 \begin{split}
&(U_{\mu,P}(x), V_{\mu,P}(x))=( \mu^{-\frac{N}{q+1}}U( \mu^{-1}(x-P) )   ,\mu^{-\frac{N}{p+1}}V( \mu^{-1}(x-P) )      ),
\end{split}
\end{equation*}
for any $\mu>0,P\in\mathbb{R}^{N}$, exhausts all the positive ground states of \eqref{inteq5} ( see Wang \cite{Wang} and Hulshof and Van der Vorst \cite{Hul-Van}  ).

We denote $( U_{1},V_{1}) = (U_{\mu_{1},P_{1}}(x), V_{\mu_{1},P_{1}}(x))$, where $\mu_{1}=d_{1}\mu$ and $d_{1}$ is a positive constant, $P_1\in \mathbb{R}^{N}.$  Set $(PU_{1}, PV_{1})$ be the solution of the system
\begin{equation}\label{inteq7}
 \begin{cases}
-\Delta PU_{1} = V_{1}^{p}, &\hbox{ in }\Omega,
\\-\Delta PV_{1}  = U_{1}^{q}, &\hbox{ in }\Omega,
\\PU_{1}=PV_{1}=0,&\hbox{ on }\partial\Omega.
\end{cases}
\end{equation}
And let $PU_{d_{1},P_{1}}$ be the unique smooth solution of the probelm
\begin{equation*}
 \begin{cases}
-\Delta PU_{d_{1},P_{1}} = PV_{1}^{p} &\hbox{ in }\Omega,
\\PU_{d_{1},P_{1}}=0&\hbox{ on }\partial\Omega.
\end{cases}
\end{equation*}
Set $\frac{1}{p^{*}}=\frac{1}{q+1}+\frac{1}{N}$ and $\frac{1}{q^{*}}=\frac{1}{p+1}+\frac{1}{N} $,
we define the Banach space
\begin{equation*}
X_{p,q}:=\bigg(   W^{2,\frac{p+1}{p}}(\Omega)\cap W_{0}^{2,p^{*}}(\Omega)\bigg)\times\bigg(   W^{2,\frac{q+1}{q}}(\Omega)\cap W_{0}^{2,q^{*}}(\Omega)\bigg)
\end{equation*}
equipped with the norm
\begin{equation*}
||(u,v)||_{X_{p,q}}:=||\Delta u||_{L^{ \frac{p+1}{p}}}+||\Delta v||_{L^{ \frac{q+1}{q}}}.
\end{equation*}
Let $G$ be the Green's function of the Laplacian $-\Delta$ in $\Omega$ with Dirichlet
boundary condition. And $H$ be its regular part, then $G(x,y)=S(x,y)-H(x,y)$ with $S(x,y) = \frac{\gamma_{N}}{|x-y|^{N-2}}$, where $\gamma_{N}=\frac{1}{(N-2)|\mathbb{S}^{N-1}|}$. Then set $ \tilde{G}(x,y)= \int_{\Omega}G(x,z)G^{p}(z,y)dy$, $ \tilde{H}(x,y) =  \int_{\Omega}G(x,z)G^{p}(y,z)dy - \int_{\mathbb{R}^{N}}S(x,y)S^{p}(z,y)dy$ and $\tau(x) =\tilde{H}(x,x) $.

 The result obtained in \cite{Kim-Pis} states as the following.
 \begin{thmA}
 Assume that $N \geq 8$, $p \in (1, \frac{N-1}{N-2} )$, and $(p, q)$ satisfies \eqref{inteq4}. If one of the conditions in \eqref{condition} is satisfied,
then there
exists a small number $\epsilon_{0}>0$ depending only on $N$, $p$, $\Omega$, $\alpha$, $\beta_{1}$ and $\beta_{2}$ such that for any
$\epsilon \in (0, \epsilon_{0})$, system \eqref{inteq1} has a solution in $(C^{2}(\Omega))^{2}$ with the form
\begin{equation*}
(u_{\epsilon},v_{\epsilon}) = (PU_{d_{1,\epsilon},P_{1,\epsilon}}, PV_{1,\epsilon} ) + ( \psi_{\epsilon},\phi_{\epsilon}),
\end{equation*}
where
\begin{equation*}
  \mu = \epsilon^{N_{\beta_1,\beta_2,\alpha}}=
  \begin{cases}
  \frac{p+1}{(N-2)p^2-4p+N-2}, \qquad &\text{if} \,\,\, \beta_1>0,\\
  \frac{1}{(N-2)p-4}, \qquad \qquad &\text{if} \,\,\,\, \beta_1=0 \,\,\text{and}\,\,\alpha>0, \\
\frac{q+1}{N(p-q+2)}, \qquad \qquad &\text{if} \,\,\,\, \beta_1=\alpha=0 \,\,\text{and}\,\,\beta_2>0,
   \end{cases}
\end{equation*}
 $d_{1,\epsilon}\to d_{\beta_1,\beta_2,\alpha}>0$ (where $ d_{\beta_1,\beta_2,\alpha,N}$ represents a constant related to ${\beta_1,\beta_2,\alpha,N}$), $P_{1,\epsilon}\to P_{0} (P_{0}\in \Omega \hbox{ and }\nabla \tau (P_0) =0)$  and $ ||( \psi_{\epsilon},\phi_{\epsilon}) ||_{X_{p,q}} =O(\mu^{\frac{Npq}{q+1}})$, as $\epsilon \to 0$. Moreover, for any
$\epsilon \in (0, \epsilon_{0})$, we also have $ d_{1,\epsilon}\in (\delta_{1}^{-1},\delta_{1}) $ and $ dist( P_{1,\epsilon},\Omega) > \delta_{2}$, where $ \delta_{1}$ and $ \delta_{2}$
 are small fixed positive constants.

 Moreover, if $\beta_{1},\beta_{2}\geq 0$, then system \eqref{inteq1} has a solution with positive components showing
the prescribed blowing-up behavior.
 \end{thmA}
Our main result is as follows:
 \begin{theorem}\label{Tmain}
 Assume that $N \geq 8$, $p \in (1, \frac{N-1}{N-2} )$, and $(p, q)$ satisfies \eqref{inteq4}. If $\beta_{1}=\alpha=0,\hbox{ } \beta_{2}>0$ and $P_{0}$ is the non-degenerate critical point of $\tau(x)$,
then there
exists a small number $\tilde{\epsilon}_{0}>0$ depending  on $N$, $p$, $\Omega$ and $\beta_{2}$ such that for any
$\epsilon \in (0, \tilde{\epsilon}_{0})$,  the solution $(u_\epsilon, v_\epsilon)$ constructed in Theorem A is non-degenerate in the sense that if $(\eta,\xi)\in (H^{1}_{0}(\Omega))^{2}$ is a pair solution of the following problem:
  \begin{equation*}
 \begin{cases}
-\Delta \eta = pv_{\epsilon}^{p-1}\xi, &\hbox{ in }\Omega,
\\-\Delta \xi  =q u_{\epsilon}^{q-1}\eta +\epsilon\beta_2 \eta, &\hbox{ in }\Omega,
\\\eta=\xi=0,&\hbox{ on }\partial\Omega.
\end{cases}
\end{equation*}
Then $(\eta ,\xi)=(0,0)$.
 \end{theorem}
 {\bf{ Remark 1.}} We would like to mention that, Theorem \ref{Tmain} is still true for case (i) $\beta_1>0$ and case (ii) $\beta_1=0 \,\,\text{and}\,\,\alpha>0$. We can also prove that the non-degeneracy of $(u_\epsilon, v_\epsilon)$ in case (i) and case (ii), similar to the proof approach used in Theorem \ref{Tmain},  so we leave the details to
readers.

 {\bf{ Remark 2.}} When $ p=1$ and $\alpha = 0$, system \eqref{inteq1} degenerates into the following single equation
 \begin{equation}\label{danfang}
     \begin{cases}
         (-\Delta)^{2} u = (1+\epsilon\beta_1)|u|^{\frac{8}{N-4}}u + (1+\epsilon\beta_1)\epsilon\beta_2 u ,\hbox{ in }\Omega,
         \\u = \Delta u = 0, \hbox{ on }\partial\Omega.
     \end{cases}
 \end{equation}
 Given the existence results in \cite{Kim-Pis} and our non-degeneracy, we believe that we can construct solutions to equation \eqref{danfang} and prove their non-degeneracy. This will be our future work.

{\bf{Remark 3.}} When $ p=1$, $\tilde{G}$ becomes the Green’s function of the bi-Laplacian $\Delta^{2}$ in $\Omega$ with
the Navier boundary condition and $\tau$ is the corresponding Robin function. The second-order case of equation \eqref{danfang} is as follows.
\begin{equation}\label{danfang2}
     \begin{cases}
         (-\Delta) u = |u|^{\frac{4}{N-2}}u + \epsilon\beta_2 u ,\hbox{ in }\Omega,
         \\u = 0, \hbox{ on }\partial\Omega.
     \end{cases}
 \end{equation}
 As well know that, equation \eqref{danfang2} has explosive solutions at the critical point of Robin function. Moerover, if the critical point of Robin function is non-degenerate, then the explosive solutions is non-degenerate. From this, we can see that our non-degenerate result is consistent with this.

To prove  the theorem, roughly speaking, we will use the local Pohozaev identities for system and proceed a contrary argument. For this purpose, we have to first obtain the point wise estimates for the solutions in Theorem {\bf{A}}. Hence we have to revisit  problem and give a different
proof of Theorem A, so that we can obtain the  necessary estimates that we  needed in the proof of the
non-degeneracy result. The main ingredients of the paper are the following:

First, we note that the remainder terms in the decomposition of the corresponding linearized system will be very complicated, more extra techniques and careful computation are needed.

Second, due to the decay rate of the ground state,  $\displaystyle \int_{\mathbb R^{N}}U^{q}_{1,0}$ is integrable, but $\displaystyle\int_{\mathbb R^{N}}V^{p}_{1,0}$ is non-integrable, which leads to  new challenges to be overcome.

Third,  $PU_{d_1,p_1}$ has no $C^2$ estimate, which is not enough for our problem. To overcome this difficulty, we have to estimate carefully the term  $\displaystyle {\frac{  \partial PU_{d_{1,\epsilon},P_{1,\varepsilon}} -U_{1,\epsilon} }{\partial x_i}(y) -\frac{  \partial PU_{d_{1,\epsilon},P_{1,\epsilon}} -U_{1,\epsilon} }{\partial x_i}(P_{0})},$ some more extra ideas and technique are needed.

At last,  since the system is a Hamiltonian type system, which is strongly indefinitely. We have to  use  the local Pohozaev identities in a quite different way, to  decide where and how to use the identities.  We believe that our  variety new ideas and technique computations  will help  to deal with other related problems involving critical exponents and Hamiltonian-type system.

 We would like to point out that the local Pohozaev identities play an important role in   many problems, for examples, the existence of infinitely bubbling solutions (see \cite{Guo2020, PWY}),  the local uniqueness  for bubbling solutions (see  \cite{dly, GNNT,  GPY}) and the non-degeneracy of bubbling solutions (see \cite{guo1}). To our best knowledge, it is the first time to use local Pohozaev identities to study the problems involving Hamiltonian system with critical exponents.

The paper is organized as follows. In Section 2, we establish the local Pohozaev identities. In
order to obtain the point wise  estimates for the solutions, we revisit the proof of Theorem {\bf{A}} in Section 3.
Section 4 is devoted to the proof of Theorem 1.1. Some important estimations and preliminary results are attached in the Appendix. Indeed, some of the estimates in this part are independently interesting, which we believe that they are very useful for other related problems.

\section{Local Pohozaev identities}

Let
\begin{equation*}
    \begin{cases}
    -\Delta u = v^{p}+\epsilon (\alpha u+\beta_1 v), \hbox{ in } \Omega,
    \\-\Delta v = u^{q} + \epsilon (\beta_2 u+\alpha v), \hbox{ in } \Omega,
    \\ u = v =0, \hbox{ on } \partial\Omega,
    \end{cases}
    \text{and}\,\,
    \begin{cases}
    -\Delta \eta = p v^{p-1}\xi+\epsilon (\alpha \eta+\beta_1 \xi), \hbox{ in } \Omega,
    \\-\Delta \xi = q u^{q-1}\eta + \epsilon (\beta_2 \eta+\alpha \xi), \hbox{ in } \Omega,
    \\ \eta = \xi =0, \hbox{ on } \partial\Omega.
    \end{cases}
\end{equation*}

Assume that $\Omega_1$  is a smooth domain in $ \Omega$. We have the following identities.
\begin{lemma}\label{le2.1}
It holds
\begin{equation}\label{2.1}
\begin{split}
    &\int_{\partial\Omega_1}(-\frac{\partial u}{\partial\nu}\frac{\partial\xi}{\partial x_i} -\frac{\partial \xi}{\partial\nu}\frac{\partial u}{\partial x_i} +\frac{\partial u}{\partial\nu}\frac{\partial\xi}{\partial\nu}\nu_i)ds+\int_{\partial\Omega_1}(-\frac{\partial v}{\partial\nu}\frac{\partial\eta}{\partial  x_i} -\frac{\partial \eta}{\partial\nu}\frac{\partial v}{\partial  x_i} +\frac{\partial v}{\partial\nu}\frac{\partial\eta}{\partial\nu}\nu_i)ds
    \\=&\int_{\partial\Omega_1}(u^{q}\eta+v^{p}\xi)\nu_ids+\epsilon\alpha \int_{\partial\Omega_1}u\xi\nu_i+v\eta \nu_i ds+\epsilon\beta_1\int_{\partial\Omega_1}v\xi\nu_i ds+\epsilon\beta_2\int_{\partial\Omega_1}u\eta\nu_i ds,
    \end{split}
\end{equation}
and
\begin{equation}\label{2.2}
    \begin{split}
       &\int_{\partial\Omega_1}(-\frac{\partial u}{\partial\nu}\langle \nabla\xi,x-y_0 \rangle -\frac{\partial \xi}{\partial\nu}\langle \nabla u,x-y_0 \rangle+\langle  \nabla u,\nabla\xi\rangle\langle \nu,x-y_0\rangle)ds \\&+\int_{\partial\Omega_1}(-\frac{\partial v}{\partial\nu}\langle \nabla\eta,x-y_0 \rangle -\frac{\partial \eta}{\partial\nu}\langle \nabla v,x-y_0 \rangle+\langle  \nabla v,\nabla\eta\rangle\langle \nu,x-y_0\rangle)ds
       \\&-\frac{N}{p+1}\int_{\partial\Omega_1}(\frac{\partial u}{\partial\nu}\xi+\frac{\partial\eta}{\partial\nu}v)ds-\frac{N}{q+1}\int_{\partial\Omega_1}(\frac{\partial v}{\partial\nu}\eta+\frac{\partial\xi}{\partial\nu}u)ds
       \\=&(-N+\frac{2N}{q+1})\epsilon\beta_2\int_{\Omega_1}u\eta dx+(-N+\frac{2N}{p+1})\epsilon\beta_1\int_{\Omega_1}v\xi dx-2\epsilon\alpha\int_{\Omega_1}(v\eta+u\xi)ds
       \\&+\epsilon\beta_2\int_{\partial\Omega_1}u\eta\langle\nu,x-y_0 \rangle ds+\epsilon\beta_1\int_{\partial\Omega_1}v\xi\langle\nu,x-y_0 \rangle ds+\epsilon\alpha\int_{\partial\Omega_1}(  u\xi+v\eta)\langle\nu,x-y_0 \rangle ds
       \\&+ \int_{\partial\Omega_1}(u^{q}\xi+v^{p}\eta)\langle \nu,x-y_0\rangle ds.
    \end{split}
\end{equation}
\end{lemma}

\begin{proof}First, we prove \eqref{2.1}. By direct calculations, we have
\begin{equation*}
    \begin{split}
       &\int_{\Omega_1} ( -\Delta u\frac{\partial\xi}{\partial x_i} -\Delta v\frac{\partial\eta}{\partial x_i} -\Delta \eta \frac{\partial v}{\partial x_i}  -\Delta \xi \frac{\partial u}{\partial x_i})dx
       \\=&\int_{\Omega_1}(v^{p}\frac{\partial\xi}{\partial x_i} + pv^{p-1}\frac{\partial v}{\partial x_i}+u^{q}\frac{\partial\eta}{\partial x_i}+qu^{q-1}\frac{\partial u}{\partial x_i})dx+\epsilon\beta_2\int_{\Omega_1}(u\frac{\partial\eta}{\partial x_i}+\eta\frac{\partial u}{\partial x_i})dx
        \\&+\epsilon\alpha\int_{\Omega_1}(v\frac{\partial\eta}{\partial x_i}+\eta\frac{\partial v}{\partial x_i})dx
        +\epsilon\beta_1\int_{\Omega_1}(v\frac{\partial\xi}{\partial x_i}+\xi \frac{\partial v}{\partial x_i})dx
        +\epsilon \alpha \int_{\Omega_1}(u\frac{\partial\xi}{\partial x_i}+\xi \frac{\partial u}{\partial x_i})dx.
    \end{split}
\end{equation*}
One can check that
\begin{equation*}
    \begin{split}
        &\int_{\Omega_1}(v^{p}\frac{\partial\xi}{\partial x_i} + pv^{p-1}\frac{\partial v}{\partial x_i}+u^{q}\frac{\partial\eta}{\partial x_i}+qu^{q-1}\frac{\partial u}{\partial x_i})dx+\epsilon\beta_2\int_{\Omega_1}(u\frac{\partial\eta}{\partial x_i}+\eta\frac{\partial u}{\partial x_i})dx
        \\&+\epsilon\alpha\int_{\Omega_1}(v\frac{\partial\eta}{\partial x_i}+\eta\frac{\partial v}{\partial x_i})dx
        +\epsilon\beta_1\int_{\Omega_1}(v\frac{\partial\xi}{\partial x_i}+\xi \frac{\partial v}{\partial x_i})dx
        +\epsilon \alpha \int_{\Omega_1}(u\frac{\partial\xi}{\partial x_i}+\xi \frac{\partial u}{\partial x_i})dx
        \\=&\int_{\partial\Omega_1}(u^{q}\eta+v^{p}\xi)\nu_i ds
        +\epsilon\alpha\int_{\partial\Omega_1}(u\eta\nu_i+v\xi\nu_i)ds
        +\epsilon\beta_1\int_{\partial\Omega_1}v\xi\nu_ids+\epsilon\beta_2\int_{\partial\Omega_1}u\eta\nu_ids,
    \end{split}
\end{equation*}
and
\begin{equation*}
    \begin{split}
    &\int_{\Omega_1} ( -\Delta u\frac{\partial\xi}{\partial x_i} -\Delta v\frac{\partial\eta}{\partial x_i} -\Delta \eta \frac{\partial v}{\partial x_i}  -\Delta \xi \frac{\partial u}{\partial x_i})dx
    \\=&\int_{\partial\Omega_1}(-\frac{\partial u}{\partial\nu}\frac{\partial\xi}{\partial x_i} -\frac{\partial \xi}{\partial\nu}\frac{\partial u}{\partial x_i} +\frac{\partial u}{\partial\nu}\frac{\partial\xi}{\partial\nu}\nu_i)ds+\int_{\partial\Omega_1}(-\frac{\partial v}{\partial\nu}\frac{\partial\eta}{\partial  x_i} -\frac{\partial \eta}{\partial\nu}\frac{\partial v}{\partial  x_i} +\frac{\partial v}{\partial\nu}\frac{\partial\eta}{\partial\nu}\nu_i)ds.
    \end{split}
\end{equation*}
Therefore, we have proved \eqref{2.1}.

Next, we prove  \eqref{2.2}. We have
\begin{equation*}
\begin{split}
   & \int_{\Omega_1}(-\Delta u\langle \nabla\xi,x-y_0 \rangle  -\Delta v\langle \nabla\eta,x-y_0 \rangle  -\Delta \eta\langle \nabla v,x-y_0 \rangle-\Delta \xi\langle \nabla u,x-y_0 \rangle  )dx
    \\=&\int_{\Omega_1}(v^{p}\langle \nabla\xi,x-y_0 \rangle+u^{q}\langle \nabla\eta,x-y_0 \rangle + pv^{p-1}\xi\langle \nabla v,x-y_0 \rangle+qu^{q-1}\eta \langle \nabla u,x-y_0 \rangle )dx
    \\&+\epsilon\beta_2\int_{\Omega_1}(u\langle \nabla\eta,x-y_0 \rangle+\eta\langle \nabla u,x-y_0 \rangle)dx+\epsilon\alpha\int_{\Omega_1}(v\langle \nabla\eta,x-y_0 \rangle+\eta\langle \nabla v,x-y_0 \rangle)dx
    \\&+ \epsilon\alpha\int_{\Omega_1}(u\langle \nabla\xi,x-y_0 \rangle+\xi\langle \nabla u,x-y_0 \rangle)dx+\epsilon\beta_{1}\int_{\Omega_1}(v\langle \nabla\xi,x-y_0 \rangle+\xi\langle \nabla v,x-y_0 \rangle)dx     .
    \end{split}
\end{equation*}
Moreover,
\begin{equation}\label{2.4}
    \begin{split}
      &\int_{\Omega_1}(v^{p}\langle \nabla\xi,x-y_0 \rangle+u^{q}\langle \nabla\eta,x-y_0 \rangle + pv^{p-1}\xi\langle \nabla v,x-y_0 \rangle+qu^{q-1}\eta \langle \nabla u,x-y_0 \rangle )dx
    \\&+\epsilon\beta_2\int_{\Omega_1}(u\langle \nabla\eta,x-y_0 \rangle+\eta\langle \nabla u,x-y_0 \rangle)dx+\epsilon\alpha\int_{\Omega_1}(v\langle \nabla\eta,x-y_0 \rangle+\eta\langle \nabla v,x-y_0 \rangle)dx
    \\&+ \epsilon\alpha\int_{\Omega_1}(u\langle \nabla\xi,x-y_0 \rangle+\xi\langle \nabla u,x-y_0 \rangle)dx+\epsilon\beta_{1}\int_{\Omega_1}(v\langle \nabla\xi,x-y_0 \rangle+\xi\langle \nabla v,x-y_0 \rangle)dx
      \\&+\epsilon\beta_2\int_{\Omega_1}(u\langle \eta,x-y_0 \rangle+\eta\langle u,x-y_0 \rangle)dx
      \\= &\int_{\partial\Omega_1}(u^{q}\eta+v^{p}\xi )\langle \nu,x-y_0\rangle ds
      -N\int_{\Omega_1}(u^{q}\eta+v^{p}\xi)dx
      +\epsilon\beta_2\int_{\partial\Omega_1}u\eta\langle\nu,x-y_0\rangle ds
      \\&-N\epsilon\beta_2\int_{\Omega_1}u\eta dx
      + \epsilon\alpha\int_{\partial\Omega_1}(v\eta+u\xi)\langle\nu,x-y_0\rangle ds-N\epsilon\alpha\int_{\Omega_1}(v\eta+u\xi) dx
      \\&+ \epsilon\beta_1\int_{\partial\Omega_1}v\xi\langle\nu,x-y_0\rangle ds-N\epsilon\beta_1\int_{\Omega_1}v\xi dx   ,
    \end{split}
\end{equation}
and
\begin{equation*}
    \begin{split}
    &\int_{\Omega_1}(-\Delta u\langle \nabla\xi,x-y_0 \rangle  -\Delta v\langle \nabla\eta,x-y_0 \rangle  -\Delta \eta\langle \nabla v,x-y_0 \rangle-\Delta \xi\langle \nabla u,x-y_0 \rangle  )dx
    \\=&\int_{\partial\Omega_1}(-\frac{\partial u}{\partial\nu}\langle \nabla\xi,x-y_0 \rangle -\frac{\partial \xi}{\partial\nu}\langle \nabla u,x-y_0 \rangle+\langle  \nabla u,\nabla\xi\rangle\langle \nu,x-y_0\rangle)ds
    \\&+\int_{\partial\Omega_1}(-\frac{\partial v}{\partial\nu}\langle \nabla\eta,x-y_0 \rangle -\frac{\partial \eta}{\partial\nu}\langle \nabla v,x-y_0 \rangle+\langle  \nabla v,\nabla\eta\rangle\langle \nu,x-y_0\rangle)ds
    \\&+(2-N)\int_{\Omega_1}(\langle \nabla u,\nabla\xi \rangle + \langle \nabla v,\nabla\eta \rangle)dx.
    \end{split}
\end{equation*}
We also have
\begin{equation*}
    \begin{split}
        &\int_{\Omega_1}(q+1)u^{q}\eta
        \\=&\int_{\Omega_1}(-\Delta v -\epsilon\beta_2u -\epsilon\alpha v)\eta+(-\Delta \xi - \epsilon\beta_2\eta -\epsilon\alpha\xi)u
        \\=&\int_{\partial\Omega_1}-\frac{\partial v}{\partial\nu}\eta-\frac{\partial\xi}{\partial\nu}ds+\int_{\Omega}\langle\nabla v,\nabla\eta\rangle+\langle \nabla\xi,\nabla u\rangle-2\epsilon\beta_2\eta u dx
        \\=&\int_{\partial\Omega_1}-\frac{\partial v}{\partial\nu}\eta-\frac{\partial\xi}{\partial\nu}+\frac{\partial u}{\partial\nu}+\frac{\partial\eta}{\partial\nu}v ds+\int_{\Omega}(-2\epsilon\beta_2\eta u-\epsilon\alpha(v\eta+u\xi) -\Delta u\xi-\Delta\eta v)dx
        \\=&\int_{\partial\Omega_1}-\frac{\partial v}{\partial\nu}\eta-\frac{\partial\xi}{\partial\nu}+\frac{\partial u}{\partial\nu}+\frac{\partial\eta}{\partial\nu}v ds+\int_{\Omega}\bigg((p+1)v^{p}\xi +2\epsilon\beta_{1}v\xi-2\epsilon\beta_2u\eta\bigg)dx,
    \end{split}
\end{equation*}
and
\begin{equation}\label{2.7}
\begin{split}
    &(2-N)\int_{\Omega_1}(\langle \nabla u,\nabla\xi \rangle + \langle \nabla v,\nabla\eta \rangle)
    \\=&(2-N)\bigg(\int_{\partial\Omega_1}\frac{\partial u}{\partial \nu}\xi+\frac{\partial\eta}{\partial\nu}v + \int_{\Omega_1}(p+1)v^{p}\xi
    +\int_{\Omega_1}2\epsilon\beta_{1}v\xi+\epsilon\alpha(v\eta+u\xi)    \bigg).
\end{split}
\end{equation}
Collecting \eqref{2.4}-\eqref{2.7}, together with $\frac{1}{p+1} + \frac{1}{q+1} = \frac{N-2}{N}$, we obtain \eqref{2.2}.
\end{proof}

\section{ Revisit the existence problem}
In this section, we will revisit the existence problem for \eqref{inteq1} and give a different
proof of Theorem {\bf{A}}. In \cite{Kim-Pis}, they proves the existence of $(u_{\epsilon},v_{\epsilon})$ in the case of norm $||u_{\epsilon},v_{\epsilon}||_{ X_{p,q}}$. Different from \cite{Kim-Pis}, in this paper, we need uniform point-by-point estimation in the process of proving non-degenerate by using the Pohozaev identity, so we need the existence of $(u_{\epsilon},v_{\epsilon})$ in the case of norm $||u_{\epsilon},v_{\epsilon}||_{*,\epsilon}$.

We first introduce the following norms:
\begin{align*}
    ||(f,g)||_{*,\epsilon} = &\sup_{x\in \Omega}\frac{|f(x)|}{\left(\frac{\mu_{1,\epsilon}^{-\frac{N}{q+1}}}{(1+\mu_{1,\epsilon}^{-1}|x-P_{1,\epsilon}|)^{\frac{N(p+1)}{2(q+1)}+\theta}}\right)}+\sup_{x\in \Omega}\frac{|g(x)|}{\left(\frac{\mu_{1,\epsilon}^{-\frac{N}{p+1}}}{(1+\mu_{1,\epsilon}^{-1}|x-P_{1,\epsilon}|)^{\frac{N-2}{2}+1+\theta}}\right)},
\end{align*}
and
\begin{align*}
    ||(f,g)||_{**,\epsilon} = &\sup_{x\in \Omega}\frac{|f(x)|}{\left(\frac{\mu_{1,\epsilon}^{-\frac{N}{q+1}-2}}{(1+\mu_{1,\epsilon}^{-1}|x-P_{1,\epsilon}|)^{\frac{N(p+1)}{2(q+1)}+2+\theta}}\right)}
    +\sup_{x\in \Omega}\frac{|g(x)|}{\left(\frac{\mu_{1,\epsilon}^{-\frac{N}{p+1}-2}}{(1+\mu_{1,\epsilon}^{-1}|x-P_{1,\epsilon}|)^{\frac{N+2}{2}+1+\theta}}\right)},
\end{align*}
where $\theta$ is a small constant.

 We set
 $
 (  U_{1,\epsilon},V_{1,\epsilon})=(U_{\mu_{1,\epsilon},P_{1,\epsilon}}(x), V_{\mu_{1,\epsilon},P_{1,\epsilon}}(x)),
 $
 \begin{align*}
    (\Psi_{0},\Phi_{0}) = \bigg(x\cdot\nabla U_{1,0}+\frac{N U_{1,0}}{q+1},x\cdot\nabla V_{1,0}+\frac{N V_{1,0}}{p+1}    \bigg),
\end{align*}
and
\begin{align*}
    (\Psi_{l},\Phi_{l}) = ( \frac{\partial U_{1,0}}{\partial x_{l}},\frac{\partial V_{1,0}}{\partial x_{l}} ),\,\,\,\hbox{ for }l=1,\cdots,N.
\end{align*}
Also let
 $(\Psi_{1,\epsilon,i},\Phi_{1,\epsilon,i}) = (\mu_{1,\epsilon}^{-\frac{N}{q+1}-1}\Psi_{i}(\mu_{1,\epsilon}^{-1}(x-P_{1,\epsilon}))  , \mu_{1,\epsilon}^{-\frac{N}{p+1}-1}\Phi_{i}(\mu_{1,\epsilon}^{-1}(x-P_{1,\epsilon}))      ) ,$
and  $( P\Psi_{1,\epsilon,i},P\Phi_{1,\epsilon,i}  ) $ be the unique smooth solution of the system
\begin{equation*}
    \begin{cases}
    -\Delta P\Psi_{1,\epsilon,i} =p V_{1,\epsilon}^{p-1}P\Phi_{1,\epsilon,i},  &\hbox{ in }\Omega,
    \\ -\Delta P\Phi_{1,\epsilon,i} =p U_{1,\epsilon}^{q-1}P\Psi_{1,\epsilon,i},  &\hbox{ in }\Omega,
    \\P\Psi_{1,\epsilon,i}=P\Phi_{1,\epsilon,i}=0, &\hbox{ in }\partial\Omega,
    \end{cases}
\end{equation*}
for $i=1,\cdots,N $.

We consider the following problem:
\begin{equation}\label{3.1}
    \begin{cases}
    \displaystyle -\Delta \tilde{\psi}_{\epsilon} = h_{\epsilon}
    + p V_{1,\epsilon}^{p-1}\tilde{\phi}_{\epsilon} +\sum_{i=0}^{N}c_{i,\epsilon}p V_{1,\epsilon}^{p-1}\Phi_{1,\epsilon,i}, &\hbox{ in }\Omega,
    \\ \displaystyle -\Delta\tilde{\phi}_{\epsilon}=g_{\epsilon} +q U_{1,\epsilon}^{q-1}\tilde{\psi}_{\epsilon} +{\epsilon}\beta_2\psi_{\epsilon}+\sum_{i=0}^{N}c_{i,\epsilon}q U_{1,\epsilon}^{q-1}\Psi_{1,\epsilon,i}, &\hbox{ in }\Omega,
    \\ \displaystyle \tilde{\psi}_{\epsilon}=\tilde{\phi}_{\epsilon} =0, &\hbox{ on }\partial\Omega,
    \\ \displaystyle\int_{\Omega}p V_{1,\epsilon}^{p-1}\Phi_{1,\epsilon,i}\tilde{\phi}_{\epsilon}+q U_{1,\epsilon}^{q-1}\Psi_{1,\epsilon,i}\tilde{\psi}_{\epsilon} dx =0.
    \end{cases}
\end{equation}

\begin{lemma}\label{le3.1}
Assume that $(c_{i,\epsilon},\tilde{\psi}_{\epsilon},\tilde{\phi}_{\epsilon})$ solves \eqref{3.1} for $ (h_{\epsilon},g_{\epsilon})$. If $||(h_{\epsilon},g_{\epsilon})||_{**,\epsilon}\to 0$, as $\epsilon\to 0$, then $||(\tilde{\psi}_{\epsilon},\tilde{\phi}_{\epsilon})||_{*,\epsilon}\to 0$,  as $\epsilon\to 0$.
\end{lemma}
\begin{proof}
We argue by contradiction. Suppose that there exist $(c_{i,\epsilon},\tilde{\psi}_{\epsilon},\tilde{\phi}_{\epsilon})$  solves system \eqref{3.1} for $ (h_{\epsilon},g_{\epsilon})$, such that $||(h_{\epsilon},g_{\epsilon})||_{**,\epsilon} \to 0$ and $||(\tilde{\psi}_{\epsilon},\tilde{\phi}_{\epsilon})||_{*,\epsilon} =1 $. We have
\begin{equation*}
    \begin{split}
        \tilde{\psi}_{\epsilon}(x) = \int_{\Omega}G(x,y)(h_{\epsilon}+ p V_{1,\epsilon}^{p-1}\tilde{\phi}_{\epsilon} +\sum_{i=0}^{N}c_{i,\epsilon}p V_{1,\epsilon}^{p-1}\Phi_{1,\epsilon,i})dy,
    \end{split}
\end{equation*}
where $G(x,y)$ be the Green's function of the Laplacian $-\Delta$ in $\Omega$ with Dirichlet
boundary condition.

It is easy to check that
\begin{equation*}
    \begin{split}
        | \int_{\Omega}G(x,y) h_{\epsilon} | \leq C||(h_{\epsilon},g_{\epsilon})||_{**,\epsilon}\frac{\mu_{1,\epsilon}^{-\frac{N}{q+1}}}{(1+\mu_{1,\epsilon}^{-1}|x-P_{1,\epsilon}|)^{\frac{N(p+1)}{2(q+1)}+\theta}}.
    \end{split}
\end{equation*}
Since $\frac{N-2}{2} + (N-2)(p-1) +1-2- \frac{N(p+1)}{2(q+1)}= \frac{(N-2)(p-1)}{2}>0 $, we have
\begin{equation*}
    \begin{split}
        |\int_{\Omega}G(x,y) p V_{1,\epsilon}^{p-1}\tilde{\phi}_{\epsilon}| \leq C\frac{\mu_{1,\epsilon}^{-\frac{N}{q+1}} }{ (1+\mu_{1,\epsilon}^{-1}|x-P_{1,\epsilon}|)^{\frac{N(p+1)}{2(q+1)}+\theta+ \theta_1}    },
    \end{split}
\end{equation*}
where $\theta_1 > 0$ is a small constant. Similarly to the proof of Lemma 4.4 in \cite{Kim-Pis}, we can proof $c_{i,\epsilon} = o(\mu_{\epsilon})$. Due to $  (N-2)p-2 -\frac{N(p+1)}{2(q+1)}-\theta =\frac{N(p+1)}{2(q+1)}-\theta>0$, we also have
\begin{equation*}
\left|\int_{\Omega}G(x,y) \sum_{i=0}^{N}c_{i,\epsilon}p V_{1,\epsilon}^{p-1}\Phi_{1,\epsilon,i}\right|
        \leq C o(1)  \frac{\mu_{1,\epsilon}^{-\frac{N}{q+1}} }{ (1+\mu_{1,\epsilon}^{-1}|x-P_{1,\epsilon}|)^{\frac{N(p+1)}{2(q+1)}+\theta}}.
\end{equation*}
Because of $ \frac{N(p+1)q}{q+1} -2 > \frac{N}{2}+\theta$, we get
\begin{equation*}
    \begin{split}
        |\int_{\Omega}G(x,y) \sum_{i=0}^{N}c_{i,\epsilon}p U_{1,\epsilon}^{q-1}\Psi_{1,\epsilon,i}|= o(1)  \frac{\mu_{1,\epsilon}^{-\frac{N}{q+1}} }{ (1+\mu_{1,\epsilon}^{-1}|x-P_{1,\epsilon}|)^{\frac{N-2}{2}+1+\theta}    }.
    \end{split}
\end{equation*}

Analogous, we  have
\begin{equation*}
    \begin{split}
        \tilde{\phi}_{\epsilon}(x) = \int_{\Omega}G(x,y)(g_{\epsilon} +q U_{1,\epsilon}^{q-1}\tilde{\psi}_{\epsilon} +{\epsilon}\beta_2\psi_{\epsilon}+\sum_{i=0}^{N}c_{i,\epsilon}q U_{1,\epsilon}^{q-1}\Psi_{1,\epsilon,i})dy.
    \end{split}
\end{equation*}
 By direct calculations,
\begin{equation*}
    \begin{split}
         |\int_{\Omega}G(x,y)g_{\epsilon}|\leq C||(h_{\epsilon},g_{\epsilon})||_{**,\epsilon}\frac{\mu_{1,\epsilon}^{-\frac{N}{p+1}}}{(1+\mu_{1,\epsilon}^{-1}|x-P_{1,\epsilon}|)^{\frac{N-2}{2}+1+\theta}},
    \end{split}
\end{equation*}
and
\begin{equation*}
    \begin{split}
        |\int_{\Omega}G(x,y) q U_{1,\epsilon}^{q-1}\tilde{\psi}_{\epsilon}  | \leq C\frac{\mu_{1,\epsilon}^{-\frac{N}{p+1}} }{ (1+\mu_{1,\epsilon}^{-1}|x-P_{1,\epsilon}|)^{\frac{N-2}{2}+1+\theta+ \theta_1}    },
    \end{split}
\end{equation*}
where we have used the fact that $ \frac{N(p+1)}{q+1}(q-\frac{1}{2}) -2 > \frac{N-2}{2}+1$.

Since $\mu_{1,\epsilon} \leq \frac{C}{(1+\mu_{1,\epsilon}^{-1}|x-P_{1,\epsilon}|)}$, it holds
\begin{equation*}
    \begin{split}
         |\int_{\Omega}G(x,y) {\epsilon}\beta_2\psi_{\epsilon}| \leq &C\frac{\mu_{1,\epsilon}^{2+\frac{N(p-q+1)}{q+1}}}{(1+\mu_{1,\epsilon}^{-1}|x-P_{1,\epsilon}|)^{\frac{N(p+1)}{2(q+1)}+\theta-2}}
         \\\leq &o(1) \frac{\mu_{1,\epsilon}^{-\frac{N}{p+1}} }{ (1+\mu_{1,\epsilon}^{-1}|x-P_{1,\epsilon}|)^{\frac{N-2}{2}+1+\theta}    }.
    \end{split}
\end{equation*}

By testing with $  (P\Phi_{1,{\epsilon},j},P\Psi_{1,\epsilon,j})$, we obtain
\begin{equation*}
    \begin{split}
        &\int_{\Omega}\bigg( p(V_{1,\epsilon}^{p-1}-     PV_{1,\epsilon}^{p-1}) P\Phi_{1,{\epsilon},j}\tilde{\phi}_{\epsilon} + q(  U_{1,\epsilon}^{q-1} -PU_{d_{1,\epsilon},P_{1,\epsilon}}^{q-1})P\Psi_{1,\epsilon,j}\tilde{\psi}_{\epsilon}\bigg)
        -\int_{\Omega}{\epsilon}\beta_2\psi_{\epsilon}P\Psi_{1,{\epsilon},j}\tilde{\psi}_{\epsilon}
        \\=&\int_{\Omega}h_{\epsilon} P\Phi_{1,{\epsilon},j} + g_{\epsilon} P\Psi_{1,{\epsilon},j}
        +\int_{\Omega}\sum_{i=0}^{N}c_{i,\epsilon}(  pV_{1,\epsilon}^{p-1}\Phi_{1,\epsilon,i} P\Phi_{1,{\epsilon},j} +q U_{1,\epsilon}^{q-1}\Psi_{1,\epsilon,i}P\Phi_{1,{\epsilon},j}     ).
    \end{split}
\end{equation*}
We shall estimate the LHS of the above equations. They are divided into three parts.
\begin{equation}\label{3.11}
    \begin{split}
        |\int_{\Omega}(V_{1,\epsilon}^{p-1}-     PV_{1,\epsilon}^{p-1})P\Phi_{1,{\epsilon},j}\tilde{\phi}_{\epsilon}|
         \leq &C\int_{\Omega}\mu_{1,\epsilon}^{\frac{N}{q+1}}\frac{\mu_{1,\epsilon}^{-1  -\frac{Np}{p+1}}   }{   (1  +\mu_{1,\epsilon}^{-1}|  x-P_{1,\epsilon} | )^{(N-2)(p-1)+\frac{N}{2}+\theta }  }
        \\ \leq &C\mu_{1,\epsilon}^{-3+ (N-2)(p-1)+\frac{N}{2}+\theta} = o( \mu_{1,\epsilon}^{-1} ),
    \end{split}
\end{equation}
\begin{equation*}
    \begin{split}
        |\int_{\Omega}(U_{1,\epsilon}^{q-1}-     PU_{d_{1,\epsilon},P_{1,\epsilon}}^{q-1})P\Psi_{1,{\epsilon},j}\tilde{\psi}_{\epsilon}|
        \leq &C\int_{\Omega}\mu_{1,\epsilon}^{\frac{Np}{q+1}} \frac{\mu_{1,\epsilon}^{-1  -\frac{Nq}{q+1}}   }{   (1  +\mu_{1,\epsilon}^{-1}|  x-P_{1,\epsilon} | )^{\frac{N(p+1)}{q+1}(q-\frac{1}{2})+\theta }  }
        \\ \leq &C\mu_{1,\epsilon}^{\frac{N(p+1)}{q+1}(q+\frac{1}{2}) -N  +\theta }=o( \mu_{1,\epsilon}^{-1} ),
    \end{split}
\end{equation*}
and
\begin{equation*}
    \begin{split}
       | \int_{\Omega}{\epsilon}\beta_2\psi_{\epsilon}P\Psi_{1,{\epsilon},j}\tilde{\psi}_{\epsilon}|
      \leq &C\mu_{1,\epsilon}^{\frac{N(p-q+2)}{q+1}}\int_{\Omega}\frac{\mu_{1,\epsilon}^{-1  -\frac{2N}{q+1}}   }{   (1  +\mu_{1,\epsilon}^{-1}|  x-P_{1,\epsilon} | )^{\frac{3N(p+1)}{2(q+1)}+\theta }  }
       \\ \leq &C\mu_{1,\epsilon}^{   -1 +\frac{N(p+1)}{q+1}} = o( \mu_{1,\epsilon}^{-1} ).
    \end{split}
\end{equation*}
Next, we shall estimate the RHS of the equations \eqref{3.11}. We get that
\begin{equation*}
    \begin{split}
       |\int_{\Omega}h_{\epsilon} P\Phi_{1,{\epsilon},j} |
     \leq C ||(h_{\epsilon},g_{\epsilon})||_{**,\epsilon}  \mu_{1,\epsilon}^{-1}\int_{\Omega} \frac{\mu_{1,\epsilon}^{N}}{  (1  +\mu_{1,\epsilon}^{-1}|  x-P_{1,\epsilon} | )^{  N+\frac{N(p+1)}{2(q+1)}+\theta} }
       =o( \mu_{1,\epsilon}^{-1} ),
    \end{split}
\end{equation*}
and
\begin{equation*}
    \begin{split}
        |\int_{\Omega}g_{\epsilon} P\Phi_{1,{\epsilon},j} | = o( \mu_{1,\epsilon}^{-1} ).
    \end{split}
\end{equation*}
From Lemma 3.1 of \cite{Kim-Pis}, we have
\begin{equation*}
    \begin{split}
        \int_{\Omega}(  pV_{1,\epsilon}^{p-1}\Phi_{1,\epsilon,i} P\Phi_{1,{\epsilon},j} +q U_{1,\epsilon}^{q-1}\Psi_{1,\epsilon,i}P\Phi_{1,{\epsilon},j}     ) = \mu_{1,\epsilon}^{-2}C_{i}\delta_{ij}+ o(\mu_{1,\epsilon}^{-2}).
    \end{split}
\end{equation*}
So we have $c_{i,\epsilon} = o( \mu_{1,\epsilon} )$ for $i=0,\cdots,N.$

Hence,  we have
\begin{equation}\label{3.3}
    \begin{split}
       & |\tilde{\psi}_{\epsilon}(x)|\left( \frac{\mu_{1,\epsilon}^{-\frac{N}{q+1}}}{(1+\mu_{1,\epsilon}^{-1}|x-P_{1,\epsilon}|)^{\frac{N(p+1)}{2(q+1)}+\theta}} \right)^{-1} + |\tilde{\phi}_{\epsilon}(x)|\left(\frac{\mu_{1,\epsilon}^{-\frac{N}{p+1}}}{(1+\mu_{1,\epsilon}^{-1}|x-P_{1,\epsilon}|)^{\frac{N-2}{2}+1+\theta}}\right)^{-1}
        \\\leq &C\left[ o(1) + \frac{(1+\mu_{1,\epsilon}^{-1}|x-P_{1,\epsilon}|)^{\frac{N(p+1)}{2(q+1)}+\theta}}{(1+\mu_{1,\epsilon}^{-1}|x-P_{1,\epsilon}|)^{\frac{N(p+1)}{2(q+1)}+\theta+\theta_1}} + \frac{(1+\mu_{1,\epsilon}^{-1}|x-P_{1,\epsilon}|)^{\frac{N-2}{2}+1+\theta}}{(1+\mu_{1,\epsilon}^{-1}|x-P_{1,\epsilon}|)^{\frac{N-2}{2}+1+\theta+\theta_1}}\right].
    \end{split}
\end{equation}
Since $||(\tilde{\psi}_{\epsilon},\tilde{\phi}_{\epsilon} )||_{*,\epsilon}=1$, we obtain from \eqref{3.3} that there is $R>0$, such that
\begin{equation}\label{3.4}
    \begin{split}
        \sup_{x\in B_{\mu_{1,\epsilon} R}(P_{1,\epsilon})}|\tilde{\phi}_{\epsilon}(x)|+|\tilde{\psi}_{\epsilon}(x)| \geq a >0.
    \end{split}
\end{equation}
Let $(\bar{\tilde{\psi}}_{\epsilon},\bar{\tilde{\phi}}_{\epsilon}) = (\mu_{1,\epsilon}^{\frac{N}{q+1}}\tilde{\psi}_{\epsilon}(\mu_{1,\epsilon} x + P_{1,\epsilon} )  , \mu_{1,\epsilon}^{\frac{N}{p+1}}\tilde{\phi}_{\epsilon}(\mu_{1,\epsilon} x + P_{1,\epsilon}) )$, then $(\bar{\tilde{\psi}}_{\epsilon},\bar{\tilde{\phi}}_{\epsilon}) \to (\tilde{\psi},\tilde{\phi}) $ in $C^{1}_{loc}(\mathbb{R}^{N})$ and
\begin{equation}\label{3.2}
    \begin{cases}
    -\Delta \tilde{\psi} = p V_{1,0}^{p-1}\tilde{\phi}, \hbox{ in } \mathbb{R}^{N},
    \\-\Delta \tilde{\phi} = q U_{1,0}^{q-1}\tilde{\psi},  \hbox{ in } \mathbb{R}^{N}.
    \end{cases}
\end{equation}
Since $[(N-2)(p-1)+\frac{N}{2}+\theta]\frac{p+1}{p}>N$ and $ (\frac{N(p+1)(q-\frac{1}{2})}{q+1} +\theta)\frac{q+1}{q}>N$ for $ p\in (1,\frac{N-1}{N-2})$, we have
\begin{equation*}
    \begin{split}
        \int_{\mathbb{R}^{N}}  |V_{1,0}^{p-1}\tilde{\phi}|^{\frac{p+1}{p}} \leq C, \hbox{ } \int_{\mathbb{R}^{N}}  |U_{1,0}^{p-1}\tilde{\psi}|^{\frac{q+1}{q}} \leq C.    \end{split}
\end{equation*}
So $ (\tilde{\psi},\tilde{\phi})  \in \dot{W}^{2,\frac{p+1}{p}}\times \dot{W}^{2,\frac{q+1}{q}}$.
Recalling \eqref{3.1}, we can check that $( \tilde{\psi},\tilde{\phi})$ is perpendicular to the kernel of \eqref{3.2}. So $( \tilde{\psi},\tilde{\phi}) = (0,0)$.  This is a contradiction to \eqref{3.4}.
\end{proof}

 We denote that $\mathcal{L}_{\epsilon}( g_{\epsilon},h_{\epsilon})$ is the solution of problem \eqref{3.1}. With the help of Lemma \ref{le3.1}, similar to Proposition 4.1 in \cite{PFM}, we have the following proposition.

\begin{proposition}
There exist $\epsilon_1 > 0$ and a constant $C > 0$, independent of $\epsilon$, such that for all $\epsilon \in(0,\epsilon_1)$ and all $(g,h) \in (L^{\infty}( \Omega))^2$, problem \eqref{3.1} has a unique solution $( \tilde{\psi},\tilde{\phi} ) = \mathcal{L}_{\epsilon}(g,h )$ with $\|(\tilde{\psi},\tilde{\phi})\|_{*,\epsilon} \leq C\|(g,h)\|_{**,\epsilon}, \; |c_{i,\epsilon}| \leq C\mu\|(g,h)\|_{**,\epsilon}$.
\end{proposition}

In the following, we will solve the following problem:
\begin{equation}\label{pro3.3}
\begin{cases}
\displaystyle-\Delta (PU_{d_{1,\epsilon},P_{1,\epsilon}} + \tilde{\psi}_{\epsilon} ) = ( PV_{1,\epsilon} +\tilde{\phi}_{\epsilon} )^{p} + \sum_{i=0}^{N}c_{i,\epsilon}p V_{1,\epsilon}^{p-1}\Phi_{1,\epsilon,i}, &\hbox{ in }\Omega,
\\ \displaystyle- \Delta ( PV_{1,\epsilon} +\tilde{\phi}_{\epsilon} ) = (PU_{d_{1,\epsilon},P_{1,\epsilon}} + \tilde{\psi}_{\epsilon} )^{q} +{\epsilon}\beta_2\psi_{\epsilon}+\sum_{i=0}^{N}c_{i,\epsilon}q U_{1,\epsilon}^{q-1}\Psi_{1,\epsilon,i}, &\hbox{ in }\Omega,
\\  \tilde{\psi}_{\epsilon}=\tilde{\phi}_{\epsilon} =0, &\hbox{ on }\partial\Omega,
\\ \displaystyle\int_{\Omega}p V_{1,\epsilon}^{p-1}\Phi_{1,\epsilon,i}\tilde{\phi}_{\epsilon}+q U_{1,\epsilon}^{q-1}\Psi_{1,\epsilon,i}\tilde{\psi}_{\epsilon} dx =0.
\end{cases}
\end{equation}
That is
\begin{equation*}
    \begin{cases}
    \displaystyle -\Delta \tilde{\psi}_{\epsilon} = (PV_{1,\epsilon}+\tilde{\phi}_{\epsilon})^p -PV^{p}_{1,\epsilon} - p PV^{p-1}_{1,\epsilon}\tilde{\phi}_{\epsilon}+ p V_{1,\epsilon}^{p-1}\tilde{\phi}_{\epsilon} +\sum_{i=0}^{N}c_{i,\epsilon}p V_{1,\epsilon}^{p-1}\Phi_{1,\epsilon,i}, &\hbox{ in }\Omega,
    \\ \displaystyle -\Delta\phi_{\epsilon}=(PU_{d_{1,\epsilon},P_{1,\epsilon}} + \tilde{\psi}_{\epsilon})^q - PU^q_{d_{1,\epsilon},P_{1,\epsilon}}  - q PU^{q-1}_{d_{1,\epsilon},P_{1,\epsilon}}\tilde{\psi}_{\epsilon}+U^{q}_{1,\epsilon}
    \\- PU_{d_{1,\epsilon},P_{1,\epsilon}} + \epsilon\beta_2PU_{d_{1,\epsilon},P_{1,\epsilon}}
    +q U_{1,\epsilon}^{q-1}\tilde{\psi}_{\epsilon} +{\epsilon}\beta_2\psi_{\epsilon}+\sum_{i=0}^{N}c_{i,\epsilon}q U_{1,\epsilon}^{q-1}\Psi_{1,\epsilon,i}, &\hbox{ in }\Omega,
    \\  \tilde{\psi}_{\epsilon}=\tilde{\phi}_{\epsilon} =0, &\hbox{ on }\partial\Omega,
    \\ \displaystyle\int_{\Omega}p V_{1,\epsilon}^{p-1}\Phi_{1,\epsilon,i}\tilde{\phi}_{\epsilon}+q U_{1,\epsilon}^{q-1}\Psi_{1,\epsilon,i}\tilde{\psi}_{\epsilon} dx =0.
    \end{cases}
\end{equation*}

Now we set
$$l_{\epsilon} = (0, U^{q}_{1,\epsilon}- PU_{d_{1,\epsilon},P_{1,\epsilon}}^{q} - \epsilon\beta_2PU_{d_{1,\epsilon},P_{1,\epsilon}} ),$$
\begin{align*}
&N(\tilde{\psi}_{\epsilon},\tilde{\phi}_{\epsilon} )= \left( (PV_{1,\epsilon}+\tilde{\phi}_{\epsilon})^p -PV^{p}_{1,\epsilon} - p PV^{p-1}_{1,\epsilon}\tilde{\phi}_{\epsilon}, (PU_{d_{1,\epsilon},P_{1,\epsilon}} + \tilde{\psi}_{\epsilon})^q - PU^q_{d_{1,\epsilon},P_{1,\epsilon}}  - q PU^{q-1}_{d_{1,\epsilon},P_{1,\epsilon}}\tilde{\psi}_{\epsilon}\right),
\end{align*}
and $A = \frac{N(p+1)q}{q+1} -\frac{N+2}{2}-1-\theta$.
\begin{lemma}
There exists a positive constant $C$ such that
$$||N(\tilde{\psi}_{\epsilon},\tilde{\phi}_{\epsilon} )||_{**,\epsilon} \leq C ||(\tilde{\psi}_{\epsilon},\tilde{\phi}_{\epsilon} )||_{*,\epsilon}^{\min(p,q)}.$$
\end{lemma}
\begin{proof}
Using the fact that
$$
    |(PV_{1,\epsilon}+\tilde{\phi}_{\epsilon})^p -PV^{p}_{1,\epsilon} - p PV^{p-1}_{1,\epsilon}\tilde{\phi}_{\epsilon}| \leq C |\tilde{\phi}_{\epsilon}|^p,
$$
and
$$
    (\frac{N-2}{2}+1+\theta)p - \frac{N(p+1)}{2(q+1)}-2-\theta = (1+\theta)(p-1)>0,
$$
we have
\begin{equation*}
    \begin{split}
     &|(PV_{1,\epsilon}+\tilde{\phi}_{\epsilon})^p -PV^{p}_{1,\epsilon} - p PV^{p-1}_{1,\epsilon}\tilde{\phi}_{\epsilon}|\left(\frac{\mu_{1,\epsilon}^{-\frac{N}{q+1}-2}}{(1+\mu_{1,\epsilon}^{-1}|x-P_{1,\epsilon}|)^{\frac{N(p+1)}{2(q+1)}+2+\theta}}\right)^{-1}
     \\\leq& C||(\tilde{\psi}_{\epsilon},\tilde{\phi}_{\epsilon} )||_{*,\epsilon}^{p}\frac{(1+\mu_{1,\epsilon}^{-1}|x-P_{1,\epsilon}|)^{\frac{N(p+1)}{2(q+1)}+2+\theta} }{(1+\mu_{1,\epsilon}^{-1}|x-P_{1,\epsilon}|)^{(\frac{N-2}{2}+1+\theta)p}}
    \leq C||(\tilde{\psi}_{\epsilon},\tilde{\phi}_{\epsilon} )||_{*,\epsilon}^{p}.
    \end{split}
\end{equation*}
Similarly, due to
$
    |(PU_{d_{1,\epsilon},P_{1,\epsilon}} + \tilde{\psi}_{\epsilon})^q - PU^q_{d_{1,\epsilon},P_{1,\epsilon}}  - q PU^{q-1}_{d_{1,\epsilon},P_{1,\epsilon}}\tilde{\psi}_{\epsilon}| \leq C |\tilde{\psi}_{\epsilon}|^{q}
$
and $\frac{N(p+1)q}{2(q+1)} +q\theta > \frac{N+4}{2}+\theta$ ($p\in(1,\frac{N-1}{N-2}$)),
we deduce
\begin{equation*}
    \begin{split}
      & |(PU_{d_{1,\epsilon},P_{1,\epsilon}} + \tilde{\psi}_{\epsilon})^q - PU^q_{d_{1,\epsilon},P_{1,\epsilon}}  - q PU^{q-1}_{d_{1,\epsilon},P_{1,\epsilon}}\tilde{\psi}_{\epsilon}| (\frac{\mu_{1,\epsilon}^{-\frac{N}{p+1}-2}}{(1+\mu_{1,\epsilon}^{-1}|x-P_{1,\epsilon}|)^{\frac{N-2}{2}+3+\theta}})^{-1}
      \\\leq &C||(\tilde{\psi}_{\epsilon},\tilde{\phi}_{\epsilon} )||_{*,\epsilon}^{q}\frac{(1+\mu_{1,\epsilon}^{-1}|x-P_{1,\epsilon}|)^{\frac{N-2}{2}+3+\theta}}{(1+\mu_{1,\epsilon}^{-1}|x-P_{1,\epsilon}|)^{\frac{N(p+1)q}{2(q+1)} +q\theta}}
      \leq C||(\tilde{\psi}_{\epsilon},\tilde{\phi}_{\epsilon} )||_{*,\epsilon}^{q}.
    \end{split}
\end{equation*}
So
$
    ||N(\tilde{\psi}_{\epsilon},\tilde{\phi}_{\epsilon} )||_{**,\epsilon} \leq C ||(\tilde{\psi}_{\epsilon},\tilde{\phi}_{\epsilon} )||_{*,\epsilon}^{\min(p,q)}.
$
\end{proof}

\begin{lemma}
\begin{equation*}
    \begin{split}
        ||l_{\epsilon}||_{**,\epsilon}\leq C\mu^{\frac{N(p+1)q}{q+1}-\frac{N+2}{2}-1-\theta}\frac{\mu_{1,\epsilon}^{-\frac{N}{p+1}-2}}{( 1+\mu_{1,\epsilon}^{-1}|x-P_{1,\epsilon}|  )^{ \frac{N+2}{2}+1+\theta  }}.
    \end{split}
\end{equation*}
\end{lemma}
\begin{proof}
Since $\mu_{1,\epsilon} \leq \frac{C}{(1+\mu_{1,\epsilon}^{-1}|x-P_{1,\epsilon}|)}$ in $ \Omega$, we have
\begin{equation*}
    \begin{split}
        |U^{q}_{1,\epsilon}- PU^{q}_{d_{1,\epsilon},P_{1,\epsilon}} | \leq &C (U^{q-1}_{1,\epsilon}\mu_{1,\epsilon}^{\frac{Np}{q+1}} + \mu_{1,\epsilon}^{\frac{Npq}{q+1}})
        \\\leq& C\mu_{1,\epsilon}^{\frac{N(p+1)q}{q+1}-\frac{N+2}{2}-1-\theta}\frac{\mu_{1,\epsilon}^{-\frac{N}{p+1}-2}}{( 1+\mu_{1,\epsilon}^{-1}|x-P_{1,\epsilon}|  )^{ \frac{N+2}{2}+1+\theta  }},
    \end{split}
\end{equation*}
and
\begin{equation*}
    \begin{split}
        |\epsilon\beta_2 PU_{d_{1,\epsilon},P_{1,\epsilon}}   |\leq &C\mu^{\frac{N(p-q+2)}{q+1}}\frac{\mu_{1,\epsilon}^{-\frac{N}{q+1}}}{ ( 1+\mu_{1,\epsilon}^{-1}|x-P_{1,\epsilon}|  )^{  (N-2)p-2 }  }
        \\\leq &C\mu^{ \frac{N(p-q+2)}{q+1}-\frac{N}{q+1}+\frac{N}{p+1}+2  }\frac{\mu_{1,\epsilon}^{-\frac{N}{p+1}-2}}{ ( 1+\mu_{1,\epsilon}^{-1}|x-P_{1,\epsilon}|  )^{  \frac{N+2}{2}+1+\theta }  }
        \\\leq&C\mu^{\frac{N(p+1)q}{q+1}-\frac{N+2}{2}-1-\theta}\frac{\mu_{1,\epsilon}^{-\frac{N}{p+1}-2}}{( 1+\mu_{1,\epsilon}^{-1}|x-P_{1,\epsilon}|  )^{ \frac{N+2}{2}+1+\theta  }}.
    \end{split}
\end{equation*}

\end{proof}

Now, by standard discussion, we can deduce that there exists unique solution $( PU_{d_{1,\epsilon},P_{1,\epsilon}}+\tilde{\psi}_{\epsilon}, PV_{1,\epsilon}+\tilde{\phi}_{\epsilon} )$ solving  probelm \eqref{pro3.3} and $||( \tilde{\psi}_{\epsilon},\tilde{\phi_{\epsilon}})||_{**,\epsilon}\leq C \mu_{1,\epsilon}^{\frac{N(p+1)q}{q+1}-\frac{N+2}{2}-1-\theta}  $. 

Define
\begin{equation*}
    \begin{split}
        ||(f,g)||_{Y} = ||f||_{L^{\frac{p+1}{p}}(\Omega)} + ||g||_{L^{\frac{q+1}{q}}(\Omega)}.
    \end{split}
\end{equation*}

It follows from  Lemma 3.3 of \cite{Kim-Pis}, we have
\begin{lemma}\label{ele1}
It holds
$
        ||l_{\epsilon}||_{Y} \leq C_{1} \bigg(\mu^{\frac{Npq}{q+1}} + \mu^{\frac{N(p+1)}{q+1}} + {\epsilon}(\mu^{\frac{Np}{q+1}} + \mu^{\frac{N(q-1)}{q+1}}   )\beta_{2} \bigg).
   $
\end{lemma}

\begin{lemma}
$
    ||N(\tilde{ \psi}_{\epsilon},\tilde{ \phi}_{\epsilon})||_{Y}\leq C\mu_{1,\epsilon}^{Ap}.
$
\end{lemma}
\begin{proof}
 Direct computation shows that $$ ||N(\tilde{ \psi}_{\epsilon},\tilde{ \phi}_{\epsilon})||_{Y} \leq C(||\tilde{ \phi}_{\epsilon}^{p}||_{  L^{\frac{p+1}{p}}(\Omega)}+ ||\tilde{ \psi}_{\epsilon}^{q}||_{  L^{\frac{q+1}{q}}(\Omega)}),$$

\begin{equation*}
    \begin{split}
        ||\tilde{ \phi}_{\epsilon}^{p}||_{  L^{\frac{p+1}{p}}(\Omega)}\leq C \mu_{1,\epsilon}^{Ap}(\int_{\mathbb{R}^{N}}\frac{1}{(1+|y|)^{(\frac{N}{2}+\theta)(p+1)}}dy)^{\frac{p}{p+1}}
    \leq C\mu_{1,\epsilon}^{Ap},
    \end{split}
\end{equation*}
and
\begin{equation*}
    \begin{split}
        ||\tilde{ \psi}_{\epsilon}^{q}||_{  L^{\frac{q+1}{q}}(\Omega)}\leq C\mu_{1,\epsilon}^{Aq}(\int_{\mathbb{R}^{N}}\frac{1}{(1+|y|)^{\frac{N(p+1)}{2}}}dy)^{\frac{p}{p+1}}
       \leq  C\mu_{1,\epsilon}^{Aq}.
    \end{split}
\end{equation*}
Thus, $||N(\tilde{ \psi}_{\epsilon},\tilde{ \phi}_{\epsilon}))||_{Y}\leq C\mu_{1,\epsilon}^{Ap}$.
\end{proof}
Let $\theta$ be small enough, then we have $Ap> \frac{Npq}{p+1}$. Using Corollary 4.3 in \cite{Kim-Pis}, we can get that $ ||\tilde{ \psi}_{\epsilon},\tilde{ \phi}_{\epsilon}||_{X_{p,q}} \leq CC_{1} \mu^{\frac{Npq}{q+1}}$, where $C_1$ is the same as in Lemma \ref{ele1} and $ C$ is the same as in Proposition 4.6 in \cite{Kim-Pis}. Using the fixed point theorem, we get that the solution is unique. Thus from Proposition 4.6 in \cite{Kim-Pis}, we have $ ( \tilde{\psi}_{\epsilon},\tilde{\phi}_{\epsilon}) =  ( \psi_{\epsilon},\phi_{\epsilon})    $ where $( \psi_{\epsilon},\phi_{\epsilon}) $ is the same as in Theorem A. Thus $||( \psi_{\epsilon},\phi_{\epsilon})||_{**,\epsilon}\leq C \mu_{1,\epsilon}^{\frac{N(p+1)q}{q+1}-\frac{N+2}{2}-1-\theta}  $ .

\section{ Prove of Theorem \ref{Tmain}}

In this section, we will give the proof of Theorem \ref{Tmain}. For the simplicity, in the following, we will denote $d_{\beta_1,\beta_2,\alpha,N}$ as $d_0$.
Arguing by contradiction, suppose that there are $\epsilon_{k_{n}}\rightarrow 0$, satisfying $||(\eta_{n},\xi_{n})||_{*}=1$,  and
\begin{equation*}
    \begin{cases}
    -\Delta \eta_{n} = pv_{k_{n}}^{p-1}\xi_{n},&\hbox{ in }\Omega,
    \\ -\Delta \xi_{n} = qu_{k_{n}}^{q-1}\eta_{n} +\epsilon_{k_{n}}\beta_{2}\eta_{n},&\hbox{ in }\Omega,
    \\ \eta_{n}=\xi_{n}=0,&\hbox{ on }\partial\Omega.
    \end{cases}
\end{equation*}
Let $ (\tilde{\eta}_{n}(y),\tilde{\xi}_{n}(y)) = (\mu_{1,k_{n}}^{\frac{N}{q+1}} \eta_{n}(\mu_{1,k_{n}}y+P_{1,k_{n}}), \mu_{1,k_{n}}^{\frac{N}{p+1}} \xi_{n}(\mu_{1,k_{n}}y+P_{1,k_{n}})  ).$
\begin{lemma}\label{le4.1}
It holds
\begin{equation*}
    \begin{split}
       (\tilde{\eta}_{n}(y),\tilde{\xi}_{n}(y))\to \sum_{i=0}^{N}b_{i}(\Psi_{i},\Phi_{i}),
    \end{split}
\end{equation*}
uniformly in $C^{1}(B_{R}(0))$ for any $R>0$, where $b_{0},\cdots,b_{N}$ are some constants.
\end{lemma}
\begin{proof}
 In view of $| \tilde{\eta}_{n}  |+ |\tilde{\xi}_{n}|\leq C $, we may assume that $ (\tilde{\eta}_{n}(y),\tilde{\xi}_{n}(y))\to (f,g)$. Then $ (f,g)$  satisfies
 \begin{equation*}
     \begin{cases}
     -\Delta f = p V_{1,0}^{p-1} g, \hbox{ in }\mathbb{R}^{N},
     \\ -\Delta g = q U_{1,0}^{q-1} h, \hbox{ in }\mathbb{R}^{N},
     \end{cases}
 \end{equation*}
 which gives
 $
         (f,g) = \sum_{i=0}^{N}b_{i}(\Psi_{i},\Phi_{i}).
 $
\end{proof}

Set
\begin{align*}
\begin{cases}
    -\Delta \widetilde{P\Psi}_{1,k_{n},j}  = p(PV_{1,k_{n}})^{p-1}P\Phi_{1,k_{n},j} &\hbox{ in }\Omega
    \\\widetilde{P\Psi}_{1,k_{n},j} =0, &\hbox{ on }\partial \Omega,
\end{cases}
\end{align*}
for $j=1,\cdots,N$.
Then we decompose
\begin{equation*}
    \begin{split}
        (\eta_{n},\xi_{n}) = \sum_{i=0}^{N}b_{i,n}\mu_{1,k_{n}}(\widetilde{P\Psi}_{1,k_{n},i},P\Phi_{1,k_{n},i}) + (\eta_{n}^{*},\xi_{n}^{*}),
    \end{split}
\end{equation*}
where $  (\eta_{n}^{*},\xi_{n}^{*}) $ satisfies
\begin{equation*}
    \begin{split}
        \int_{\Omega}p(PV)_{1,k_{n}}^{p-1}P\Phi_{1,k_{n},i}\xi_{n}^{*} + qU_{1,k_{n}}^{q-1}\Psi_{1,k_{n},i}\eta_{n}^{*} dy =0.
    \end{split}
\end{equation*}
It follows form Lemma \ref{le4.1} that $ b_{0,n},\cdots,b_{N,n}$ are bounded.

Recalling that $||(\psi_{k_{n}},\phi_{k_{n}})||_{*}\leq C\mu_{1,k_{n}}^{\frac{N(p+1)q}{q+1}-\frac{N+2}{2}-1-\theta}$, and $A = \frac{N(p+1)q}{q+1} -\frac{N+2}{2}-1-\theta$.

\begin{lemma}
We have $\displaystyle ||( \eta_{n}^{*},\xi_{n}^{*}  )||_{*}\leq C \mu_{1,k_{n}}^{ A }  $.
\end{lemma}
\begin{proof}
First, we have
\begin{align*}
   -\Delta \eta^{*}_{n} = pv_{k_n}^{p-1}\xi^{*}_{n}+p(v_{k_n}^{p-1}\sum_{i=0}^{N}b_{i}\mu_{1,k_{n}} P\Phi_{1,k_{n},i}-\sum_{i=0}^{N}b_{i}\mu_{1,k_{n}} (PV_{1,k_{n}})^{p-1} P\Phi_{1,k_{n},i} ).
   \end{align*}
By directly computing, we get that
\begin{align*}
    &\sum_{i=0}^{N}|b_{i,n}||(v_{k_n}^{p-1}\mu_{1,k_{n}} P\Phi_{1,k_{n},i} -  (PV_{1,k_{n}})^{p-1}P\Phi_{1,k_{n},i}|
    \\\leq &C \sum_{i=0}^{N}|b_{i,n}|(PV_{1,k_{n}})^{p-2}\phi_{\epsilon}P\Phi_{1,k_{n},i}|
    \\\leq &C ||(\psi_{\epsilon},\phi_{\epsilon}||_{*} \frac{\mu_{1,k_{n}}^{-\frac{N}{q+1}-2}}{(1+\mu_{1,k_{n}}^{-1}|y-P_1  |)^{\frac{(N-2)p-2}{2} +2 +\theta }   }.
\end{align*}


We also have
\begin{equation*}
    \begin{split}
        -\Delta \xi_{n}^{*} = &qu_{k_{n}}^{q-1}\eta_{n}^{*} + \epsilon\beta_2\eta_{n}^{*}+\epsilon\beta_2\sum_{0=1}^{N}b_{i,n}\mu_{1,k_{n}} P\Psi_{1,k_{n},i}
        \\&+qu_{k_{n}}^{q-1}\sum_{0=1}^{N}b_{i,n}\mu_{1,k_{n}} P\Psi_{1,k_{n},i}-qU_{1,k_{n}}^{q-1}\sum_{0=1}^{N}b_{i,n}\mu_{1,k_{n}} \Psi_{1,k_{n},i}.
    \end{split}
\end{equation*}
By directly computing, we get that 
\begin{equation*}
    \begin{split}
        &|   qu_{k_{n}}^{q-1}\sum_{i=1}^{N}b_{i,n}\mu_{1,k_{n}}\widetilde{ P\Psi}_{1,k_{n},i}-qU_{1,k_{n}}^{q-1}\sum_{i=1}^{N}b_{i,n}\mu_{1,k_{n}} \Psi_{1,k_{n},i} |
        \\\leq&C| U_{1,k_{n}}^{q-1}\sum_{i=1}^{N}b_{i,n}\mu_{1,k_{n}} (\widetilde{ P\Psi}_{1,k_{n},i}   - \Psi_{1,k_{n},i}) |+ C |U_{1,k_{n}}^{q-2} \mu_{1,k_{n}}^{\frac{Np}{q+1}}\sum_{i=1}^{N}b_{i,n}\mu_{1,k_{n}}\widetilde{ P\Psi}_{1,k_{n},i}|
        \\&+ C|U_{1,k_{n}}^{q-2} \psi_{k_{n}}\sum_{i=1}^{N}b_{i,n}\mu_{1,k_{n}}\widetilde{ P\Psi}_{1,k_{n},i}   |
:=J_1+J_2+J_3
    \end{split}
\end{equation*}
Computing $J_1$, $J_2$, $J_3 $ one by one, we get that
\begin{equation*}
    \begin{split}
        | J_1| 
        \leq& C b_{0,n}\mu_{1,k_{n}}^{\frac{N(p+1)q}{q+1}-\frac{N+2}{2}-1-\theta} +  \frac{\mu_{1,k_{n}}^{-\frac{N}{p+1}-2}}{  (  1+ \mu_{1,k_{n}}^{-1}|x-P_{1,k_{n}}| )^{ \frac{N+2}{2}+1+\theta }   }
        \\&+C\sum_{i=1}^{N}b_{i,n}\mu_{1,k_{n}}^{\frac{N(p+1)q}{q+1}-\frac{N+2}{2}-\theta} +  \frac{\mu_{1,k_{n}}^{-\frac{N}{p+1}-2}}{  (  1+ \mu_{1,k_{n}}^{-1}|x-P_{1,k_{n}}| )^{ \frac{N+2}{2}+1+\theta }   },
    \end{split}
\end{equation*}
\begin{equation*}
    \begin{split}
    |J_2|\leq &C b_{0,n}\mu_{1,k_{n}}^{\frac{N(p+1)q}{q+1}-\frac{N+2}{2}-1-\theta}  \frac{\mu_{1,k_{n}}^{-\frac{N}{p+1}-2}}{  (  1+ \mu_{1,k_{n}}^{-1}|x-P_{1,k_{n}}| )^{ \frac{N+2}{2}+1+\theta }   }
        \\ &+C\sum_{i=1}^{N}b_{i,n}\mu_{1,k_{n}}^{\frac{N(p+1)q}{q+1}-\frac{N+2}{2}-\theta}  \frac{\mu_{1,k_{n}}^{-\frac{N}{p+1}-2}}{  (  1+ \mu_{1,k_{n}}^{-1}|x-P_{1,k_{n}}| )^{ \frac{N+2}{2}+1+\theta }   },
    \end{split}
\end{equation*}
and
\begin{equation*}
    \begin{split}
        | J_3|
        \leq &C\mu^{\frac{N(p+1)}{q+1}(2q-\frac{1}{2})-\theta -N-4}b_{0,n}\frac{\mu_{1,k_{n}}^{-\frac{N}{p+1}-2}}{  (  1+ \mu_{1,k_{n}}^{-1}|x-P_{1,k_{n}}| )^{ \frac{N+2}{2}+1+\theta }   }
        \\&+C\sum_{i=1}^{N}b_{i,n}\mu_{1,k_{n}}^{\frac{N(p+1)}{q+1}(2q-\frac{1}{2}+1)-\theta -N-4}\frac{\mu_{1,k_{n}}^{-\frac{N}{p+1}-2}}{  (  1+ \mu_{1,k_{n}}^{-1}|x-P_{1,k_{n}}| )^{ \frac{N+2}{2}+1+\theta }   }.
    \end{split}
\end{equation*}
Finally, we have
\begin{equation*}
    \begin{split}
       &|\epsilon\beta_2\sum_{i=1}^{N}b_{i,n}\mu_{1,k_{n}}\widetilde{ P\Psi}_{1,k_{n},i} |
       \\\leq &Cb_{0,n}\mu_{1,k_{n}}^{\frac{2N(p+1)}{q+1}   -\frac{N+4}{2}}\frac{\mu_{1,k_{n}}^{-\frac{N}{p+1}-2}}{  (  1+ \mu_{1,k_{n}}^{-1}|x-P_{1,k_{n}}| )^{ \frac{N+2}{2}+1+\theta }   }
       \\&+C\sum_{i=1}^{N}b_{i,n}\mu_{1,k_{n}}^{\frac{2N(p+1)}{q+1} +1  -\frac{N+4}{2}}\frac{\mu_{1,k_{n}}^{-\frac{N}{p+1}-2}}{  (  1+ \mu_{1,k_{n}}^{-1}|x-P_{1,k_{n}}| )^{ \frac{N+2}{2}+1+\theta }   }.
    \end{split}
\end{equation*}

Set
\begin{align*}
&(g_n,f_n) =
(qu_{k_{n}}^{q-1}\sum_{i=1}^{N}b_{i,n}\mu_{1,k_{n}}\widetilde{ P\Psi}_{1,k_{n},i}
-qU_{1,k_{n}}^{q-1}\sum_{i=1}^{N}b_{i,n}\mu_{1,k_{n}} \Psi_{1,k_{n},i}
\\&+\epsilon\beta_2\sum_{i=1}^{N}b_{i,n}\mu_{1,k_{n}}\widetilde{ P\Psi}_{1,k_{n},i},pv_{k_{n}}^{p-1}\sum_{i=0}^{N}b_{i,n}\mu_{1,k_{n}} P\Phi_{1,k_{n},i} - pV_{1,k_{n}}^{p-1}\sum_{i=0}^{N}b_{i,n}\mu_{1,k_{n}} \Phi_{1,k_{n},i}).
\end{align*}
Then
\begin{align*}
||(g_n,f_n)||_{**}\leq Cb_{0,n} \mu_{1,k_{n}}^{ A } +C\sum_{i=1}^{N}b_{i,n}\mu_{1,k_{n}}^{ A+1 }.
\end{align*}
Since
\begin{equation*}
    \begin{split}
        \int_{\Omega}p(PV)_{1,k_{n}}^{p-1}P\Phi_{1,k_{n},i}\xi_{n}^{*} + qU_{1,k_{n}}^{q-1}\Psi_{1,k_{n},i}\eta_{n}^{*} dy =0,
    \end{split}
\end{equation*}
and similar to the proof of Lemma \ref{le3.1}, we have
$$   ||(\eta_{n}^{*},\xi_{n}^{*}  )||_{*}\leq C||(g_n,f_n)||_{**}\leq Cb_{0,n} \mu_{1,k_{n}}^{ A } +C\sum_{i=1}^{N}b_{i,n}\mu_{1,k_{n}}^{ A+1 } .$$
\end{proof}

\begin{lemma}\label{le4.3}For any $i=1,\cdots,N$, we have
\begin{equation*}
    \begin{split}
        |\frac{\partial \psi_{k_{n}}}{\partial x_i}(x)| \leq C\frac{\mu_{1,k_{n}}^{A-1-\frac{N}{q+1}}}{(  1+ \mu_{1,k_{n}}^{-1}|x-P_{1,k_{n}}| )^{\frac{(N-2)p}{2}+\theta+\frac{(N-2)(p-1)}{2}}  },
    \end{split}
\end{equation*}
and
\begin{equation*}
    \begin{split}
        |\frac{\partial \phi_{k_{n}}}{\partial x_i}(x)| \leq C \frac{\mu_{1,k_{n}}^{  \frac{N(p-q+1)}{q+1} +1 }  }{(  1+ \mu_{1,k_{n}}^{-1}|x-P_{1,k_{n}}| )^{\frac{N(p+1)}{q+1}-1 } }+C\frac{\mu_{1,k_{n}}^{-\frac{N}{p+1}-1+A}}{ (  1+ \mu_{1,k_{n}}^{-1}|x-P_{1,k_{n}}| )^{  \frac{N+2}{2}+\theta}       } .
    \end{split}
\end{equation*}
\end{lemma}
\begin{proof} Since
\begin{equation*}
    \begin{cases}
    -\Delta\psi_{k_{n}} = (PV_{1,k_{n}} +\phi_{k_{n}})^{p} - PV^{p}_{1,k_{n}}, &\hbox{ in }\Omega,
    \\ \psi_{k_{n}} = 0, &\hbox{ on }\partial\Omega,
    \end{cases}
\end{equation*}
we have
\begin{equation*}
    \begin{split}
        \psi_{k_{n}}(x) = \int_{\Omega}G(x,y) \bigg((PV_{1,k_{n}} +\phi_{k_{n}})^{p} - PV^{p}_{1,k_{n}}\bigg)dy.
    \end{split}
\end{equation*}
By direct computation, we also have
\begin{equation*}
    \begin{split}
        \frac{\partial \psi_{k_{n}}}{\partial x_i}(x) = \int_{\Omega}\frac{\partial G}{\partial x_i}(x,y) \bigg((PV_{1,k_{n}} +\phi_{k_{n}})^{p} - PV^{p}_{1,k_{n}}\bigg)dy.
    \end{split}
\end{equation*}
So
\begin{equation*}
    \begin{split}
       | \frac{\partial \psi_{k_{n}}}{\partial x_i}(x)  |
        \leq& C \int_{\Omega}|\frac{1}{|x-y|^{N-1}}V_{1,k_{n}}^{p-1}\phi_{k_{n}}|dy +C \int_{\Omega}|\frac{1}{|x-y|^{N-1}}\phi^{p}_{n}|dy
        \\:=&I_1+I_2.
    \end{split}
\end{equation*}
Further computation shows that
\begin{equation*}
    \begin{split}
        |I_1|\leq& C\int_{\Omega}\frac{1}{|x-y|^{N-1}}\frac{\mu_{1,k_{n}}^{-\frac{Np}{p+1}+A}}{  (  1+ \mu_{1,k_{n}}^{-1}|x-P_{1,k_{n}}| )^{(N-2)(p-1) +\frac{N-2}{2}+1 +\theta}   }
        \\\leq &C\frac{\mu_{1,k_{n}}^{A-1-\frac{N}{q+1}}}{(  1+ \mu_{1,k_{n}}^{-1}|x-P_{1,k_{n}}| )^{\frac{(N-2)p}{2}+\theta+\frac{(N-2)(p-1)}{2}}  },
    \end{split}
\end{equation*}
and
\begin{equation*}
    \begin{split}
        |I_2|\leq &C\int_{\Omega}\frac{1}{|x-y|^{N-1}}\frac{\mu_{1,k_{n}}^{-\frac{Np}{p+1}+Ap}}{  (  1+ \mu_{1,k_{n}}^{-1}|x-P_{1,k_{n}}| )^{(\frac{N-2}{2}+1 +\theta)p}   }
        \\\leq &C\frac{\mu_{1,k_{n}}^{pA-1-\frac{N}{q+1}}}{(  1+ \mu_{1,k_{n}}^{-1}|x-P_{1,k_{n}}| )^{\frac{(N-2)p}{2}+\theta+(p-1)(1+\theta)}  }.
    \end{split}
\end{equation*}

On the other hand, since
\begin{equation*}
    \begin{cases}
    -\Delta \phi_{k_{n}} = ( PU_{d_{1,k_{n}},P_{1,k_{n}}}  +\psi_{k_{n}})^{q} - U_{1,k_{n}}^{q} + \epsilon\beta_2(PU_{d_{1,k_{n}},P_{1,k_{n}}}  +\psi_{k_{n}}   ), \hbox{ in }\Omega.
    \\ \phi_{k_{n}}=0,\hbox{ on }\partial\Omega,
    \end{cases}
\end{equation*}
we have
\begin{equation*}
    \phi_{k_{n}}(x) = \int_{\Omega}G(x,y) \bigg(( PU_{d_{1,k_{n}},P_{1,k_{n}}}  +\psi_{k_{n}})^{q} - U_{1}^{q} + \epsilon\beta_2(PU_{d_{1,k_{n}},P_{1,k_{n}}}  +\psi_{k_{n}}   ) \bigg)dy,
\end{equation*}
and
\begin{equation*}
    \begin{split}
        \frac{\partial\phi_{k_{n}} }{\partial x_i}(x)= \int_{\Omega}\frac{\partial G }{\partial x_i}(x,y) \bigg(( PU_{d_{1,k_{n}},P_{1,k_{n}}}  +\psi_{k_{n}})^{q} - U_{1}^{q} + \epsilon\beta_2(PU_{d_{1,k_{n}},P_{1,k_{n}}}  +\psi_{k_n}   )\bigg) dy.
    \end{split}
\end{equation*}
Thus, the estimate of $|\frac{\partial\phi_{k_{n}} }{\partial x_i}(x) |$ given by
\begin{equation*}
    \begin{split}
        |\frac{\partial\phi_{k_{n}} }{\partial x_i}(x) |
        \leq &C|\int_{\Omega}\frac{1}{|x-y|^{N-1}}U_{1,k_{n}}^{q-1}(\psi_{k_{n}}+\mu_{1,k_{n}}^{\frac{Np}{q+1}})dy|+ C|\int_{\Omega}\frac{1}{|x-y|^{N-1}}(\psi_{k_{n}}^{q} + \mu_{1,k_{n}}^{\frac{Npq}{q+1}})dy|
        \\&+C|\int_{\Omega}\frac{1}{|x-y|^{N-1}}\epsilon\beta_2(PU_{d_{1,k_{n}},P_{1,k_{n}}}  +\psi_{k_n}   )dy|
        \\:=& J_1+J_2+J_3.
    \end{split}
\end{equation*}
Moreover, we have
\begin{equation*}
    \begin{split}
        |J_1|
        \leq &C\int_{\Omega}\frac{1}{|x-y|^{N-1}}\frac{\mu_{1,k_{n}}^{-\frac{Nq}{q+1}+A}}{  (  1+ \mu_{1,k_{n}}^{-1}|x-P_{1,k_{n}}| )^{  \frac{N(p+1)}{q+1}(q-\frac{1}{2})+\theta}    }\\&+C\int_{\Omega}\frac{1}{|x-y|^{N-1}}\frac{\mu_{1,k_{n}}^{-\frac{N(q-1)}{q+1}+\frac{Np}{q+1}}}{  (  1+ \mu_{1,k_{n}}^{-1}|x-P_{1,k_{n}}| )^{  \frac{N(p+1)}{q+1}(q-1)}}
        \\\leq &C\frac{\mu_{1,k_{n}}^{-\frac{N}{p+1}-1+A}}{ (  1+ \mu_{1,k_{n}}^{-1}|x-P_{1,k_{n}}| )^{  \frac{N+2}{2}+\theta}       },
    \end{split}
\end{equation*}
where we have used the fact that $ \frac{N(p+1)}{q+1}(q-\frac{1}{2})+\theta-1 > \frac{N+2}{2}+\theta $. Similarly,
 \begin{equation*}
     \begin{split}
         |J_2|
       \leq &C\frac{\mu_{1,k_{n}}^{-\frac{N}{p+1}-1+Ap}}{ (  1+ \mu_{1,k_{n}}^{-1}|x-P_{1,k_{n}}| )^{  (\frac{N(p+1)}{2(q+1)}+\theta)q-1}       }+ C\frac{\mu_{1,k_{n}}^{-\frac{N}{p+1}-1+A}}{ (  1+ \mu_{1,k_{n}}^{-1}|x-P_{1,k_{n}}| )^{  \frac{N+2}{2}+\theta}       }
        \\\leq &C\frac{\mu_{1,k_{n}}^{-\frac{N}{p+1}-1+A}}{ (  1+ \mu_{1,k_{n}}^{-1}|x-P_{1,k_{n}}| )^{  \frac{N+2}{2}+\theta}       },
     \end{split}
 \end{equation*}
 and
 \begin{equation*}
     \begin{split}
         &|J_3|
         \\\leq&C \int_{\Omega}\frac{1}{|x-y|^{N-1}}\mu_{1,k_{n}}^{\frac{N(p-q+2)}{1+q}}(\frac{\mu_{1,k_{n}}^{-\frac{N}{q+1}}}{  (  1+ \mu_{1,k_{n}}^{-1}|x-P_{1,k_{n}}| )^{  \frac{N(p+1)}{q+1}}}+\frac{ \mu_{1,k_{n}}^{ -\frac{N}{q+1}+A     }}{   (  1+ \mu_{1,k_{n}}^{-1}|x-P_{1,k_{n}}| )^{  \frac{N(p+1)}{2(q+1)}+\theta }  }     )
         \\\leq&C \frac{\mu_{1,k_{n}}^{  \frac{N(p-q+1)}{q+1} +1 }  }{(  1+ \mu_{1,k_{n}}^{-1}|x-P_{1,k_{n}}| )^{\frac{N(p+1)}{q+1}-1 } }+C\frac{\mu_{1,k_{n}}^{-\frac{N}{p+1}-1+A}}{ (  1+ \mu_{1,k_{n}}^{-1}|x-P_{1,k_{n}}| )^{  \frac{N+2}{2}+\theta}       }.
     \end{split}
 \end{equation*}

\end{proof}

\begin{lemma}\label{le4.5}
For $\alpha$ small, $ x \in \partial B_{\mu_{1,k_{n}}^{\alpha}}(P_{1,k_{n}})$ we have that
\begin{equation*}
    \begin{split}
        |\frac{\partial \eta_{n}^{*} }{ \partial x_i}(x) |\leq C \mu_{1,k_{n}}^{\frac{Np}{q+1} +1 + \sigma_{p,q,N}-C_{N}\alpha},
    \end{split}
\end{equation*}
and
\begin{equation*}
    \begin{split}
         |    \frac{\partial \xi_{n}^{*} }{ \partial x_i}(x) |\leq C\mu_{1,k_{n}}^{\frac{N}{q+1}+1+\sigma_{p,q,N}-C_{N}\alpha},
    \end{split}
\end{equation*}
where $\sigma_{p,q,N}>0 $ and $C_{N}>0 $ are constants.
\end{lemma}
\begin{proof}
Since
\begin{equation*}
    \begin{split}
        -\Delta \xi_{n}^{*} = &qu_{k_{n}}^{q-1}\eta_{n}^{*} + \epsilon\beta_2\eta_{n}^{*}+\epsilon\beta_2\sum_{i=0}^{N}b_{i,n}\mu_{1,k_{n}} \widetilde{P\Psi}_{1,k_{n},i}
        \\&+qu_{k_{n}}^{q-1}\sum_{i=0}^{N}b_{i,n}\mu_{1,k_{n}}\widetilde{P\Psi}_{1,k_{n},i}-qU_{1,k_{n}}^{q-1}\sum_{i=0}^{N}b_{i,n}\mu_{1,k_{n}} \Psi_{1,k_{n},i},
    \end{split}
\end{equation*}
we have
\begin{equation*}
    \begin{split}
         \frac{\partial \xi_{n}^{*} }{ \partial x_i}(x) = &\int_{\Omega}\frac{\partial G}{ \partial x_i }(x,y)( qu_{k_{n}}^{q-1}\eta_{n}^{*} + \epsilon\beta_2\eta_{n}^{*}+\epsilon\beta_2\sum_{i=0}^{N}b_{i,n}\mu_{1,k_{n}} \widetilde{P\Psi}_{1,k_{n},i}
        \\&+qu_{k_{n}}^{q-1}\sum_{i=0}^{N}b_{i,n}\mu_{1,k_{n}} \widetilde{P\Psi}_{1,k_{n},i}-qU_{1,k_{n}}^{q-1}\sum_{i=0}^{N}b_{i,n}\mu_{1,k_{n}} \Psi_{1,k_{n},i}).
    \end{split}
\end{equation*}

By directly computing, we get that
\begin{equation*}
    \begin{split}
        &|\int_{\Omega}\frac{\partial G}{ \partial x_i }(x,y)qu_{k_{n}}^{q-1}\eta_{n}^{*}|\\ \leq &C|b_{0,n}|\frac{\mu_{1,k_{n}}^{A+1-\frac{Nq}{q+1}}}{(  1+\mu_{1,k_{n}}^{-1}|x-P_{1,k_{n}}|)^{ \frac{N(p+1)}{q+1}(q-\frac{1}{2})+\theta-1}    }
        \\\leq &C\mu_{1,k_{n}}^{\frac{N}{q+1}+1+\sigma_{p,q,N}-C_{N}\alpha},
    \end{split}
\end{equation*}
\begin{equation*}
    \begin{split}
        &|\int_{\Omega}\frac{\partial G}{\partial x_i} \epsilon\beta_2\eta_{n}^{*}|
        \\\leq &C\frac{\mu_{1,k_{n}}^{\frac{N(p-q+2)}{q+1}+A+1-\frac{N}{q+1}}}{(  1+\mu_{1,k_{n}}^{-1}|x-P_{1,k_{n}}|)^{ \frac{N(p+1)}{2(q+1)}+\theta-1} }\leq C\mu_{1,k_{n}}^{\frac{N}{q+1}+1+\sigma_{p,q,N}-C_{N}\alpha},
    \end{split}
\end{equation*}
and
\begin{equation*}
    \begin{split}
        &|\int_{\Omega}\frac{\partial G}{\partial x_i}\epsilon\beta_2\sum_{i=0}^{N}b_{i,n}\mu_{1,k_{n}}\widetilde{P\Psi}_{1,k_{n},i} | \\ \leq &C\frac{\mu_{1,k_{n}}^{\frac{N(p-q+2)}{q+1}+1-\frac{N}{q+1}}}{(  1+\mu_{1,k_{n}}^{-1}|x-P_{1,k_{n}}|)^{ \frac{N(p+1)}{(q+1)}-1} }
        \leq C\mu_{1,k_{n}}^{\frac{N}{q+1}+1+\sigma_{p,q,N}-C_{N}\alpha}.
    \end{split}
\end{equation*}

We also have
\begin{equation*}
    \begin{split}
        &|\int_{\Omega}\frac{\partial G}{\partial x_i}(qu_{k_{n}}^{q-1}\sum_{i=0}^{N}b_{i,n}\mu_{1,k_{n}} \widetilde{P\Psi}_{1,k_{n},i}-qU_{1,k_{n}}^{q-1}\sum_{0=1}^{N}b_{i,n}\mu_{1,k_{n}} \Psi_{1,k_{n},i})dy|
        \\\leq &C\int_{\Omega}\frac{1}{|x-y|^{N-1}}U_{1,k_{n}}^{q-1}(|b_{0,n}|\mu_{1,k_{n}}^{  \frac{Np}{q+1}  }+\sum_{i=1}^{N}|b_{i,n}|\mu_{1,k_{n}}^{1+\frac{Np}{q+1}})dy
        \\&+C\int_{\Omega}\frac{1}{|x-y|^{N-1}}U_{1,k_{n}}^{q-2}(|\psi_{k_{n}}|+\mu_{1,k_{n}}^{ \frac{Np}{q+1}}) \sum_{i=0}^{N}|b_{i,n}|\mu_{1,k_{n}}|\Psi_{1,k_{n},i}|dy
        \\=:&J_1+J_2.
    \end{split}
\end{equation*}
 Noting  $  -\frac{N(q-1)}{q+1}+\frac{Np}{q+1}+1 +\frac{N(p+1)(q-1)}{q+1}-1 -\frac{N}{q+1}=\frac{N(pq-1)}{q+1}  $, we obtain
\begin{equation*}
    \begin{split}
        |J_1|
        \leq& C\frac{\mu_{1,k_{n}}^{-\frac{N(q-1)}{q+1}+\frac{Np}{q+1}+1}}{  (  1+\mu_{1,k_{n}}^{-1}|x-P_{1,k_{n}}|)^{    \frac{N(p+1)(q-1)}{q+1}-1} }\leq C\mu_{1,k_{n}}^{\frac{N}{q+1}+1+\sigma_{p,q,N}-C_{N}\alpha},
    \end{split}
\end{equation*}

Since $  -\frac{N(q-1)}{q+1}+\frac{Np}{q+1}+1 +\frac{N(p+1)(q-1)}{q+1}-1 -\frac{N}{q+1}=\frac{N(pq-1)}{q+1}  $ and $-\frac{Nq}{q+1}+A+1 +\frac{N(p+1)}{q+1}(q-\frac{1}{2})-1+\theta > \frac{N(pq-1)}{q+1}$, we have
\begin{equation*}
    \begin{split}
        &|J_2|\\\leq &C|b_{0,n}|\frac{\mu_{1,k_{n}}^{-\frac{N(q-1)}{q+1}+\frac{Np}{q+1}+1}}{  (  1+\mu_{1,k_{n}}^{-1}|x-P_{1,k_{n}}|)^{    \frac{N(p+1)(q-1)}{q+1}-1} }+C\sum_{i=1}^{N}|b_{i,n}|\frac{\mu_{1,k_{n}}^{-\frac{N(q-1)}{q+1}+\frac{Np}{q+1}+1}}{  (  1+\mu_{1,k_{n}}^{-1}|x-P_{1,k_{n}}|)^{    \frac{N(p+1)(q-1)}{q+1}}  }
        \\&+C|b_{0,n}|\frac{\mu_{1,k_{n}}^{-\frac{Nq}{q+1}+A+1}}{  (  1+\mu_{1,k_{n}}^{-1}|x-P_{1,k_{n}}|)^{    \frac{N(p+1)}{q+1}(q-\frac{1}{2})-1+\theta}}+C\sum_{i=1}^{N}|b_{i,n}|\frac{\mu_{1,k_{n}}^{-\frac{Nq}{q+1}+A+1}}{  (  1+\mu_{1,k_{n}}^{-1}|x-P_{1,k_{n}}|)^{    \frac{N(p+1)}{q+1}(q-\frac{1}{2})+\theta}}
        \\\leq &C\mu_{1,k_{n}}^{\frac{N}{q+1}+1+\sigma_{p,q,N}-C_{N}\alpha}.
    \end{split}
\end{equation*}
Similarly, we can obtain the estimate for  $\frac{ \partial  \eta^{*}_{n}}{ \partial x_{i}} $.
\end{proof}


\begin{lemma}
It holds that $ b_{i,n} \to 0$, as $n \to +\infty$, for $i=0,\cdots,N$.
\end{lemma}
\begin{proof}

\textbf{Step 1.} We apply the identities in Lemma \ref{le2.1} in the ball $B_{\mu_{1,k_{n}}^{\alpha}}(P_{1,k_{n}}) $, where $\alpha>0$ is a small fix constant.
\begin{equation}\label{eq4.1}
    \begin{split}
       &\int_{\partial B_{\mu_{1,k_{n}}^{\alpha}}(P_{1,k_{n}})  }(-\frac{\partial u_{k_{n}}}{\partial\nu}\langle \nabla\xi_{n},x-P_{1,k_{n}} \rangle -\frac{\partial \xi_{n}}{\partial\nu}\langle \nabla u_{k_{n}},x-P_{1,k_{n}} \rangle
       \\&+\langle  \nabla u_{k_{n}},\nabla\xi_{n}\rangle\langle \nu,x-P_{1,k_{n}}\rangle)ds-\frac{N}{p+1}\int_{\partial B_{\mu_{1,k_{n}}^{\alpha}}(P_{1,k_{n}})}(\frac{\partial u_{k_{n}}}{\partial\nu}\xi_{n}+\frac{\partial\eta_{n}}{\partial\nu}v_{k_{n}})ds
       \\&+\int_{\partial B_{\mu_{1,k_{n}}^{\alpha}}(P_{1,k_{n}})}(-\frac{\partial v_{k_{n}}}{\partial\nu}\langle \nabla\eta_{n},x-P_{1,k_{n}} \rangle -\frac{\partial \eta_{n}}{\partial\nu}\langle \nabla v_{k_{n}},x-P_{1,k_{n}} \rangle
       \\&+\langle  \nabla v_{k_{n}},\nabla\eta_{n}\rangle\langle \nu,x-P_{1,k_{n}}\rangle)ds
       -\frac{N}{q+1}\int_{\partial B_{\mu_{1,k_{n}}^{\alpha}}(P_{1,k_{n}})}(\frac{\partial v_{k_{n}}}{\partial\nu}\eta_{n}+\frac{\partial\xi_{n}}{\partial\nu}u_{k_{n}})ds
       \\=&(-N+\frac{2N}{q+1})\epsilon\beta_2\int_{B_{\mu_{1,k_{n}}^{\alpha}}(P_{1,k_{n}})}u_{k_{n}}\eta_{n} dx+\epsilon\beta_2\int_{\partial B_{\mu_{1,k_{n}}^{\alpha}}(P_{1,k_{n}})}u_{k_{n}}\eta_{n}\langle\nu,x-P_{1,k_{n}} \rangle ds
       \\&+ \int_{\partial B_{\mu_{1,k_{n}}^{\alpha}}(P_{1,k_{n}})}(u_{k_{n}}^{q}\xi_{n}+v_{k_{n}}^{p}\eta_{n})\langle \nu,x-P_{1,k_{n}}\rangle ds.
    \end{split}
\end{equation}
We compute,
\begin{equation*}
    \begin{split}
        &\int_{\partial B_{\mu_{1,k_{n}}^{\alpha}}(P_{1,k_{n}})  }(-\frac{\partial u_{k_{n}}}{\partial\nu}\langle \nabla\xi_{n},x-P_{1,k_{n}} \rangle -\frac{\partial \xi_{n}}{\partial\nu}\langle \nabla u_{k_{n}},x-P_{1,k_{n}} \rangle
        \\&+\langle  \nabla u_{k_{n}},\nabla\xi_{n}\rangle\langle \nu,x-P_{1,k_{n}}\rangle)ds-\frac{N}{p+1}\int_{\partial B_{\mu_{1,k_{n}}^{\alpha}}(P_{1,k_{n}})}(\frac{\partial u_{k_{n}}}{\partial\nu}\xi_{n}+\frac{\partial\eta_{n}}{\partial\nu}v_{k_{n}})ds
        \\&+\int_{\partial B_{\mu_{1,k_{n}}^{\alpha}}(P_{1,k_{n}})}(-\frac{\partial v_{k_{n}}}{\partial\nu}\langle \nabla\eta_{n},x-P_{1,k_{n}} \rangle -\frac{\partial \eta_{n}}{\partial\nu}\langle \nabla v_{k_{n}},x-P_{1,k_{n}} \rangle
        \\&+\langle  \nabla v_{k_{n}},\nabla\eta_{n}\rangle\langle \nu,x-P_{1,k_{n}}\rangle)ds
       -\frac{N}{q+1}\int_{\partial B_{\mu_{1,k_{n}}^{\alpha}}(P_{1,k_{n}})}(\frac{\partial v_{k_{n}}}{\partial\nu}\eta_{n}+\frac{\partial\xi_{n}}{\partial\nu}u_{k_{n}})ds
       \\=&\int_{  B_{\mu_{1,k_{n}}^{\alpha}}(P_{1,k_{n}}) }-\Delta u_{k_{n}}\langle    \nabla\xi_{n},x-P_{1,k_{n}}  \rangle -\Delta \xi_{n}\langle    \nabla u_{k_{n}},x-P_{1,k_{n}}  \rangle dy
       \\&+\int_{  B_{\mu_{1,k_{n}}^{\alpha}}(P_{1,k_{n}}) }-\Delta v_{k_{n}}\langle    \nabla\eta_{n},x-P_{1,k_{n}}  \rangle -\Delta \eta_{n}\langle    \nabla v_{k_{n}},x-P_{1,k_{n}}  \rangle dy
       \\&+\frac{N}{p+1}\int_{  B_{\mu_{1,k_{n}}^{\alpha}}(P_{1,k_{n}}) }-\Delta u_{k_{n}}\xi_{n}-\Delta \eta_{n}v_{k_{n}}dy
       \\&+\frac{N}{q+1}\int_{  B_{\mu_{1,k_{n}}^{\alpha}}(P_{1,k_{n}}) }-\Delta v_{k_{n}}\eta_{n}-\Delta \xi_{n}u_{k_{n}}dy.
    \end{split}
\end{equation*}
Set
\begin{equation*}
    \begin{split}
        &I_{1,n}( U,V,\eta,\xi  )
        = \int_{  B_{\mu_{1,k_{n}}^{\alpha}}(P_{1,k_{n}}) }-\Delta U\langle    \nabla\xi,x-P_{1,k_{n}}  \rangle -\Delta \xi\langle    \nabla U,x-P_{1,k_{n}}  \rangle dy
       \\&+\int_{  B_{\mu_{1,k_{n}}^{\alpha}}(P_{1,k_{n}}) }-\Delta V\langle    \nabla\eta,x-P_{1,k_{n}}  \rangle -\Delta \eta\langle    \nabla V_{1,k_{n}},x-P_{1,k_{n}}  \rangle dy
       \\&+\frac{N}{p+1}\int_{  B_{\mu_{1,k_{n}}^{\alpha}}(P_{1,k_{n}}) }-\Delta U\xi-\Delta \eta Vdy
    +\frac{N}{q+1}\int_{  B_{\mu_{1,k_{n}}^{\alpha}}(P_{1,k_{n}}) }-\Delta V\eta-\Delta \xi Udy,
    \end{split}
\end{equation*}
then
\begin{equation*}
    \begin{split}
        &I_{1,n}( U_{1,k_{n}},V_{1,k_{n}},\eta_{n},\xi_{n}  )
        \\=& I_{1,n}( PU_{d_{1,k_{n}},P_{1,k_{n}}},PV_{1,k_{n}},\sum_{i=0}^{N}b_{i,n}\mu_{1,k_{n}} \widetilde{P\Psi}_{1,k_{n},i}+\eta_{1,n}^{*}, \sum_{i=0}^{N}b_{i,n}\mu_{1,k_{n}} P\Phi_{1,k_{n},i} )
        + H_n
        \\=&\int_{  B_{\mu_{1,k_{n}}^{\alpha}}(P_{1,k_{n}}) }-\Delta PU_{d_{1,k_{n}},P_{1,k_{n}}} \langle    \nabla\sum_{i=0}^{N}b_{i,n}\mu_{1,k_{n}} P\Phi_{1,k_{n},i} ,x-P_{1,k_{n}}  \rangle
        \\&+\int_{  B_{\mu_{1,k_{n}}^{\alpha}}(P_{1,k_{n}}) }-\Delta V_{1,k_{n}}\langle    \nabla\sum_{i=0}^{N}b_{i,n}\mu_{1,k_{n}} \widetilde{P\Psi}_{1,k_{n},i},x-P_{1,k_{n}}  \rangle
        \\&+\int_{  B_{\mu_{1,k_{n}}^{\alpha}}(P_{1,k_{n}}) }-\Delta  \sum_{i=0}^{N}b_{i,n}\mu_{1,k_{n}} \widetilde{P\Psi}_{1,k_{n},i} \langle \nabla PV_{1,k_{n}} ,x-P_{1,k_{n}}  \rangle
        \\&+\int_{  B_{\mu_{1,k_{n}}^{\alpha}}(P_{1,k_{n}}) }-\Delta  \sum_{i=0}^{N}b_{i,n}\mu_{1,k_{n}} P\Phi_{1,k_{n},i} \langle \nabla PU_{d_{1,k_{n}},P_{1,k_{n}}} ,x-P_{1,k_{n}}  \rangle
        \\&+\frac{N}{q+1}\int_{  B_{\mu_{1,k_{n}}^{\alpha}}(P_{1,k_{n}}) }-\Delta  \sum_{i=0}^{N}b_{i,n}\mu_{1,k_{n}} P\Phi_{1,k_{n},i} PU_{d_{1,k_{n}},P_{1,k_{n}}}
        -\Delta PV_{1,k_{n}} \sum_{i=0}^{N}b_{i,n}\mu_{1,k_{n}} \widetilde{P\Psi}_{1,k_{n},i}
        \\&+ \frac{N}{p+1}\int_{  B_{\mu_{1,k_{n}}^{\alpha}}(P_{1,k_{n}}) }-\Delta PU_{d_{1,k_{n}},P_{1,k_{n}}} \sum_{i=0}^{N}b_{i,n}\mu_{1,k_{n}} P\Phi_{1,k_{n},i}
        -\Delta \sum_{i=0}^{N}b_{i,n}\mu_{1,k_{n}} \widetilde{P\Psi}_{1,k_{n},i} PV_{1,k_{n}}
       +H_n
     \\=&I_{1,n}^{1}+...+I_{1,n}^{8}+H_n,
    \end{split}
\end{equation*}
where $H_{n}$ is the remainder term. In the following, let us compute the above terms one by one.

First, we compute $I_{1,n}^{1}$ and $I_{1,n}^{2}$. Since $\nabla \tau (P_{0}) = 0$, using Lemma \ref{Ble8}, we have
\begin{equation*}
    \begin{split}
        I_{1,n}^{1}=&\int_{  B_{\mu_{1,k_{n}}^{\alpha}}(P_{1,k_{n}}) }-\Delta PU_{d_{1,k_{n}},P_{1,k_{n}}} \langle    \nabla\sum_{i=0}^{N}b_{i,n}\mu_{1,k_{n}} P\Phi_{1,k_{n},i} ,x-P_{1,k_{n}}  \rangle
        \\=& \int_{  B_{\mu_{1,k_{n}}^{\alpha}}(P_{1,k_{n}}) }PV_{1,k_{n}}^{p}\langle    \nabla\sum_{i=0}^{N}b_{i,n}\mu_{1,k_{n}} P\Phi_{1,k_{n},i} ,x-P_{1,k_{n}}  \rangle
           \\=&\int_{  B_{\mu_{1,k_{n}}^{\alpha}}(P_{1,k_{n}}) }V_{1,k_{n}}^{p}\langle    \nabla b_{0,n}\mu_{1,k_{n}} \Phi_{1,k_{n},0} ,x-P_{1,k_{n}}  \rangle
           \\&+ O ( |b_{0,n}|\mu_{1,k_{n}}^{\frac{N(p+1)}{q+1} +\alpha(N+1-(N-2)p)} + \sum_{i=1}^{N}|b_{i,n}|\mu_{1,k_{n}}^{ \frac{N(p+1)}{q+1}+1+\alpha (N-(N-2)p)  }),
    \end{split}
\end{equation*}
and
\begin{equation*}
    \begin{split}
         I_{1,n}^{2}=&\int_{  B_{\mu_{1,k_{n}}^{\alpha}}(P_{1,k_{n}}) }-\Delta PV_{1,k_{n}}\langle    \nabla\sum_{i=0}^{N}b_{i,n}\mu_{1,k_{n}} \widetilde{P\Psi}_{1,k_{n},i},x-P_{1,k_{n}}  \rangle
        \\=&\int_{  B_{\mu_{1,k_{n}}^{\alpha}}(P_{1,k_{n}}) }U_{1,k_{n}}^{q}\langle    \nabla\sum_{i=0}^{N}b_{i,n}\mu_{1,k_{n}}\widetilde{P\Psi}_{1,k_{n},i},x-P_{1,k_{n}}  \rangle
        \\=&\int_{  B_{\mu_{1,k_{n}}^{\alpha}}(P_{1,k_{n}}) }U_{1,k_{n}}^{q}\langle    \nabla b_{0,n}\mu_{1,k_{n}} \Psi_{1,k_{n},i},x-P_{1,k_{n}}  \rangle
        \\&+O(  |b_{0,n}| \mu_{1,k_{n}}^{\frac{N(p+1)}{q+1} +1}+ \sum_{i=1}^{N}|b_{i,n}| \mu_{1,k_{n}}^{\frac{N(p+1)}{q+1} +2}).
    \end{split}
\end{equation*}
Second, we compute $ I_{1,n}^{3}$ and $ I_{1,n}^{4}$. We have
\begin{equation*}
    \begin{split}
         I_{1,n}^{3}=&\int_{  B_{\mu_{1,k_{n}}^{\alpha}}(P_{1,k_{n}}) }-\Delta  \sum_{i=0}^{N}b_{i,n}\mu_{1,k_{n}} \widetilde{P\Psi}_{1,k_{n},i} \langle \nabla PV_{1,k_{n}} ,x-P_{1,k_{n}}  \rangle
        \\=& \int_{  B_{\mu_{1,k_{n}}^{\alpha}}(P_{1,k_{n}}) }\sum_{i=0}^{N}b_{i,n}(PV)_{1,k_{n}}^{p-1}\mu_{1,k_{n}} P\Phi_{1,k_{n},i} \langle \nabla V_{1,k_{n}} ,x-P_{1,k_{n}}  \rangle
        \\=&\int_{  B_{\mu_{1,k_{n}}^{\alpha}}(P_{1,k_{n}}) }pb_{0,n}V_{1,k_{n}}^{p-1}\mu_{1,k_{n}} \Phi_{1,k_{n},0} \langle \nabla V_{1,k_{n}} ,x-P_{1,k_{n}}  \rangle
        \\&+O(  |b_{0,n}|\mu_{1,k_{n}}^{\frac{N(p+1)}{q+1} +\alpha(N-(N-2)p+1)} + \sum_{i=1}^{N}|b_{i,n}|\mu_{1,k_{n}}^{ \frac{N(p+1)}{q+1}+1+\alpha(N-(N-2)p)   }   ).
    \end{split}
\end{equation*}
Using Lemma \ref{ble1}, we have
\begin{align*}
    I_{1,n}^{4}=&\int_{  B_{\mu_{1,k_{n}}^{\alpha}}(P_{1,k_{n}}) }-\Delta  \sum_{i=0}^{N}b_{i,n}\mu_{1,k_{n}} P\Phi_{1,k_{n},i} \langle \nabla PU_{d_{1,k_{n}},P_{1,k_{n}}} ,x-P_{1,k_{n}}  \rangle
        \\=&\int_{  B_{\mu_{1,k_{n}}^{\alpha}}(P_{1,k_{n}}) }\sum_{i=0}^{N}b_{i,n}qU_{1,k_{n}}^{q-1}\mu_{1,k_{n}} \Psi_{1,k_{n},i} \langle \nabla PU_{d_{1,k_{n}},P_{1,k_{n}}} ,x-P_{1,k_{n}}  \rangle
        \\=&\int_{  B_{\mu_{1,k_{n}}^{\alpha}}(P_{1,k_{n}}) }b_{0,n}qU_{1,k_{n}}^{q-1}\mu_{1,k_{n}} \Psi_{1,k_{n},0} \langle \nabla U_{1,k_{n}} ,x-P_{1,k_{n}}  \rangle+
        \\+&O(  |b_{0,n}| \mu_{1,k_{n}}^{\frac{N(p+1)}{q+1} +1})+o( \sum_{i=1}^{N}|b_{i,n}| \mu_{1,k_{n}}^{\frac{N(p+1)}{q+1} +1}),
\end{align*}
and
\begin{align*}
     I_{1,n}^{5}=&\frac{N}{q+1}\int_{  B_{\mu_{1,k_{n}}^{\alpha}}(P_{1,k_{n}}) }-\Delta  \sum_{i=0}^{N}b_{i,n}\mu_{1,k_{n}} P\Phi_{1,k_{n},i} PU_{d_{1,k_{n}},P_{1,k_{n}}}
        \\=&\frac{N}{q+1}\int_{  B_{\mu_{1,k_{n}}^{\alpha}}(P_{1,k_{n}}) }\sum_{i=0}^{N}b_{i,n}qU_{1,k_{n}}^{q-1}\mu_{1,k_{n}} \Psi_{1,k_{n},i}PU_{d_{1,k_{n}},P_{1,k_{n}}}
        \\=&\frac{N}{q+1}\int_{  B_{\mu_{1,k_{n}}^{\alpha}}(P_{1,k_{n}}) }b_{0,n}qU_{1,k_{n}}^{q-1}\mu_{1,k_{n}} \Psi_{1,k_{n},0}U_{1,k_{n}}
       \\&-\frac{N}{q+1}b_{0,n}(\frac{b_{N,p}}{\gamma_{N}})^{p}\mu_{1,k_{n}}^{\frac{Np}{q+1}}\tilde{H}( P_{1,k_{n}},P_{1,k_{n}}) \int_{  B_{\mu_{1,k_{n}}^{\alpha}}(P_{1,k_{n}}) }qU_{1,k_{n}}^{q-1}\mu_{1,k_{n}} \Psi_{1,k_{n},0} \\&+ o(|b_{0,n}| \mu_{1,k_{n}}^{\frac{N(p+1)}{q+1}})+ o(  \sum_{i=1}^{N}|b_{i,n}|\mu_{1,k_{n}}^{\frac{N(p+1)}{q+1}+1}  ).
\end{align*}
Moreover, we have
\begin{align*}
    I_{1,n}^{6}=&\frac{N}{q+1}\int_{  B_{\mu_{1,k_{n}}^{\alpha}}(P_{1,k_{n}}) }-\Delta PV_{1,k_{n}} \sum_{i=0}^{N}b_{i,n}\mu_{1,k_{n}} \widetilde{P\Psi}_{1,k_{n},i}
        \\=&\frac{N}{q+1}\int_{  B_{\mu_{1,k_{n}}^{\alpha}}(P_{1,k_{n}}) }U_{1,k_{n}}^{q}\sum_{i=0}^{N}b_{i,n}\mu_{1,k_{n}} \widetilde{P\Psi}_{1,k_{n},i}
        \\=&\frac{N}{q+1}\int_{  B_{\mu_{1,k_{n}}^{\alpha}}(P_{1,k_{n}}) }U_{1,k_{n}}^{q}b_{0,n}\mu_{1,k_{n}} \Psi_{1,k_{n},0}
        \\&-\frac{Np}{q+1}(\frac{ b_{N,p}b_{0,n}}{\gamma_{N}})^{p} \mu_{1,k_{n}}^{ \frac{Np}{q+1}} \tilde{H}(  P_{1,k_{n}},P_{1,k_{n}} )\int_{  B_{\mu_{1,k_{n}}^{\alpha}}(P_{1,k_{n}}) } U_{1,k_{n}}^{q}
        \\&+o( |b_{0,n}|\mu_{1,k_{n}}^{\frac{N(p+1)}{q+1}}   )+o(  \sum_{i=1}^{N}|b_{i,n}|\mu_{1,k_{n}}^{\frac{N(p+1)}{q+1}+1}  ).
\end{align*}
Forth, we calculate $I_{1,n}^{7}$ and $I_{1,n}^{8}$.
\begin{equation*}
    \begin{split}
        I_{1,n}^{7}=&\frac{N}{p+1}\int_{  B_{\mu_{1,k_{n}}^{\alpha}}(P_{1,k_{n}}) }-\Delta PU_{d_{1,k_{n}},P_{1,k_{n}}} \sum_{i=0}^{N}b_{i,n}\mu_{1,k_{n}} P\Phi_{1,k_{n},i}
        \\=&\frac{N}{p+1}\int_{  B_{\mu_{1,k_{n}}^{\alpha}}(P_{1,k_{n}}) }PV_{1,k_{n}}^{p} \sum_{i=0}^{N}b_{i,n}\mu_{1,k_{n}} P\Phi_{1,k_{n},i}
        \\=&\frac{N}{p+1}\int_{  B_{\mu_{1,k_{n}}^{\alpha}}} V_{1,k_{n}}^{p}b_{0,n}\mu_{1,k_{n}} \Phi_{1,k_{n},0}
        \\&+O(   |b_{0,n}|\mu_{1,k_{n}}^{ \frac{N(p+1)}{q+1} +\alpha(N-(N-2)p)}  +\sum_{i=1}^{N}|b_{i,n}|\mu_{1,k_{n}}^{ \frac{N(p+1)}{q+1}+1 +\alpha(N-(N-2)p-1)} ),
    \end{split}
\end{equation*}
and
\begin{equation*}
    \begin{split}
        I_{1,n}^{8}=&\frac{N}{p+1}\int_{  B_{\mu_{1,k_{n}}^{\alpha}}(P_{1,k_{n}}) }-\Delta \sum_{i=0}^{N}b_{i,n}\mu_{1,k_{n}} \widetilde{P\Psi}_{1,k_{n},i} PV_{1,k_{n}}
        \\=&\frac{N}{p+1}\int_{  B_{\mu_{1,k_{n}}^{\alpha}}(P_{1,k_{n}}) }\sum_{i=0}^{N}p(PV)_{1,k_{n}}^{p-1}b_{i,n}\mu_{1,k_{n}} P\Psi_{1,k_{n},i}PV_{1,k_{n}}
        \\=&\frac{N}{p+1}\int_{  B_{\mu_{1,k_{n}}^{\alpha}}(P_{1,k_{n}}) }b_{0,n} pV_{1,k_{n}}^{p}\mu_{1,k_{n}} \Psi_{1,k_{n},i}
        \\&+O( |b_{0,n}|\mu_{1,k_{n}}^{ \frac{N(p+1)}{q+1} +\alpha(N-(N-2)p)}+\sum_{i=1}^{N}|b_{i,n}|\mu_{1,k_{n}}^{ \frac{N(p+1)}{q+1}+1 +\alpha(N-(N-2)p-1)}  ) .
    \end{split}
\end{equation*}

Thus, we have
\begin{equation}\label{eq4.2}
    \begin{split}
       &I_{1,n}( PU_{d_{1,k_{n}},P_{1,k_{n}}},PV_{1,k_{n}},\sum_{i=0}^{N}b_{i,n}\mu_{1,k_{n}} \widetilde{P\Psi}_{1,k_{n},i}, \sum_{i=0}^{N}b_{i,n}\mu_{1,k_{n}} P\Phi_{1,k_{n},i} )
       \\=&\int_{  B_{\mu_{1,k_{n}}^{\alpha}}(P_{1,k_{n}}) }V_{1,k_{n}}^{p}\langle    \nabla b_{0,n}\mu_{1,k_{n}} \Phi_{1,k_{n},0} ,x-P_{1,k_{n}}  \rangle+\int_{  B_{\mu_{1,k_{n}}^{\alpha}}(P_{1,k_{n}}) }U_{1,k_{n}}^{q}\langle    \nabla b_{0,n}\mu_{1,k_{n}} \Psi_{1,k_{n},i},x-P_{1,k_{n}}  \rangle
       \\&+\int_{  B_{\mu_{1,k_{n}}^{\alpha}}(P_{1,k_{n}}) }pb_{0,n}V_{1,k_{n}}^{p-1}\mu_{1,k_{n}} \Phi_{1,k_{n},0} \langle \nabla V_{1,k_{n}} ,x-P_{1,k_{n}}  \rangle
       \\&+\int_{  B_{\mu_{1,k_{n}}^{\alpha}}(P_{1,k_{n}}) }b_{0,n}qU_{1,k_{n}}^{q-1}\mu_{1,k_{n}} \Psi_{1,k_{n},0} \langle \nabla U_{1,k_{n}} ,x-P_{1,k_{n}}  \rangle
       \\&+\frac{N}{q+1}\int_{  B_{\mu_{1,k_{n}}^{\alpha}}(P_{1,k_{n}}) }b_{0,n}qU_{1,k_{n}}^{q-1}\mu_{1,k_{n}} \Psi_{1,k_{n},0}U_{1,k_{n}}
       +\frac{N}{q+1}\int_{  B_{\mu_{1,k_{n}}^{\alpha}}(P_{1,k_{n}}) }U_{1,k_{n}}^{q}b_{0,n}\mu_{1,k_{n}} \Psi_{1,k_{n},0}
       \\&-\frac{N}{q+1}b_{0,n}(\frac{b_{N,p}}{\gamma_{N}})^{p}\mu_{1,k_{n}}^{\frac{Np}{q+1}}\tilde{H}( P_{1,k_{n}},P_{1,k_{n}}) \int_{  B_{\mu_{1,k_{n}}^{\alpha}}(P_{1,k_{n}}) }qU_{1,k_{n}}^{q-1}\mu_{1,k_{n}} \Psi_{1,k_{n},0}
        \\&+\frac{N}{q+1}a_{N,p}b_{0,n}\mu_{1,k_{n}}^{ \frac{Np}{q+1}} \tilde{H}(  P_{1,k_{n}},P_{1,k_{n}} )\int_{  B_{\mu_{1,k_{n}}^{\alpha}}(P_{1,k_{n}}) } U_{1,k_{n}}^{q}
        \\&+\frac{N}{p+1}\int_{  B_{\mu_{1,k_{n}}^{\alpha}}} V_{1,k_{n}}^{p}b_{0,n}\mu_{1,k_{n}} \Phi_{1,k_{n},0}+\frac{N}{p+1}\int_{  B_{\mu_{1,k_{n}}^{\alpha}}(P_{1,k_{n}}) }b_{0,n} pV_{1,k_{n}}^{p}\mu_{1,k_{n}} \Psi_{1,k_{n},i}
        \\&+o(|b_{0,n}|\mu_{1,k_{n}}^{\frac{N(p+1)}{q+1}})
        +o(\sum_{i=1}^{N}|b_{i,n}|\mu_{1,k_{n}}^{\frac{N(p+1)}{q+1}+1}  ).
    \end{split}
\end{equation}

Using Lemma  \ref{le4.3} and \ref{le4.5}, we get that
\begin{equation}\label{eq4.3}
    \begin{split}
        H_{n} = o(|b_{0,n}|\mu_{1,k_{n}}^{\frac{N(p+1)}{q+1}})
        +o(\sum_{i=1}^{N}|b_{i,n}|\mu_{1,k_{n}}^{\frac{N(p+1)}{q+1}+1}  ).
    \end{split}
\end{equation}

 Next, we compute the right hand side of \eqref{eq4.1}. Through direct computations, we have
\begin{equation*}
    \begin{split}
       &(-N+\frac{2N}{q+1})\epsilon\beta_2\int_{B_{\mu_{1,k_{n}}^{\alpha}}(P_{1,k_{n}})}u_{k_{n}}\eta_{n} dx
       =(-N+\frac{2N}{q+1})\epsilon\beta_2\int_{B_{\mu_{1,k_{n}}^{\alpha}}}U_{1,k_{n}}\mu_{1,k_{n}}b_{0,n}\Psi_{1,k_{n},0}
       \\&+O( |b_{0,n}| \mu_{1,k_{n}}^{\frac{N(3p-q+2)}{q+1} +\alpha(N-\frac{N(p+1)}{q+1})    } + \sum_{i=1}^{N}|b_{i,n}|\mu_{1,k_{n}}^{\frac{N(3p-q+2)}{q+1} +1+\alpha(N-\frac{N(p+1)}{q+1}-1)    }
       \\&+   |b_{0,n}| \mu_{1,k_{n}}^{\frac{N3(p+1)}{2(q+1)} -\theta    } + \sum_{i=1}^{N}|b_{i,n}|\mu_{1,k_{n}}^{\frac{N3(p+1)}{2(q+1)}+1 -\theta    }),
    \end{split}
\end{equation*}
and
\begin{equation*}
    \begin{split}
        &\epsilon\beta_2\int_{\partial B_{\mu_{1,k_{n}}^{\alpha}}(P_{1,k_{n}})}u_{k_{n}}\eta_{n}\langle\nu,x-P_{1,k_{n}} \rangle ds
        \\=&O(  |b_{0,n}| \mu_{1,k_{n}}^{\frac{N(3p+2-q)}{q+1}   +\alpha(N-\frac{2N(p+1)}{q+1}) } + \sum_{i=1}^{N}|b_{i,n}|  \mu_{1,k_{n}}^{\frac{N(3p+2-q)}{q+1}   +\alpha(N-\frac{2N(p+1)}{q+1}) } ).
    \end{split}
\end{equation*}
We calculate the second term of the right-hand term in \eqref{eq4.1}.
\begin{equation*}
\begin{split}
    &\int_{\partial B_{\mu_{1,k_{n}}^{\alpha}}(P_{1,k_{n}})}u_{k_{n}}^{q}\eta_{n}\langle \nu,x-P_{1,k_{n}}\rangle
    \\=&\int_{\partial B_{\mu_{1,k_{n}}^{\alpha}}(P_{1,k_{n}})}U_{1,k_{n}}^{q}b_{0,n}\mu_{1,k_{n}}\psi_{k_{n}}\langle \nu,x-P_{1,k_{n}}\rangle
    \\&+O( |b_{0,n}|\mu_{1,k_{n}}^{ NP+\alpha(N-\frac{N(p+1)q}{q+1})} + \sum_{i=1}^{N}|b_{i,n}| \mu_{1,k_{n}}^{ NP+1+\alpha(N-\frac{N(p+1)q}{q+1})}
    )
    \\=&\int_{\partial B_{\mu_{1,k_{n}}^{\alpha}}(P_{1,k_{n}})}U_{1,k_{n}}^{q}b_{0,n}\mu_{1,k_{n}}\psi_{k_{n}}\langle \nu,x-P_{1,k_{n}}\rangle
    +o(|b_{0,n}|\mu_{1,k_{n}}^{\frac{N(p+1)}{q+1}})
        +o(\sum_{i=1}^{N}|b_{i,n}|\mu_{1,k_{n}}^{\frac{N(p+1)}{q+1}+1}  ).
\end{split}
\end{equation*}
Next, we calculate the third term of the right-hand term in \eqref{eq4.1}.
\begin{equation*}
    \begin{split}
         &\int_{\partial B_{\mu_{1,k_{n}}^{\alpha}}(P_{1,k_{n}})}v_{k_{n}}^{p}\xi_{n}\langle \nu,x-P_{1,k_{n}}\rangle
         =\int_{\partial B_{\mu_{1,k_{n}}^{\alpha}}(P_{1,k_{n}})}V_{1,k_{n}}^{p}b_{0,n}\mu_{1,k_{n}}\phi_{k_{n}}\langle \nu,x-P_{1,k_{n}}\rangle
         \\&+O(|b_{0,n}|\mu_{1,k_{n}}^{\frac{N(p+1)}{q+1} +\alpha(N-(N-2)p)  }  +  \sum_{i=1}^{N}|b_{i,n}|  \mu_{1,k_{n}}^{\frac{N(p+1)}{q+1}+1 +\alpha(N-(N-2)p-1)  }     )
         \\&+O( |b_{0,n}| \mu_{1,k_{n}}^{\frac{N(p+1)}{q+1}+\frac{N(qp-1)}{q+1} +\alpha(N-(N-2)(p+\frac{1}{2} )  -1-\theta)    }
         \\&+ \sum_{i=1}^{N}|b_{i,n}| \mu_{1,k_{n}}^{\frac{N(p+1)}{q+1}+1+\frac{N(qp-1)}{q+1}
        +\alpha(N-(N-2)(p+\frac{1}{2} )  -1-\theta)    }).
    \end{split}
\end{equation*}
Thus
\begin{equation}\label{eq4.4}
    \begin{split}
        &\quad RHS \,\,of \,\, Eq.4.48
        =(-N+\frac{2N}{q+1})\epsilon\beta_2\int_{B_{\mu_{1,k_{n}}^{\alpha}}}U_{1,k_{n}}\mu_{1,k_{n}}b_{0,n}\Psi_{1,k_{n},0}
\\&+\int_{\partial B_{\mu_{1,k_{n}}^{\alpha}}(P_{1,k_{n}})}U_{1,k_{n}}^{q}b_{0,n}\mu_{1,k_{n}}\psi_{k_{n}}\langle \nu,x-P_{1,k_{n}}\rangle
        +\int_{\partial B_{\mu_{1,k_{n}}^{\alpha}}(P_{1,k_{n}})}V_{1,k_{n}}^{p}b_{0,n}\mu_{1,k_{n}}\phi_{k_{n}}\langle \nu,x-P_{1,k_{n}}\rangle
        \\&+o(|b_{0,n}|\mu_{1,k_{n}}^{\frac{N(p+1)}{q+1}})
        +o(\sum_{i=1}^{N}|b_{i,n}|\mu_{1,k_{n}}^{\frac{N(p+1)}{q+1}+1}  ).
    \end{split}
\end{equation}
By careful computation, we see that
\begin{equation*}
    \begin{split}
        &\int_{  B_{\mu_{1,k_{n}}^{\alpha}}(P_{1,k_{n}}) }V_{1,k_{n}}^{p}\langle    \nabla b_{0,n}\mu_{1,k_{n}} \Phi_{1,k_{n},0} ,x-P_{1,k_{n}}  \rangle
        \\&+\int_{  B_{\mu_{1,k_{n}}^{\alpha}}(P_{1,k_{n}}) }pb_{0,n}V_{1,k_{n}}^{p-1}\mu_{1,k_{n}} \Phi_{1,k_{n},0} \langle \nabla V_{1,k_{n}} ,x-P_{1,k_{n}}  \rangle
        \\&+\int_{  B_{\mu_{1,k_{n}}^{\alpha}}(P_{1,k_{n}}) }U_{1,k_{n}}^{q}\langle    \nabla b_{0,n}\mu_{1,k_{n}} \Psi_{1,k_{n},i},x-P_{1,k_{n}}  \rangle
        \\&+\int_{  B_{\mu_{1,k_{n}}^{\alpha}}(P_{1,k_{n}}) }b_{0,n}qU_{1,k_{n}}^{q-1}\mu_{1,k_{n}} \Psi_{1,k_{n},0} \langle \nabla U_{1,k_{n}} ,x-P_{1,k_{n}}  \rangle
        \\&+\frac{N}{q+1}\int_{  B_{\mu_{1,k_{n}}^{\alpha}}(P_{1,k_{n}}) }b_{0,n}qU_{1,k_{n}}^{q-1}\mu_{1,k_{n}} \Psi_{1,k_{n},0}U_{1,k_{n}}+\frac{N}{q+1}\int_{  B_{\mu_{1,k_{n}}^{\alpha}}(P_{1,k_{n}}) }U_{1,k_{n}}^{q}b_{0,n}\mu_{1,k_{n}} \Psi_{1,k_{n},0}
        \\&+\frac{N}{p+1}\int_{  B_{\mu_{1,k_{n}}^{\alpha}}} V_{1,k_{n}}^{p}b_{0,n}\mu_{1,k_{n}} \Phi_{1,k_{n},0}+\frac{N}{p+1}\int_{  B_{\mu_{1,k_{n}}^{\alpha}}(P_{1,k_{n}}) }b_{0,n} pV_{1,k_{n}}^{p}\mu_{1,k_{n}} \Phi_{1,k_{n},0}
        \\&+\int_{\partial B_{\mu_{1,k_{n}}^{\alpha}}(P_{1,k_{n}})}U_{1,k_{n}}^{q}b_{0,n}\mu_{1,k_{n}}\Psi_{1,k_{n},0}\langle \nu,x-P_{1,k_{n}}\rangle
        \\&+\int_{\partial B_{\mu_{1,k_{n}}^{\alpha}}(P_{1,k_{n}})}V_{1,k_{n}}^{p}b_{0,n}\mu_{1,k_{n}}\Phi_{1,k_{n},0}\langle \nu,x-P_{1,k_{n}}\rangle
    =0.
    \end{split}
\end{equation*}

Thus, from \eqref{eq4.1}, \eqref{eq4.2}, \eqref{eq4.3} and \eqref{eq4.4}, we have
\begin{equation*}
    \begin{split}
        &-\frac{N}{q+1}b_{0,n}(\frac{b_{N,p}}{\gamma_{N}})^{p}\mu_{1,k_{n}}^{\frac{Np}{q+1}}\tilde{H}( P_{1,k_{n}},P_{1,k_{n}}) \int_{  B_{\mu_{1,k_{n}}^{\alpha}}(P_{1,k_{n}}) }qU_{1,k_{n}}^{q-1}\mu_{1,k_{n}} \Psi_{1,k_{n},0}
        \\&-\frac{Np}{q+1}(\frac{ b_{N,p}b_{0,n}}{\gamma_{N}})^{p} \mu_{1,k_{n}}^{ \frac{Np}{q+1}} \tilde{H}(  P_{1,k_{n}},P_{1,k_{n}} )\int_{  B_{\mu_{1,k_{n}}^{\alpha}}(P_{1,k_{n}}) } U_{1,k_{n}}^{q}
        \\&-(-N+\frac{2N}{q+1})\epsilon\beta_2\int_{B_{\mu_{1,k_{n}}^{\alpha}}}U_{1,k_{n}}\mu_{1,k_{n}}b_{0,n}\Psi_{1,k_{n},0}
        \\=&o(|b_{0,n}|\mu_{1,k_{n}}^{\frac{N(p+1)}{q+1}})
        +o(\sum_{i=1}^{N}|b_{i,n}|\mu_{1,k_{n}}^{\frac{N(p+1)}{q+1}+1}  ).
    \end{split}
\end{equation*}
From equation (6.1) of \cite{Kim-Pis}, we have
\begin{equation*}
    \begin{split}
        \tilde{H}( P_{1,k_{n}},P_{1,k_{n}}){\int_{\mathbb{R}^{N}}U_{1,0}^{q}}
        \frac{N}{q+1}(\frac{b_{N,p}}{\gamma_{N}})^{p}d_{1,k_{n}}^{\frac{N(p+1)}{q+1}-1}-\frac{\beta_2N(q-1)\int_{\mathbb{R}^{N}}U_{1,0}^{2}}{2(q+1)}d_{1,k_{n}}^{\frac{N(q-1)}{q+1}-1} = o_{\mu_{1,k_{n}}}(1).
    \end{split}
\end{equation*}
Thus we have
\begin{align*}
    (( q-1)(p+1)-(q-1)(1-q) +o(1) )b_{0,n}\mu^{\frac{N(p+1)}{q+1}} = o(\mu^{\frac{N(p+1)}{q+1}+1} ).
\end{align*}
So $b_{0,n} = o(\mu)$.

\textbf{Step 2.}  We prove $b_{i,n} \to 0$ for $i=1,\cdots,N$. We apply the identities in Lemma \ref{le2.1} in the ball $B_{\mu_{1,k_{n}}^{\alpha}}(P_{1,k_{n}}) $, where $\alpha>0$ is the same constant as in \textbf{Step 1}.  We have
\begin{equation}\label{eq4.5}
    \begin{split}
         &\int_{\partial B_{\mu_{1,k_{n}}^{\alpha}}(P_{1,k_{n}})}(-\frac{\partial u_{k_{n}}}{\partial\nu}\frac{\partial\xi_{n}}{\partial x_i} -\frac{\partial \xi_{n}}{\partial\nu}\frac{\partial u_{k_{n}}}{\partial x_i} +\frac{\partial u_{k_{n}}}{\partial\nu}\frac{\partial\xi_{n}}{\partial\nu}\nu_i)ds
         \\&+\int_{\partial B_{\mu_{1,k_{n}}^{\alpha}}(P_{1,k_{n}})}(-\frac{\partial v_{k_{n}}}{\partial\nu}\frac{\partial\eta_{n}}{\partial  x_i} -\frac{\partial \eta_{n}}{\partial\nu}\frac{\partial v_{k_{n}}}{\partial  x_i} +\frac{\partial v_{k_{n}}}{\partial\nu}\frac{\partial\eta_{n}}{\partial\nu}\nu_i)ds
    \\=&\int_{\partial B_{\mu_{1,k_{n}}^{\alpha}}(P_{1,k_{n}})}(u_{k_{n}}^{q}\eta_{n}+v_{k_{n}}^{p}\xi_{n})\nu_ids+\epsilon\beta_2\int_{\partial B_{\mu_{1,k_{n}}^{\alpha}}(P_{1,k_{n}})}u_{k_{n}}\eta_{n}\nu_i ds.
    \end{split}
\end{equation}
We compute that
\begin{equation*}
    \begin{split}
       &\int_{\partial B_{\mu_{1,k_{n}}^{\alpha}}(P_{1,k_{n}})}(-\frac{\partial u_{k_{n}}}{\partial\nu}\frac{\partial\xi_{n}}{\partial x_i} -\frac{\partial \xi_{n}}{\partial\nu}\frac{\partial u_{k_{n}}}{\partial x_i} +\frac{\partial u_{k_{n}}}{\partial\nu}\frac{\partial\xi_{n}}{\partial\nu}\nu_i)ds
         \\&+\int_{\partial B_{\mu_{1,k_{n}}^{\alpha}}(P_{1,k_{n}})}(-\frac{\partial v_{k_{n}}}{\partial\nu}\frac{\partial\eta_{n}}{\partial  x_i} -\frac{\partial \eta_{n}}{\partial\nu}\frac{\partial v_{k_{n}}}{\partial  x_i} +\frac{\partial v_{k_{n}}}{\partial\nu}\frac{\partial\eta_{n}}{\partial\nu}\nu_i)ds
         \\=&\int_{ B_{\mu_{1,k_{n}}^{\alpha}}(P_{1,k_{n}})}-\Delta u_{k_{n}} \frac{\partial \xi_{n}}{\partial x_i}-\Delta \xi_{n}\frac{\partial u_{k_{n}}}{\partial x_i}dy+\int_{ B_{\mu_{1,k_{n}}^{\alpha}}(P_{1,k_{n}})}-\Delta v_{k_{n}} \frac{\partial \eta_{n}}{\partial x_i}-\Delta \eta_{n}\frac{\partial v_{k_{n}}}{\partial x_i}dy.
    \end{split}
\end{equation*}
Set
\begin{equation*}
    \begin{split}
        I_{2,n}(U,V,\eta,\xi)
        = \int_{ B_{\mu_{1,k_{n}}^{\alpha}}(P_{1,k_{n}})}-\Delta U \frac{\partial \xi}{\partial x_i}-\Delta \xi \frac{\partial U}{\partial x_i}dy+\int_{ B_{\mu_{1,k_{n}}^{\alpha}}(P_{1,k_{n}})}-\Delta V \frac{\partial \eta}{\partial x_i}-\Delta \eta \frac{\partial V}{\partial x_i}dy.
    \end{split}
\end{equation*}
We have
\begin{equation*}
    \begin{split}
      &\quad I_{2,n}(U_{1,k_{n}},V_{1,k_{n}},\eta_{n},\xi_{n})
      \\= & \int_{ B_{\mu_{1,k_{n}}^{\alpha}}(P_{1,k_{n}})}-\Delta PU_{d_{1,k_{n}},P_{1,k_{n}}} \frac{\partial \sum_{j=0}^{N}b_{j,n}\mu_{1,k_{n}} P\Phi_{1,k_{n},j}}{\partial x_i}
      \\&+ \int_{ B_{\mu_{1,k_{n}}^{\alpha}}(P_{1,k_{n}})}-\Delta PV_{1,k_{n}} \frac{\partial \sum_{j=0}^{N}b_{j,n}\mu_{1,k_{n}} \widetilde{P\Psi}_{1,k_{n},j}}{\partial x_i}
      +\int_{ B_{\mu_{1,k_{n}}^{\alpha}}(P_{1,k_{n}})}-\Delta \sum_{j=0}^{N}b_{j,n}\mu_{1,k_{n}} \widetilde{P\Psi}_{1,k_{n},j}  \frac{\partial  PV_{1,k_{n}}}{\partial x_i}
      \\&+ \int_{ B_{\mu_{1,k_{n}}^{\alpha}}(P_{1,k_{n}})}-\Delta\sum_{j=0}^{N}b_{j,n}\mu_{1,k_{n}} P\Phi_{1,k_{n},j}\frac{\partial  PU_{d_{1,k_{n}},P_{1,k_{n}}}}{\partial x_i}+ G_{n}
      \\=:&I_{2,n}^{1}+I_{2,n}^{2}+I_{2,n}^{3}+I_{2,n}^{4}+G_n,
    \end{split}
\end{equation*}
where $G_{n}$ is the remainder term.


By directly computing, we get that
\begin{equation*}
    \begin{split}
        I_{2,n}^{1}+I_{2,n}^{3}=&\int_{ B_{\mu_{1,k_{n}}^{\alpha}}(P_{1,k_{n}})} V_{1,k_{n}}^{p}b_{i,n}\frac{\partial \mu_{1,k_{n}} \Phi_{1,k_{n},j}}{\partial x_i}+\int_{ B_{\mu_{1,k_{n}}^{\alpha}}(P_{1,k_{n}})} pV_{1,k_{n}}^{p-1}b_{i,n}\mu_{1,k_{n}} \Phi_{1,k_{n},i}\frac{\partial  V_{1,k_{n}}}{\partial x_i}
       \\&+O(b_{0,n}\mu_{1,k_{n}}^{\frac{N(p+1)}{q+1}}) + o(  \sum_{j=1}^{N}b_{j,n}\mu_{1,k_{n}}^{\frac{N(p+1)}{q+1}+1} ),
    \end{split}
\end{equation*}
Applying Lemma \ref{ble5}, we have
\begin{equation*}
    \begin{split}
        I_{2,n}^{2}+I_{2,n}^{4}=&-\frac{\mu_{1,k_n}^{\frac{N(p+1)}{q+1}+1}}{p+1}(\frac{b_{N,p}}{\gamma_{N}})^{p}\int_{\mathbb{R}^{N}}U_{0,1}^{q}\sum_{j=1}^{N}b_{j,n}\frac{\partial^{2}\tau }{\partial x_{i}\partial x_{j}}(P_{0}) + o(\mu_{1,k_n}^{\frac{N(p+1)}{q+1}+1} ).
    \end{split}
\end{equation*}
Thus, we have
\begin{equation}\label{eq4.6}
\begin{split}
  &I_{2,n}( PU_{d_{1,k_{n}},P_{1,k_{n}}},PV_{1,k_{n}},\sum_{j=0}^{N}b_{j,n}\mu_{1,k_{n}} \widetilde{P\Psi}_{1,k_{n},j}, \sum_{j=0}^{N}b_{j,n}\mu_{1,k_{n}} P\Phi_{1,k_{n},j} )
  \\=&\int_{ B_{\mu_{1,k_{n}}^{\alpha}}(P_{1,k_{n}})} V_{1,k_{n}}^{p}b_{i,n}\frac{\partial \mu_{1,k_{n}} \Phi_{1,k_{n},j}}{\partial x_i}
 \\& -p\mu_{1,k_{n}}^{\frac{N}{q+1}}\frac{b_{N,p}}{\gamma_{N}}H(P_{1,k_{n}},P_{1,k_{n}} )\int_{ B_{\mu_{1,k_{n}}^{\alpha}}(P_{1,k_{n}})}V_{1,k_{n}}^{p-1}b_{i,n}\frac{\partial \mu_{1,k_{n}} \Phi_{1,k_{n},j}}{\partial x_i}
  \\&+\int_{ B_{\mu_{1,k_{n}}^{\alpha}}(P_{1,k_{n}})} U_{1,k_{n}}^{q}b_{i,n}\frac{\partial \mu_{1,k_{n}} \Psi_{1,k_{n},j}}{\partial x_i}+\int_{ B_{\mu_{1,k_{n}}^{\alpha}}(P_{1,k_{n}})} pV_{1,k_{n}}^{p-1}b_{i,n}\mu_{1,k_{n}} \Phi_{1,k_{n},i}\frac{\partial  V_{1,k_{n}}}{\partial x_i}
  \\&+\int_{ B_{\mu_{1,k_{n}}^{\alpha}}(P_{1,k_{n}})}qU_{1,k_{n}}^{q-1}b_{i,n}\mu_{1,k_{n}} \Psi_{1,k_{n},j}\frac{\partial  U_{1,k_{n}}}{\partial x_i}
        \\&-\frac{\mu_{1,k_n}^{\frac{N(p+1)}{q+1}+1}}{p+1}(\frac{b_{N,p}}{\gamma_{N}})^{p}\int_{\mathbb{R}^{N}}U_{0,1}^{q}\sum_{j=1}^{N}b_{j,n}\frac{\partial^{2}\tau }{\partial x_{i}\partial x_{j}}(P_{0})
        +O(b_{0,n}| \mu_{1,k_{n}}^{\frac{N(p+1)}{q+1}  } )
        \\&+o(\sum_{j=1}^{N}|b_{j,n}| \mu_{1,k_{n}}^{\frac{N(p+1)}{q+1}+1}  ).
\end{split}
\end{equation}

It is easy to check that
\begin{equation*}
\begin{split}\label{eq4.7}
G_{n} =O(|b_{0,n}| \mu_{1,k_{n}}^{\frac{N(p+1)}{q+1}   }) + o(\sum_{j=1}^{N}|b_{j,n}|\mu_{1,k_{n}}^{\frac{N(p+1)}{q+1}+1   } ).
\end{split}
\end{equation*}

Now we compute the right hand side of equation \eqref{eq4.5}. Direct computations shows that
\begin{equation*}
    \begin{split}
        \epsilon\beta_2\int_{\partial B_{\mu_{1,k_{n}}^{\alpha}}(P_{1,k_{n}})}u_{k_{n}}\eta_{n}\nu_i ds = O( \mu_{1,k_{n}}^{\frac{N(3p-q+2)}{q+1}-\alpha(N+6)}),
    \end{split}
\end{equation*}
\begin{equation*}
    \begin{split}
        \int_{\partial B_{\mu_{1,k_{n}}^{\alpha}}P_{1,k_{n}})}u_{k_{n}}^{q}\eta_{n}\nu_ids
        =&\int_{\partial B_{\mu_{1,k_{n}}^{\alpha}}(P_{1,k_{n}})}U_{1,k_{n}}^{q}b_{i,n}\mu_{1,k_{n}}\Psi_{1,k_{n},i}\nu_i
    \\&+O( |b_{0,n}|\mu_{1,k_{n}}^{ NP+\alpha(N-1-\frac{N(p+1)q}{q+1})} + \sum_{i=1}^{N}|b_{i,n}| \mu_{1,k_{n}}^{ NP+1+\alpha(N-1-\frac{N(p+1)q}{q+1})}
    \\&+  |b_{0,n}|\mu_{1,k_{n}}^{ NP+\alpha (-N-6)} +\sum_{i=1}^{N}|b_{i,n}| \mu_{1,k_{n}}^{ NP+1+\alpha  (-N-6)}),
    \end{split}
\end{equation*}
and
\begin{equation*}
    \begin{split}
    \int_{\partial B_{\mu_{1,k_{n}}^{\alpha}}P_{1,k_{n}})}v_{k_{n}}^{p}\xi_{n}\nu_ids
        =&\int_{\partial B_{\mu_{1,k_{n}}^{\alpha}}(P_{1,k_{n}})}V_{1,k_{n}}^{p}b_{i,n}\mu_{1,k_{n}}\Phi_{1,k_{n},i}\nu_i
         \\&+O(|b_{0,n}|\mu_{1,k_{n}}^{\frac{N(p+1)}{q+1} +\alpha(N-1-(N-2)p)  }  +  \sum_{i=1}^{N}|b_{i,n}|  \mu_{1,k_{n}}^{\frac{N(p+1)}{q+1}+1 +\alpha(N-2-(N-2)p)  }     )
         \\&+O( |b_{0,n}| \mu_{1,k_{n}}^{\frac{N(p+1)}{q+1}+\frac{N(qp-1)}{q+1} +\alpha(N-1-(N-2)(p+\frac{1}{2} )  -1-\theta)    }
         \\&+ \sum_{j=1}^{N}|b_{j,n}| \mu_{1,k_{n}}^{\frac{N(p+1)}{q+1}+1+\frac{N(qp-1)}{q+1}
        +\alpha(N-2-(N-2)(p+\frac{1}{2} )  -1-\theta)    }).
    \end{split}
\end{equation*}
Thus,
\begin{equation}\label{eq4.8}
    \begin{split}
   &RHS\, of \,\eqref{eq4.5}
    \\=&\int_{\partial B_{\mu_{1,k_{n}}^{\alpha}}(P_{1,k_{n}})}U_{1,k_{n}}^{q}b_{i,n}\mu_{1,k_{n}}\Psi_{1,k_{n},i}\nu_i+\int_{\partial B_{\mu_{1,k_{n}}^{\alpha}}(P_{1,k_{n}})}V_{1,k_{n}}^{p}b_{i,n}\mu_{1,k_{n}}\Phi_{1,k_{n},i}\nu_i
    \\&+O(|b_{0,n}| \mu_{1,k_{n}}^{\frac{N(p+1)}{q+1}   -\alpha(N+6)} + \sum_{j=1}^{N}|b_{j,n}|\mu_{1,k_{n}}^{\frac{N(p+1)}{q+1}+1   -\alpha(N+6)} ).
    \end{split}
\end{equation}

Noting that
\begin{equation*}
    \begin{split}
       & \int_{ B_{\mu_{1,k_{n}}^{\alpha}}(P_{1,k_{n}})} V_{1,k_{n}}^{p}b_{i,n}\frac{\partial \mu_{1,k_{n}} \Phi_{1,k_{n},i}}{\partial x_i}+\int_{ B_{\mu_{1,k_{n}}^{\alpha}}(P_{1,k_{n}})} pV_{1,k_{n}}^{p-1}b_{i,n}\mu_{1,k_{n}} \Phi_{1,k_{n},i}\frac{\partial  V_{1,k_{n}}}{\partial x_i}
        \\&+\int_{ B_{\mu_{1,k_{n}}^{\alpha}}(P_{1,k_{n}})} pV_{1,k_{n}}^{p-1}b_{i,n}\mu_{1,k_{n}} \Phi_{1,k_{n},i}\frac{\partial  V_{1,k_{n}}}{\partial x_i}+\int_{ B_{\mu_{1,k_{n}}^{\alpha}}(P_{1,k_{n}})}qU_{1,k_{n}}^{q-1}b_{i,n}\mu_{1,k_{n}} \Psi_{1,k_{n},j}\frac{\partial  U_{1,k_{n}}}{\partial x_i}
        \\&-\int_{\partial B_{\mu_{1,k_{n}}^{\alpha}}(P_{1,k_{n}})}U_{1,k_{n}}^{q}b_{0,n}\mu_{1,k_{n}}\Psi_{1,k_{n},i}\nu_i-\int_{\partial B_{\mu_{1,k_{n}}^{\alpha}}(P_{1,k_{n}})}V_{1,k_{n}}^{p}b_{0,n}\mu_{1,k_{n}}\Phi_{1,k_{n},i}\nu_i
        =0.
    \end{split}
\end{equation*}
So, combining \eqref{eq4.5}, \eqref{eq4.6}, \eqref{eq4.7} and \eqref{eq4.8}, we have
\begin{equation*}
    \begin{split}
    \sum_{j=1}^{N}\mu_{1,k_{n}}^{\frac{N(p+1)}{q+1}}b_{j,n}
          \frac{\partial^{2} \tau }{\partial x_{i}\partial x_{j} }(P_0)  .
    =O(|b_{0,n}| \mu_{1,k_{n}}^{\frac{N(p+1)}{q+1}   }) + o(\sum_{j=1}^{N}|b_{j,n}|\mu_{1,k_{n}}^{\frac{N(p+1)}{q+1}+1   } ).
    \end{split}
\end{equation*}
Since $b_{0,n} = o( \mu_{1,k_{n}})$, and  the matrix $\{ \frac{\partial^{2} \tau }{\partial x_{i}\partial x_{j} }(P_0)  \}$ is non-degenerate,
    then $b_{j,n} \to 0$.
\end{proof}

{\bf Proof of Theorem  1.1.} We have

\begin{equation*}
    \begin{split}
        |  \eta_{n}(x) |=&\int_{\Omega}G(x,y) p V_{1,k_{n}}^{p-1}\xi_{n}dy\\
        \leq &C||(\eta_{n},\xi_{n}  )||_{*} \int_{\Omega}\frac{1}{|x-y|^{N-2}}\frac{\mu_{1,k_{n}}^{-\frac{Np}{p+1}}}{   (1+\mu_{1,k_{n}}^{-1}| y- P_{1,k_{n}}  |)^{\frac{N}{2}+(N-2)(p-1)+\theta}}
        \\\leq &C\frac{ \mu_{1,k_{n}}^{-\frac{N}{q+1}}  }{ (1+\mu_{1,k_{n}}^{-1}| y- P_{1,k_{n}}  |)^ { \frac{N(p+1)}{2(q+1)}+\theta+\theta_{1}}   },
    \end{split}
\end{equation*}
and
\begin{equation*}
    \begin{split}
        | \xi_{n}(x)  |= &\int_{\Omega}G(x,y)\bigg( qU_{n}^{q-1}\eta_{n} +\epsilon\beta_{2}U_{1,k_{n}} \eta_{n}   \bigg)\\
        \leq &C||(\eta_{n},\xi_{n}  )||_{*} \int_{\Omega}\frac{1}{|x-y|^{N-2}}\bigg(\frac{\mu_{1,k_{n}}^{-\frac{Nq}{q+1}}}{   (1+\mu_{1,k_{n}}^{-1}| y- P_{1,k_{n}}  |)^{ \frac{N(p+1)}{q+1}(q-\frac{1}{2})+\theta}}
        \\&+ \frac{\mu_{1,k_{n}}^{\frac{Np}{q+1}-\frac{N}{p+1}}}{  (1+\mu_{1,k_{n}}^{-1}| y- P_{1,k_{n}}  |)^{ \frac{3N(p+1)}{2(q+1)}+\theta}       }            \bigg)
        \\\leq &C\frac{ \mu_{1,k_{n}}^{-\frac{N}{p+1}}  }{ (1+\mu_{1,k_{n}}^{-1}| y- P_{1,k_{n}}  |)^ { \frac{N}{2}+\theta+\theta_{1}}   }.
    \end{split}
\end{equation*}
Thus
\begin{equation*}
    \begin{split}
       &(\frac{ \mu_{1,k_{n}}^{-\frac{N}{q+1}}  }{ (1+\mu_{1,k_{n}}^{-1}| y- P_{1,k_{n}}  |)^ { \frac{N(p+1)}{2(q+1)}+\theta}   } )^{-1}|  \eta_{n}(x) |
       +\frac{ \mu_{1,k_{n}}^{-\frac{N}{p+1}}  }{ (1+\mu_{1,k_{n}}^{-1}| y- P_{1,k_{n}}  |)^ { \frac{N}{2}+\theta }  } | \xi_{n}(x)  |
       \\\leq &C\bigg(\frac{  (1+\mu_{1,k_{n}}^{-1}| y- P_{1,k_{n}}  |)^ { \frac{N(p+1)}{2(q+1)}+\theta}  }{  (1+\mu_{1,k_{n}}^{-1}| y- P_{1,k_{n}}  |)^ { \frac{N(p+1)}{2(q+1)}+\theta+\theta_1}  }
       + \frac{  (1+\mu_{1,k_{n}}^{-1}| y- P_{1,k_{n}}  |)^ { \frac{N}{2}+\theta } }{(1+\mu_{1,k_{n}}^{-1}| y- P_{1,k_{n}}  |)^ { \frac{N}{2}+\theta+\theta_{1} }}      \bigg).
    \end{split}
\end{equation*}
Since $ (\tilde{\eta}_{n}(y),\tilde{\xi}_{n}(y)) \to 0$ in $B_{R}(0)$ and $||(\eta_{n},\xi_{n}  )||_{*}=1$, we see that
\begin{align*}
    \left(\frac{ \mu_{1,k_{n}}^{-\frac{N}{q+1}}  }{ (1+\mu_{1,k_{n}}^{-1}| y- P_{1,k_{n}}  |)^ { \frac{N(p+1)}{2(q+1)}+\theta}   } \right)^{-1}|  \eta_{n}(x) |
       +\frac{ \mu_{1,k_{n}}^{-\frac{N}{p+1}}  }{ (1+\mu_{1,k_{n}}^{-1}| y- P_{1,k_{n}}  |)^ { \frac{N}{2}+\theta }  } | \xi_{n}(x)  |
\end{align*}
attains its maximum in $\Omega \backslash  B_{R\mu_{1,k_{n}}}(P_{1,k_{n}})$. Thus $||(\eta_{n},\xi_{n}  )||_{*}=o(1)$, This is a contradiction to $ ||(\eta_{n},\xi_{n}  )||_{*}=1 $.

\begin{appendix}

\section{Preliminaries results}
 Let $G$ be the Green's function of the Laplacian $-\Delta$ in $\Omega$ with Dirichlet
boundary condition. And $H$ be its regular part, then $G(x,y)=S(x,y)-H(x,y)$ with $S(x,y) = \frac{\gamma_{N}}{|x-y|^{N-2}}$. 
In addition, we introduce a function $\tilde{G}=\tilde{G}_{\Omega} :\Omega \times \Omega \to \mathbb{R}$ satisfying
\begin{equation*}
    \begin{cases}
    -\Delta_{x}\tilde{G}(x,y)=G^{p}(x,y), &\hbox{ for }x\in\Omega,
    \\\tilde{G}=0,&\hbox{ for }x\in\partial\Omega,
    \end{cases}
\end{equation*}
for each $y\in\Omega$, and  its regular part $\tilde{H} = \tilde{H}_{\Omega} :\Omega \times \Omega \to \mathbb{R}$ by
\begin{equation*}
  \begin{split}
      \tilde{H}=\frac{\tilde{\gamma}_{N,p}}{|x-y|^{(N-2)p-2}}-\tilde{G}(x,y),
  \end{split}
\end{equation*}
where
\begin{equation*}
    \begin{split}
        \tilde{\gamma}_{N,p}:=\frac{\gamma_{N}^{p}}{((N-2)p-2)(N-(N-2)p)}>0.
    \end{split}
\end{equation*}

\begin{lemma}\label{Prele1}
Let $\widehat{H}:\Omega\times\Omega\to\mathbb{R}$ be a smooth function such that
\begin{equation*}
    \begin{cases}
    -\Delta_{x}\widehat{H}(x,y)=0, &\hbox{ for }x\in\Omega,
    \\\widehat{H}(x,y)=\frac{1}{|x-y|^{(N-2)p-2}},&\hbox{ for }x\in\partial\Omega,
    \end{cases}
\end{equation*}
for any $y\in\Omega$. Then we have
$$
P U_{1}(x)=U_{1}(x)-a_{N, p} \mu_{1}^{\frac{N p}{q+1}} \widehat{H}\left(x, P_{1}\right)+o\left(\mu^{\frac{N p}{q+1}}\right),
$$
and
$$
P V_{1}(x)=V_{1}(x)-\left(\frac{b_{N, p}}{\gamma_{N}}\right) \mu_{1}^{\frac{N}{q+1}} H\left(x, P_{1}\right)+o\left(\mu^{\frac{N}{q+1}}\right),
$$
where $PU_1, PV_1$ are the same as in \eqref{inteq7}.
\end{lemma}

\begin{proof}
Lemma \ref{Prele1} is proved in  \cite{Kim-Pis} by using comparison principle. We give a different proof. Indeed,  we have
\begin{align*}
    PV_{1}(x)=\int_{\Omega}G(x,y)U_{1}^{q}dy,\quad V_{1}(x)=\int_{\mathbb R^{N}}S(x,y)U_{1}^{q}dy.
\end{align*}
Thus
\begin{align*}
     PV_{1}(x)-V_{1}(x)=&-\int_{\Omega^{c}}S(x,y)U_{1}^{q}dy-\int_{\Omega}H(x,y)U_{1}^{q}dy:=\mathcal{A}_{1}+ \mathcal{A}_{2}.
\end{align*}
Since $dist(P_{1},\Omega) >\delta_{2}>0,$ we have
\begin{align*}
    | \mathcal{A}_{1} |\leq &C\int_{\Omega^{c}}\frac{1}{|x-y|^{N-2}}\frac{\mu_{1}^{\frac{Npq}{q+1}}}{(1 + |y-P_{1}|)^{\frac{N(p+1)q}{q+1}} }
    \leq C\mu_{1}^{\frac{Npq}{q+1}}.
\end{align*}
And
\begin{align*}
    \mathcal{A}_{2} =  \int_{\tilde{\Omega}} H(x, \mu_{1}y+P_{1} )U_{1,0}^{q}(y)dy\mu_{1}^{\frac{N}{q+1}},
\end{align*}
where $ \tilde{\Omega} = \mu_{1}^{-1}( \Omega -P_{1}  )$.
Since $ H(x, \mu_{1}y+P_{1} )U_{1,0}^{q}(y) \leq CU_{1,0}^{q}(y)  $, by using the dominated convergence theorem, we have
\begin{align*}
    \int_{\tilde{\Omega}} H(x, \mu_{1}y+P_{1} )U_{1,0}^{q}(y)dy =  H(x,P_{1})\int_{\mathbb{R}^{N}}U_{1,0}^{q}(y)dy + o(1).
\end{align*}
So,
\begin{align*}
    PV_{1}(x)-V_{1}(x) = \mu_{1}^{\frac{N}{q+1}}H(x,P_{1})\int_{\mathbb{R}^{N}}U_{1,0}^{q}(y)dy +o(\mu_{1}^{\frac{N}{q+1}}) .
\end{align*}
Moreover,  $\displaystyle\int_{\mathbb{R}^{N}}U_{1,0}^{q}(y)dy = \frac{b_{N, p}}{\gamma_{N}} $.

\end{proof}
Similar, we can prove
\begin{lemma}
$
\frac{\partial P V_{1}}{\partial x_{i}}(x) = \frac{\partial V_{1}}{ \partial x_{i}}(x)-\left(\frac{b_{N, p}}{\gamma_{N}}\right)\mu_{1}^{\frac{N}{q+1}}\frac{\partial H}{ \partial x_{i}}\left(x, P_{1}\right)+o\left(\mu^{\frac{N}{q+1}}\right).
$
\end{lemma}
\begin{lemma}
[Theorem 2 in \cite{Lions}] There exist positive constants $a_{N,p}$ and $b_{N,p}$ depending
only on $N$ and $p$ such that
\begin{equation*}
    \begin{cases}
    \lim_{r\to\infty}r^{(N-2)p-2}U_{1,0}(r)=a_{N,p},
    \\\lim_{r\to\infty}r^{N-2}V_{1,0}(r)=b_{N,p},
    \end{cases}
\end{equation*}
where we wrote $U_{1,0}(x)= U_{1,0}(|x|)$, $V_{1,0}(x)= V_{1,0}(|x|)$  and $r = |x|$ by abusing notations.
Furthermore,
\begin{equation*}
    b_{N,p}^{p}=a_{N,p}((N-2)p-2)(N-(N-2)p).
\end{equation*}
\end{lemma}
\begin{lemma}
[Theorem 1 in \cite{Fra-Kim-Pis}] Set
\begin{align*}
    (\Psi_{0},\Phi_{0}) = \bigg(x\cdot\nabla U_{1,0}+\frac{N U_{1,0}}{q+1},x\cdot\nabla V_{1,0}+\frac{N V_{1,0}}{p+1}    \bigg),
\end{align*}
and
\begin{align*}
    (\Psi_{l},\Phi_{l}) = ( \frac{\partial U_{1,0}}{\partial x_{l}},\frac{\partial V_{1,0}}{\partial x_{l}} ),\,\,\,\hbox{ for }l=1,\cdots,N.
\end{align*}
Then the space of solutions to the linear system
\begin{equation*}
    \begin{cases}
    -\Delta \Psi =p V_{1,0}^{p-1}\Phi, \hbox{ in }\mathbb{R}^{N},
    \\ -\Delta \Phi =p U_{1,0}^{q-1}\Psi, \hbox{ in }\mathbb{R}^{N},
    \\( \Psi,\Phi  )\in\dot{W}^{2,\frac{p+1}{p}}( \mathbb{R}^{N}  )\times\dot{W}^{2,\frac{q+1}{q}}( \mathbb{R}^{N}  ),
    \end{cases}
\end{equation*}
is spanned by
\begin{align*}
    \{ (\Psi_{0},\Phi_{0}),(\Psi_{1},\Phi_{1}) ,\cdots,(\Psi_{N},\Phi_{N})   \}.
\end{align*}
\end{lemma}
Set
\begin{align*}
    (\Psi_{1,0},\Phi_{1,0}   ) = (\mu_{1}^{-\frac{N}{q+1}-1}\Psi_{0}(\mu_{1}^{-1}(x-P_{1}))  , \mu_{1}^{-\frac{N}{p+1}-1}\Phi_{0}(\mu_{1}^{-1}(x-P_{1}))      ),
\end{align*}
and
\begin{align*}
    (\Psi_{1,l},\Phi_{1,l}   ) = (\mu_{1}^{-\frac{N}{q+1}-1}\Psi_{l}(\mu_{1}^{-1}(x-P_{1}))  , \mu_{1}^{-\frac{N}{p+1}-1}\Phi_{l}(\mu_{1}^{-1}(x-P_{1}))      ),
\end{align*}
for $l=1,\cdots,N $. Let the pair $( P\Psi_{1,l},P\Phi_{1,l}  ) $ be the unique smooth solution of the system
\begin{equation*}
    \begin{cases}
    -\Delta P\Psi_{1,l} =p V_{1}^{p-1}P\Phi_{1,l}, \hbox{ in }\Omega,
    \\ -\Delta P\Phi_{1,l} =p U_{1}^{q-1}P\Psi_{1,l}, \hbox{ in }\Omega,
    \\P\Psi_{1,l}=P\Phi_{1,l}=0, \hbox{ in }\partial\Omega,
    \end{cases}
\end{equation*}
for $l=1,\cdots,N $. Then, we have the following Lemma.

\begin{lemma}
[Lemma 2.10. in \cite{Kim-Pis}]
$$
P \Psi_{1, l}(x)= \begin{cases}\Psi_{1, l}(x)+\frac{Np}{q+1}a_{N, p} \mu_{1}^{\frac{N p}{q+1}-1} \widehat{H}\left(x, P_{1}\right)+o\left(\mu^{\frac{N p}{q+1}-1}\right), & \text { for } l=0, \\
\Psi_{1,l}(x)+a_{N, p} \mu_{1}^{\frac{N p}{q+1}} \partial_{P_1, l} \widehat{H}\left(x, P_{1}\right)+o\left(\mu^{\frac{N p}{q+1}}\right), & \text { for } l=1, \cdots, N,\end{cases}
$$
and
$$
P \Phi_{1,l}(x)= \begin{cases}\Phi_{1,l}(x)+\left(\frac{N}{q+1}\frac{b_{N, p}}{\gamma_{N}}\right) \mu_{1}^{\frac{N}{q+1}-1} H\left(x, P_{1}\right)+o\left(\mu^{\frac{N}{q+1}-1}\right),
 & \text { for } l=0, \\ \Phi_{1,l}(x)+\left(\frac{b_{N, p}}{\gamma_{N}}\right) \mu_{1}^{\frac{N}{q+1}} \partial_{P_{1}, l} H\left(x, P_{1}\right)+o\left(\mu^{\frac{N}{q+1}}\right), & \text { for } l=1, \cdots, N,\end{cases}
$$
for $x \in \Omega$. Here, $\partial_{P_1, l} \widehat{H}(x, P_1)$ and $\partial_{P_1, l} H(x, P_1)$ stand for the $l$-th components of $\nabla_{P_1} \widehat{H}(x, P_1)$ and $\nabla_{P_1} H(x, P_1)$, respectively.
\end{lemma}
Recall that
\begin{equation*}
 \begin{cases}
-\Delta PU_{d_{1},P_{1}} = PV_{1}^{p}, &\hbox{ in }\Omega,
\\PU_{d_{1},P_{1}}=0,&\hbox{ on }\partial\Omega.
\end{cases}
\end{equation*}
Let $\tilde{G}_{d_{1},P_{1}}:=\Omega \to \mathbb{R}^{N}$ be the solution of
\begin{equation*}
 \begin{cases}
-\Delta \tilde{G}_{d_{1},P_{1}}(x) = d_{1}^{\frac{N}{q+1}}G(x,P_{1}), &\hbox{ for }x\in\Omega,
\\\tilde{G}=0,&\hbox{ for }x\in\partial\Omega.
\end{cases}
\end{equation*}
and $ \tilde{H}_{d_{1},P_{1}}:=\Omega \to \mathbb{R}^{N}$ be its regular part given by
\begin{equation*}
    \begin{split}
    \tilde{H}_{d_{1},P_{1}} = d_{1}^{\frac{N}{q+1}}\frac{\tilde{\gamma}_{N,p}}{|x-P_{1}|^{(N-2)p-2}}-\tilde{G}_{d_{1},P_{1}}(x).
    \end{split}
\end{equation*}
Then we have the following Lemma.
\begin{lemma}
[Lemma 2.12. in \cite{Kim-Pis}]For any $x\in\Omega$, we have
$$
P U_{{d_1}, {P_1}}(x)=\sum_{i=1}^{k} U_{i}(x)-\mu^{\frac{N p}{q+1}}\left(\frac{b_{N, p}}{\gamma_{N}}\right)^{p} \widetilde{H}_{{d_1}, {P_1}}(x)+o\left(\mu^{\frac{N p}{q+1}}\right).
$$

\end{lemma}

\section{Some important estimation}
Recall that $S(x,y) = \frac{\gamma_{N}}{|x-y|^{N-2}}$. And let $\epsilon_{n} \to 0$, as $ n\to +\infty$.
\begin{lemma}\label{ble1}
    $\frac{\partial PU_{d_{1,k_{n}},P_{1,k_{n}}  }}{\partial x_i}(P_{1,k_{n}}) = o( \mu_{1,k_{n}}^{\frac{Np}{q+1}} )$.
\end{lemma}
\begin{proof}
    We have
\begin{equation*}
\begin{split}
\frac{\partial PU_{d_{1,k_{n}},P_{1,k_{n}}  }}{\partial x_i}(x) =\int_{\Omega} \frac{\partial G  }{\partial x_i}(x,y)PV_{1,k_{n}}^{p}(y)dy,
\end{split}
\end{equation*}
and
\begin{equation*}
\begin{split}
\frac{\partial U_{1,k_{n}  }}{\partial x_i}(x) =\int_{\mathbb{R}^{N}} \frac{\partial S  }{\partial x_i}(x,y)V_{1,k_{n}}^{p}(y)dy.
\end{split}
\end{equation*}
So
\begin{equation*}
\begin{split}
\frac{\partial ( PU_{d_{1,k_{n}},P_{1,k_{n}} } -U_{1,k_{n}  } ) }  {\partial x_i}
=& -\int_{\Omega^{c}} \frac{\partial S  }{\partial x_i}(x,y)V_{1,k_{n}}^{p}(y)dy + \int_{\Omega} \frac{\partial S  }{\partial x_i}(x,y) (pV_{1,k_{n}}^{p} - V_{1,k_{n}}^{p})(y)dy
\\&-\int_{\Omega}\frac{\partial H  }{\partial x_i}(x,y)PV_{1,k_{n}}^{p}(y):=\mathcal{I}_1+\mathcal{I}_2+\mathcal{I}_3.
\end{split}
\end{equation*}
It is easy to check that
\begin{align*}
    \mathcal{I}_1 = -\mu_{k_{n}}^{\frac{Np}{q+1}}\int_{\Omega^{c}} \frac{\partial S  }{\partial x_i}(x,y)\frac{b_{N,p}^{p}}{\gamma^{p}_{N}}S^{y}(y,P_{0}) + o( \mu_{k_{n}}^{\frac{Np}{q+1}}),
    \\\mathcal{I}_2= \mu_{k_{n}}^{\frac{Np}{q+1}}\int_{\Omega} \frac{\partial S  }{\partial x_i}(x,y)\frac{b_{N,p}^{p}}{\gamma_{N}^{p}}(G(y,P_{0})^{p} - S^{p}(y,P_{1})  )+o(\mu_{k_{n}}^{\frac{Np}{q+1}}),
    \\\mathcal{I}_3= -\mu_{k_{n}}^{\frac{Np}{q+1}}\int_{\Omega}\frac{\partial H  }{\partial x_i}(x,y)G^{p}(y,P_{0})dy+o(\mu_{k_{n}}^{\frac{Np}{q+1}}).
\end{align*}
Thus
\begin{align*}
  \frac{\partial PU_{d_{1,k_{n}},P_{1,k_{n}}  }}{\partial x_i}(x) =  \mu_{k_{n}}^{\frac{Np}{q+1}}\frac{\partial\widetilde{H}}{\partial x_{i}}(x,P_{0})+o(\mu_{k_{n}}^{\frac{Np}{q+1}} ) .
\end{align*}
From \cite{Kim-Pis}, we get that $\frac{\partial\widetilde{H}}{\partial x_{i}}(P_{0},P_{0}) =0 $.
So $\frac{\partial PU_{d_{1,k_{n}},P_{1,k_{n}}  }}{\partial x_i}(P_{1,k_{n}}) = o( \mu_{k_{n}}^{\frac{Np}{q+1}} ) $.
\end{proof}

\begin{lemma}
For $x\in B_{\mu_{1,k_{n}}^{\alpha}}(P_{1,k_{n}})$, we have $\frac{\partial (PU_{d_{1,k_{n}},P_{1,k_{n}}}  - U_{1,k_{n}  })}{\partial x_i}(x) = O(\mu_{1,k_{n}}^{\frac{Np}{q+1}})$, as $ n\to +\infty$, where $\alpha$ is a small fixed positive constant.
\end{lemma}
\begin{proof}
We have
\begin{equation*}
\begin{split}
\frac{\partial PU_{d_{1,k_{n}},P_{1,k_{n}}  }}{\partial x_i}(x) =\int_{\Omega} \frac{\partial G  }{\partial x_i}(x,y)PV_{1,k_{n}}^{p}(y)dy,
\end{split}
\end{equation*}
and
\begin{equation*}
\begin{split}
\frac{\partial U_{1,k_{n}  }}{\partial x_i}(x) =\int_{\mathbb{R}^{N}} \frac{\partial S  }{\partial x_i}(x,y)V_{1,k_{n}}^{p}(y)dy.
\end{split}
\end{equation*}
So
\begin{equation*}
\begin{split}
\frac{\partial ( U_{1,k_{n}  }  -PU_{d_{1,k_{n}},P_{1,k_{n}}) }}  {\partial x_i}
=& \int_{\Omega^{c}} \frac{\partial S  }{\partial x_i}(x,y)V_{1,k_{n}}^{p}(y)dy + \int_{\Omega} \frac{\partial S  }{\partial x_i}(x,y) (V_{1,k_{n}}^{p} - PV_{1,k_{n}}^{p})(y)dy
\\&+\int_{\Omega}\frac{\partial H  }{\partial x_i}(x,y)PV_{1,k_{n}}^{p}(y):=\mathcal{I}_1+\mathcal{I}_2+\mathcal{I}_3.
\end{split}
\end{equation*}
Direct computations shows that
\begin{equation*}
\begin{split}
|\mathcal{I}_1|
\leq &C\mu_{1,k_{n}}^{\frac{Np}{q+1}}\int_{\Omega^{c}}\frac{1}{|x-y|^{N-1}}\frac{1}{  (1+ |y-P_{1,k_{n}}| )^{(N-2)p} }
\leq C\mu_{1,k_{n}}^{\frac{Np}{q+1}},
\end{split}
\end{equation*}
and
\begin{equation*}
\begin{split}
| \mathcal{I}_2 |\leq &C\int_{\Omega}\frac{1}{|x-y|^{N-1}}\bigg( \frac{\mu_{1,k_{n}}^{-\frac{N(p-1)}{p+1}}}{(  1+\mu_{1,k_{n}}^{-1}|y-P_{1,k_{n}}|)^{  (N-2)(p-1) }}\mu_{1,k_{n}}^{\frac{N}{q+1}}+ O(\mu_{1,k_{n}}^{\frac{Np}{q+1}}  )\bigg)dy
\leq C\mu_{1,k_{n}}^{\frac{Np}{q+1}}.
\end{split}
\end{equation*}
Since $x\in B_{\mu_{1,k_{n}}^{\alpha}}(P_{1,k_{n}})$ and $dist(P_{1,k_{n}},\Omega   )>\delta_2$, we have $ |\frac{\partial H  }{\partial x_i}(x,y)|\leq C  $. Thus
\begin{equation*}
    \begin{split}
        |\mathcal{I}_3|\leq& C\int_{\Omega}\frac{\mu_{1,k_{n}}^{-\frac{Np}{p+1}}}{(  1  + \mu_{1,k_{n}}^{-1}|  y-P_{1,k_{n}}|)^{(N-2)p}}dy
        \leq C\mu_{1,k_{n}}^{\frac{Np}{q+1}}.
    \end{split}
\end{equation*}
\end{proof}

\begin{lemma}\label{Ble8}
For $x\in B_{\mu_{1,k_{n}}^{\alpha}}(P_{1,k_{n}})$, we have
$$\frac{\partial (P\Phi_{1,k_{n},0}  - \Phi_{1,k_{n},0})}{\partial x_i}(x) = O(\mu_{1,k_{n}}^{\frac{N}{q+1}-1}),\quad
 \frac{\partial (P\Phi_{1,k_{n},j}  - \Phi_{1,k_{n},j})}{\partial x_i}(x) = O(\mu_{1,k_{n}}^{\frac{N}{q+1}}),
 $$
 for $j=1,\cdots,N$, as $ n\to +\infty$, where $\alpha$ is a small fixed positive constant.
\end{lemma}
\begin{proof}
We have
\begin{equation*}
\begin{split}
&\frac{\partial (P\Phi_{1,k_{n},0}  - \Phi_{1,k_{n},0})}{\partial x_i}(x)\\
=& \int_{\Omega^{c}} \frac{\partial S  }{\partial x_i}(x,y)qU_{1,k_{n}  }^{q-1}\Psi_{1,n,0}(y)dy
+\int_{\Omega}\frac{\partial H  }{\partial x_i}(x,y)qU_{1,k_{n}  }^{q-1}\Psi_{1,n,0}(y)
\\\leq &C\mu_{1,k_{n}}^{\frac{Npq}{q+1}-1}\int_{\Omega^{c}}\frac{1}{|x-y|^{N-1}}\frac{1}{  (1+ |y-P_{1,k_{n}}| )^{\frac{N(p+1)q }{q+1}} }+C\int_{\Omega}\frac{\mu_{1,k_{n}}^{-\frac{Np}{p+1}-1}}{(  1  + \mu_{1,k_{n}}^{-1}|  y-P_{1,k_{n}}|)^{\frac{N(p+1)q }{q+1}}}dy\\
\leq &C\mu_{1,k_{n}}^{\frac{Npq}{q+1}-1}+C\mu_{1,k_{n}}^{\frac{N}{q+1}-1}.\\
\end{split}
\end{equation*}

For $j=1,\cdots,N$, we have
\begin{equation*}
\begin{split}
&\frac{\partial (P\Phi_{1,k_{n},j}  - \Phi_{1,k_{n},j})}{\partial x_i}(x)\\
=& \int_{\Omega^{c}} \frac{\partial S  }{\partial x_i}(x,y)pU_{1,k_{n}  }^{q-1}\Psi_{1,n,j}(y)dy
+\int_{\Omega}\frac{\partial H  }{\partial x_i}(x,y)pU_{1,k_{n}  }^{q-1}\Psi_{1,n,j}(y)\\
\leq &C\mu_{1,k_{n}}^{\frac{Npq}{q+1}}\int_{\Omega^{c}}\frac{1}{|x-y|^{N-1}}\frac{1}{  (1+ |y-P_{1,k_{n}}| )^{\frac{N(p+1)q }{q+1}+1} }\\
+&\int_{B_{\delta}(P_{1,k_{n}})}\frac{\partial H  }{\partial x_i}(x,y)pU_{1,k_{n}  }^{q-1}\Psi_{1,n,j}(y)+\int_{\Omega-B_{\delta}(P_{1,k_{n}})}\frac{\partial H  }{\partial x_i}(x,y)pU_{1,k_{n}  }^{q-1}\Psi_{1,n,j}(y)\\
\leq &C\mu_{1,k_{n}}^{\frac{Npq}{q+1}}+O(\mu_{1,k_{n}}^{\frac{Npq}{q+1}}  )+O( \mu_{1,k_{n}}^{\frac{N}{q+1}} ),
\end{split}
\end{equation*}
where $\delta$ is a fix small constant.
\end{proof}

\begin{lemma}
    We have
    \begin{align*}
         \mu_{1,k_{n}} \widetilde{P\Psi}_{1,k_{n},0}(x) -  \mu_{1,k_{n}} \Psi_{1,k_{n},0}(x) = O(\mu_{1,k_{n}}^{\frac{Np}{q+1}})
    \end{align*}
    and
    \begin{align*}
        \mu_{1,k_{n}} \widetilde{P\Psi}_{1,k_{n},j}(x) -  \mu_{1,k_{n}} \Psi_{1,k_{n},j}(x) = O(\mu_{1,k_{n}}^{\frac{Np}{q+1} +1})
    \end{align*}
    for $ j\neq 0$.
\end{lemma}
\begin{proof}
   for  We have
    \begin{align*}
    &\mu_{1,k_{n}} \widetilde{P\Psi}_{1,k_{n},0}(x) -  \mu_{1,k_{n}} \Psi_{1,k_{n},0}(x)
    \\=&\int_{\Omega}G(x,y)p( \mu(PV_{1,k_{n}})^{p-1}P\Phi_{1,k_{n},0}    ) - \int_{\mathbb{R}^{N}}s(x,y)p  \mu(V_{1,k_{n}})^{p-1}\Phi_{1,k_{n},0}
    \\=&-\int_{\Omega^{c}}S(x,y)p (V_{1,k_{n}})^{p-1} \mu\Phi_{1,k_{n},0}-\int_{\Omega}H(x,y)p( (PV_{1,k_{n}})^{p-1} \mu_{1,k_{n}} P\Phi_{1,k_{n},0}    )
    \\&+ \int_{\Omega}S(x,y)p( (PV_{1,k_{n}})^{p-1}P\Phi_{1,k_{n},0} -  \mu(V_{1,k_{n}})^{p-1}\Phi_{1,k_{n},0}    )
    \\=&I_1+I_2+I_3.
\end{align*}
It is easy to check that $|I_{1}|+|I_{2}|+|I_{3}| = O(\mu_{1,k_{n}}^{\frac{Np}{q+1}}) $. Simlarly, we can prove that $\mu_{1,k_{n}} \widetilde{P\Psi}_{1,k_{n},j}(x) -  \mu_{1,k_{n}} \Psi_{1,k_{n},j}(x) = O(\mu_{1,k_{n}}^{\frac{Np}{q+1} +1})$.
\end{proof}

\begin{lemma}\label{ble5}
    We have
    \begin{align*}
        &\int_{B_{\mu_{1,k_{n}}^{\alpha}}(P_{0})}U^{q}_{1}\frac{\partial \sum_{j=1}^{N} b_{j,n}\mu_{1,k_{n}} (\widetilde{P\Psi}_{1,k_{n},j}-\Psi_{1,k_{n},j})  }{\partial x_{i}}(x)
        \\&+
        \int_{B_{\mu_{1,k_{n}}^{\alpha}}(P_{0})}\sum_{j=1}^{N}qU_{1,k_{n}}^{q-1}\Psi_{1,k_{n},j}\frac{\partial (PU_{d_{1,k_{n}},P_{1,k_{n}}} - U_{1,k_{n}})  }{  \partial x_{i}}
        \\=&-\frac{\mu_{1,k_{n}}^{\frac{N(p+1)}{q+1}+1}}{p+1}(\frac{b_{N,p}}{\gamma_{N}})^{p}\int_{\mathbb{R}^{N}}U_{0,1}^{q}\sum_{j=1}^{N}b_{j,n}\frac{\partial^{2}\tau }{\partial x_{i}\partial x_{j}}(P_{0})+o(\mu_{1,k_{n}}^{\frac{N(p+1)}{q+1}+1}).
    \end{align*}
\end{lemma}
\begin{proof}

 Firstly, we estimate $ \frac{\partial \mu_{1,k_{n}} \widetilde{P\Psi}_{1,k_{n},j}}{\partial x_{i}}(x) - \frac{\partial \mu_{1,k_{n}} \Psi_{1,k_{n},j}}{\partial x_{i}}(x)$.

\begin{align*}
    &\frac{\partial \mu_{1,k_{n}} \widetilde{P\Psi}_{1,k_{n},j}}{\partial x_{i}}(x) - \frac{\partial \mu_{1,k_{n}} \Psi_{1,k_{n},j}}{\partial x_{i}}(x)
    \\=&\int_{\Omega}\frac{\partial G}{\partial x_{i}}(x,y)p( \mu_{1,k_{n}}(PV_{1,k_{n}})^{p-1}P\Phi_{1,k_{n},j}    ) - \int_{\mathbb{R}^{N}}\frac{\partial S}{\partial x_{i}}(x,y)p  \mu_{1,k_{n}}(V_{1,k_{n}})^{p-1}\Phi_{1,k_{n},j}
    \\=&-\int_{\Omega^{c}}\frac{\partial S}{\partial x_{i}}(x,y)p (V_{1,k_{n}})^{p-1} \mu_{1,k_{n}}\Phi_{1,k_{n},j}-\int_{\Omega}\frac{\partial H}{\partial x_{i}}(x,y)p( (PV_{1,k_{n}})^{p-1} \mu_{1,k_{n}} P\Phi_{1,k_{n},j}    )
    \\&+ \int_{\Omega}\frac{\partial S}{\partial x_{i}}(x,y)p( \mu_{1,k_{n}}(PV_{1,k_{n}})^{p-1}P\Phi_{1,k_{n},j} -  \mu_{1,k_{n}}(V_{1,k_{n}})^{p-1}\Phi_{1,k_{n},j}    )
    \\=&I_1+I_2+I_3.
\end{align*}
It is easy to check that
\begin{align*}
    I_1 =\mu_{1,k_{n}}^{\frac{Np}{q+1} +1 }(\frac{b_{N,p}}{\gamma_{N}})^{p}\int_{\Omega^{c}}\frac{\partial S}{\partial x_{i}}(x,y)pS^{p-1}(P_{0},y)\frac{\partial S}{\partial x_{j}}(P_{0},y) +o(\mu_{1,k_{n}}^{\frac{Np}{q+1} +1 }) ,
    \\ I_2= \mu_{1,k_{n}}^{\frac{Np}{q+1} +1 }(\frac{b_{N,p}}{\gamma_{N}})^{p}\int_{\Omega}\frac{\partial H}{\partial x_{i}}(x,y)p G(P_{0},y)^{p-1} \frac{\partial G}{\partial x_{j}}(P_{0},y)    )+o(\mu_{1,k_{n}}^{\frac{Np}{q+1} +1 }).
\end{align*}
Now we rewrite $I_3$
\begin{align*}
    I_{3}=&\int_{\Omega}\frac{\partial S}{\partial x_{i}}(x,y)\mu_{1,k_{n}} p( (PV_{1,k_{n}})^{p-1} -(V_{1,k_{n}})^{p-1}    )\Phi_{1,k_{n},j}
    \\&+\int_{\Omega}\frac{\partial S}{\partial x_{i}}(x,y)\mu_{1,k_{n}} p (PV_{1,k_{n}})^{p-1} (P\Phi_{1,k_{n},j} - \Phi_{1,k_{n},j})
    \\=&I_{31}+I_{32}.
\end{align*}
   It is easy to check that
   \begin{align*}
      I_{32}=\mu_{1,k_{n}}^{\frac{Np}{q+1} +1 }(\frac{b_{N,p}}{\gamma_{N}})^{p} \int_{\Omega}\frac{\partial S}{\partial x_{i}}(x,y)pG^{p-1}(P_{0},y)\frac{\partial H}{\partial x_{j}}(P_{0},y)  + o(\mu_{1,k_{n}}^{\frac{Np}{q+1} +1 } ).
   \end{align*}
 Now we rewrite $I_{31}$
 \begin{align*}
     I_{31}=&\int_{\Omega}\frac{\partial S}{\partial x_{i}}(x,y)\mu_{1,k_{n}} p( (PV_{1,k_{n}})^{p-1} -(V_{1,k_{n}})^{p-1} + (p-1)V_{1,k_{n}}^{p-2}\frac{b_{N,p}}{\gamma_{N}}\mu^{\frac{N}{q+1}}H(P_{0},y)   )\Phi_{1,k_{n},j}
     \\&- \int_{\Omega}\frac{\partial S}{\partial x_{i}}(x,y)\mu^{\frac{N}{q+1}+1}pV_{1,k_{n}}^{p-2}\frac{b_{N,p}}{\gamma_{N}}(p-1)V_{1,k_{n}}^{p-2}\frac{b_{N,p}}{\gamma_{N}}\Phi_{1,k_{n},j}
     \\=&I_{311}+I_{312}.
 \end{align*}
 It is easy to check that
 \begin{align*}
     I_{311} = &-\mu_{1,k_{n}}^{\frac{Np}{q+1}+1}\int_{\Omega}\frac{\partial S}{\partial x_{i}}(P_{0},y)p(\frac{b_{N,p}}{\gamma_{N}})^{p}(   G^{p-1}(P_{0},y)
     -S^{p-1}(P_{0},y)
     \\&+(p-1)S^{p-2}(P_{0},y)H(P_{0},y)  )\frac{\partial S}{\partial x_{j}}(P_{0},y) + o(\mu_{1,k_{n}}^{\frac{Np}{q+1}+1}  ).
 \end{align*}

Now we estimate $ \frac{\partial PU_{d_{1,k_{n}},P_{1,k_{n}}} }{\partial x_{i}} -\frac{\partial U_{1,k_{n}} }{\partial x_{i}} $,
\begin{align*}
    &\frac{\partial PU_{d_{1,k_{n}},P_{1,k_{n}}} }{\partial x_{i}} -\frac{\partial U_{1,k_{n}} }{\partial x_{i}}(x)
    \\=&\int_{\Omega}\frac{\partial G}{\partial x_{i}}(x,y)(PV_{1,k_{n}})^{p}-\int_{\mathbb{R}^{N}}\frac{\partial S}{\partial x_{i}}(x,y)V_{1,k_{n}}^{p}
    \\=&-\int_{\Omega^{c}}\frac{\partial S}{\partial x_{i}}(x,y)V_{1,k_{n}}^{p}-\int_{\Omega}\frac{\partial H}{\partial x_{i}}(x,y)(PV_{1,k_{n}})^{p}
    +\int_{\Omega}\frac{\partial S}{\partial x_{i}}(x,y)((PV_{1,k_{n}})^{p} - V_{1,k_{n}}^{p})
    \\=&I_{4}+I_{5}+I_{6}.
\end{align*}
It is easy to check that
\begin{align*}
    I_{4} = &-\int_{\Omega^{c}}\frac{\partial S}{\partial x_{i}}(P_{0},y)V_{1,k_{n}}^{p}-\mu_{1,k_{n}}^{\frac{Np}{q+1}}(\frac{b_{N,p}}{\gamma_{N}})^{p}\int_{\Omega^{c}}\sum_{z=1}^{N}\frac{\partial^2 S}{\partial x_{i}\partial x_{z}}(P_{0},y)(x-P_{0})_{z}S^{p}(P_{0},y)
    \\&+O(\mu_{1,k_{n}}^{\frac{Np}{q+1}}| x-P_{0}|^{2} ).
\end{align*}
and
\begin{align*}
    I_{5} = &-\int_{\Omega}\frac{\partial H}{\partial x_{i}}(P_{0},y)(PV_{1,k_{n}})^{p}- \mu_{1,k_{n}}^{\frac{Np}{q+1}}(\frac{b_{N,p}}{\gamma_{N}})^{p}\int_{\Omega}\sum_{z=1}^{N}\frac{\partial^2 H}{\partial x_{i}\partial x_{z}}(P_{0},y)(x-P_{0})_{z}G^{p}(P_{0},y)
    \\&+O(\mu_{1,k_{n}}^{\frac{Np}{q+1}}| x-P_{0}|^{2} ).
\end{align*}
Thus
\begin{align*}
   &\int_{B_{\mu_{1,k_{n}}^{\alpha}}(P_{1})}q\mu\sum_{j=1}^{N}U_{1,k_{n}}^{q-1}\Psi_{1,k_{n},j}I_{4}dx
   \\=&\mu^{\frac{N(p+1)}{q+1}+1}(\frac{b_{N,p}}{\gamma_{N}})^{p}\int_{\mathbb{R}^{N}}U_{0,1}^{q} \int_{\Omega^{c}}\frac{\partial^2 S}{\partial x_{j}\partial x_{z}}(P_{0},y)S^{p}(P_{0},y)+o(\mu_{1,k_{n}}^{\frac{N(p+1)}{q+1}+1} ),
\end{align*}
\begin{align*}
    &\int_{B_{\mu^{\alpha}}(P_{1})}q\mu\sum_{j=1}^{N}U_{1,k_{n}}^{q-1}\Psi_{1,k_{n},j}I_{5}dx
    \\=&\mu_{1,k_{n}}^{\frac{N(p+1)}{q+1}+1}(\frac{b_{N,p}}{\gamma_{N}})^{p} \int_{\mathbb{R}^{N}}U_{0,1}^{q}\int_{\Omega}\frac{\partial^2 H}{\partial x_{i}\partial x_{j}}(P_{0},y)G^{p}(P_{0},y)+o(\mu_{1,k_{n}}^{\frac{N(p+1)}{q+1}+1} ).
\end{align*}

Now we estimate $I_{6}$.
\begin{align*}
    I_{6} = &\int_{\Omega}\frac{\partial S}{\partial x_{i}}(x,y)((PV_{1,k_{n}})^{p} - V_{1,k_{n}}^{p})dy
    \\=&\int_{\Omega}-\frac{\partial S}{\partial y_{i}}(x,y)((PV_{1,k_{n}})^{p} - V_{1,k_{n}}^{p})dy
    \\=&-\int_{\partial\Omega}S(x,y)((PV_{1,k_{n}})^{p} - V_{1,k_{n}}^{p})\nu_{i}ds+\int_{\Omega}S(x,y)\frac{\partial ((PV_{1,k_{n}})^{p} - V_{1,k_{n}}^{p})}{\partial y_{i}}dy
    \\=&-\int_{\Omega}S(x,y)((PV_{1,k_{n}})^{p} - V_{1,k_{n}}^{p})\nu_{i}dy+\int_{\Omega}S(x,y)p( (PV_{1,k_{n}})^{p-1}\frac{\partial PV_{1,k_{n}}}{\partial y_{i}}-(V_{1,k_{n}})^{p-1}\frac{\partial V_{1,k_{n}}}{\partial y_{i}}  )dy
    \\=&I_{61}+I_{62}.
\end{align*}
Direct computation shows that
\begin{align*}
    &\int_{B_{\mu^{\alpha}}(P_{0})}q\mu_{1,k_{n}} U_{1,k_{n}}^{q-1}(x)\Psi_{1,k_{n},j}(x) I_{61}dx
    \\=&-\int_{\partial B_{\mu^{\alpha}}(P_{0})}\mu_{1,k_{n}} U_{1,k_{n}}^{q}\int_{\partial\Omega}S(x,y)((PV_{1,k_{n}})^{p} - V_{1,k_{n}}^{p})ds_{y}\nu_{j,x}ds_{x}
    \\&+\int_{B_{\mu^{\alpha}}(P_{0})}\mu_{1,k_{n}} U_{1,k_{n}}^{q}\int_{\partial\Omega}\frac{\partial S}{\partial x_{j}}(x,y)((PV_{1,k_{n}})^{p} - V_{1,k_{n}}^{p})\nu_{i}ds_{y}dx
    \\=&o(\mu_{1,k_{n}}^{\frac{N(p+1)}{q+1}+1})+\mu_{1,k_{n}}^{\frac{N(p+1)}{(q+1)}+1}(\frac{b_{N,p}}{\gamma_{N}})^{p}\int_{\mathbb{R}^{N}}U_{0,1}^{q}dx\int_{\partial\Omega}\frac{\partial S}{\partial x_{j}}(P_{0},y)(G^{p}(P_{0},y) - S^{p}(P_{0},y))\nu_{i}ds_{y}.
\end{align*}
We rewrite $I_{62}$
\begin{align*}
    I_{62}=&\int_{\Omega}S(x,y)p( (PV_{1,k_{n}})^{p-1}\frac{\partial PV_{1,k_{n}}}{\partial y_{i}}-(V_{1,k_{n}})^{p-1}\frac{\partial V_{1,k_{n}}}{\partial y_{i}}  )dy
    \\=&\int_{\Omega}S(x,y) p((PV_{1,k_{n}})^{p-1}  -V_{1,k_{n}}^{p-1} )\frac{\partial V_{1,k_{n}}}{\partial y_{i}}dy
    \\&+\int_{\Omega}S(x,y) p(PV_{1,k_{n}})^{p-1} (\frac{\partial PV_{1,k_{n}}}{\partial y_{i}} -\frac{\partial V_{1,k_{n}}}{\partial y_{i}}   )dy
    \\=&\int_{\Omega}S(x,y) p((PV_{1,k_{n}})^{p-1}  -V_{1,k_{n}}^{p-1} + (p-1)\frac{b_{N,p}}{\gamma_{N}}V_{1,k_{n}}^{p-2}H(P_{0},y)\mu^{\frac{N}{q+1}} )\frac{\partial V_{1,k_{n}}}{\partial y_{i}}dy
    \\&-\int_{\Omega}S(x,y) p(p-1)\frac{b_{N,p}}{\gamma_{N}}V_{1,k_{n}}^{p-2}H(P_{0},y)\mu^{\frac{N}{q+1}} \frac{\partial V_{1,k_{n}}}{\partial y_{i}}dy
    \\&+\int_{\Omega}S(x,y) p(PV_{1,k_{n}})^{p-1} (\frac{\partial PV_{1,k_{n}}}{\partial y_{i}} -\frac{\partial V_{1,k_{n}}}{\partial y_{i}}   )dy
    \\=&I_{621}+I_{622}+I_{623}.
\end{align*}
By directly computing, we get that
\begin{align*}
    &\int_{B_{\mu^{\alpha}}(P_{0})}q\mu_{1,k_{n}} U_{1,k_{n}}^{q-1}(x)\Psi_{1,k_{n},j}(x) I_{621}dx
    \\=&\int_{\partial B_{\mu^{\alpha}}(P_{0})}\mu_{1,k_{n}}  U_{1,k_{n}}^{q}(x)I_{621}v_{j}ds_{x}-\int_{B_{\mu^{\alpha}}(P_{0})}U_{1,k_{n}}^{q}\int_{\Omega}\mu_{1,k_{n}} \frac{\partial S}{\partial x_{j}}(x,y) p((PV_{1,k_{n}})^{p-1}  -V_{1,k_{n}}^{p-1}
    \\&+ (p-1)\mu^{\frac{N}{q+1}}\frac{b_{N,p}}{\gamma_{N}}V_{1,k_{n}}^{p-2}H(P_{0},y) )\frac{\partial V_{1,k_{n}}}{\partial y_{i}}dy
    \\=& -\mu_{1,k_{n}}^{\frac{N(p+1)}{q+1}+1}(\frac{b_{N,p}}{\gamma_{N}})^{p}\int_{\mathbb{R}^{N}}U_{0,1}^{q}dx\int_{\Omega}\frac{\partial S}{\partial x_{j}}(P_{0},y) p(G^{p-1}(P_{0},y)  -S^{p-1}(P_{0},y)
    \\&+ (p-1)\frac{b_{N,p}}{\gamma_{N}}S^{p-2}(P_{0},y)H(P_{0},y) )
    \frac{\partial S}{\partial y_{i}}(P_{0},y)dy+o(\mu_{1,k_{n}}^{\frac{N(p+1)}{q+1}+1})+O( \mu_{1,k_{n}}^{Np-\alpha}),
    \end{align*}
    \begin{align*}
        &\int_{B_{\mu_{1,k_{n}}^{\alpha}}(P_{0})}q\mu_{1,k_{n}} U_{1,k_{n}}^{q-1}(x)\Psi_{1,k_{n},j}(x) I_{623}dx
        \\=&-\int_{B_{\mu_{1,k_{n}}^{\alpha}}(P_{0})}\mu_{1,k_{n}} U_{1,k_{n}}^{q}\int_{\Omega}\frac{\partial S}{\partial x_{i}}(x,y) p(PV_{1,k_{n}})^{p-1} (\frac{\partial PV_{1,k_{n}}}{\partial y_{j}} -\frac{\partial V_{1,k_{n}}}{\partial y_{i}}   )dy+O( \mu^{Np-\alpha})
        \\=&\mu_{1,k_{n}}^{\frac{N(p+1)}{q+1}+1}(\frac{b_{N,p}}{\gamma_{N}})^{p}\int_{\mathbb{R}^{N}}U_{0,1}^{q}dx\int_{\Omega}\frac{\partial S}{\partial x_{j}}(P_{0},y)pG(P_{0},y)^{p-1}\frac{\partial H}{\partial y_{i}}(P_{0},y)+o(\mu_{1,k_{n}}^{\frac{N(p+1)}{q+1}+1})+O( \mu_{1,k_{n}}^{Np-\alpha}),
    \end{align*}
    and
    \begin{align*}
        &\int_{B_{\mu_{1,k_{n}}^{\alpha}}(P_{0})}q\mu_{1,k_{n}} U_{1,k_{n}}^{q-1}(x)\psi_{i}(x)I_{622}dx
        \\=&-\int_{\partial B_{\mu_{1,k_{n}}^{\alpha}}(P_{0})}\mu_{1,k_{n}}  U_{1,k_{n}}^{q}(x)I_{622}dx+\int_{B_{\mu_{1,k_{n}}^{\alpha}}(P_{0})}U_{1,k_{n}}^{q}(x)
        \\&\times\int_{\Omega}\frac{\partial S}{\partial x_{i}}(x,y) p(p-1)\frac{b_{N,p}}{\gamma_{N}}V_{1,k_{n}}^{p-2}H(P_{0},y)\mu_{1,k_{n}}^{\frac{N}{q+1}+1} \frac{\partial V_{1,k_{n}}}{\partial y_{j}}dy
        \\=&O( \mu_{1,k_{n}}^{Np-\alpha})+\int_{B_{\mu_{1,k_{n}}^{\alpha}}(P_{0})}U_{1,k_{n}}^{q}(x)\int_{\Omega}\frac{\partial S}{\partial x_{i}}(x,y) p(p-1)\frac{b_{N,p}}{\gamma_{N}}V_{1,k_{n}}^{p-2}H(P_{0},y)\mu_{1,k_{n}}^{\frac{N}{q+1}+1} \frac{\partial V_{1,k_{n}}}{\partial y_{j}}dy.
    \end{align*}
    So, we get
    \begin{align*}
        \int_{B_{\mu_{1,k_{n}}^{\alpha}}(P_{0})}q\mu_{1,k_{n}} U_{1,k_{n}}^{q-1}(x)\psi_{i}(x)I_{622}dx+\int_{B_{\mu^{\alpha}}(P_{0})}\mu_{1,k_{n}} U_{1,k_{n}}^{q}(x)I_{312}dx=O( \mu_{1,k_{n}}^{Np-\alpha}).
    \end{align*}
    Thus, we get
     \begin{align*}
        &\int_{B_{\mu_{1,k_{n}}^{\alpha}}(P_{0})}U^{q}_{1}\frac{\partial \sum_{j=1}^{N} b_{j,n}\mu_{1,k_{n}} (\widetilde{P\Psi}_{1,k_{n},j}-\Psi_{1,k_{n},j})  }{\partial x_{i}}(x)
        \\&+
        \int_{B_{\mu_{1,k_{n}}^{\alpha}}(P_{0})}\mu_{1,k_{n}}\sum_{j=1}^{N}qU_{1,k_{n}}^{q-1}\Psi_{1,k_{n},j}\frac{\partial (PU_{d_{1,k_{n}},P_{1,k_{n}}} - U_{1,k_{n}})  }{  \partial x_{i}}
        \\=&\mu_{1,k_{n}}^{\frac{N(p+1)}{q+1}+1}(\frac{b_{N,p}}{\gamma_{N}})^{p}\int_{\mathbb{R}^{N}}U_{0,1}^{q}\sum_{j=1}^{N}b_{j,n}(\int_{\Omega^{c}}\frac{\partial S}{\partial x_{i}}(P_{0},y)pS^{p-1}(P_{0},y)\frac{\partial S}{\partial x_{j}}(P_{0},y)
        \\&+\int_{\Omega}\frac{\partial H}{\partial x_{i}}(P_{0},y)p G(P_{0},y)^{p-1} \frac{\partial G}{\partial x_{j}}(P_{0},y) + \int_{\Omega}\frac{\partial S}{\partial x_{i}}(P_{0},y)pG^{p-1}(P_{0},y)\frac{\partial H}{\partial x_{j}}(P_{0},y)
        \\&-\int_{\Omega}\frac{\partial S}{\partial x_{i}}(P_{0},y)p(\frac{b_{N,p}}{\gamma_{N}})^{p}(   G^{p-1}(P_{0},y) -S^{p-1}(P_{0},y) +(p-1)S^{p-2}(P_{0},y)H(P_{0},y)  )\frac{\partial S}{\partial x_{i}}(P_{0},y)
        \\&+\int_{\Omega^{c}}\frac{\partial^2 S}{\partial x_{j}\partial x_{i}}(P_{0},y)S^{p}(P_{0},y)+\int_{\Omega}\frac{\partial^2 H}{\partial x_{i}\partial x_{j}}(P_{0},y)G^{p}(P_{0},y)
        \\&+\int_{\partial\Omega}\frac{\partial S}{\partial x_{j}}(P_{0},y)(G^{p}(P_{0},y) - S^{p}(P_{0},y))\nu_{i}ds -\int_{\Omega}\frac{\partial S}{\partial x_{j}}(P_{0},y) p(G^{p-1}(P_{0},y)  -S^{p-1}(P_{0},y)
    \\&+ (p-1)\frac{b_{N,p}}{\gamma_{N}}S^{p-2}(P_{0},y)H(P_{0},y) )
     \frac{\partial S}{\partial y_{i}}(P_{0},y)dy
        \\&+\int_{\Omega}\frac{\partial S}{\partial x_{j}}(P_{0},y)pG(P_{0},y)^{p-1}\frac{\partial H}{\partial y_{i}}(P_{0},y))+o(\mu_{1,k_{n}}^{\frac{N(p+1)}{q+1}+1})
        \\=&-\frac{\mu_{1,k_{n}}^{\frac{N(p+1)}{q+1}+1}}{p+1}(\frac{b_{N,p}}{\gamma_{N}})^{p}\int_{\mathbb{R}^{N}}U_{0,1}^{q}\sum_{j=1}^{N}b_{j,n}\frac{\partial^{2}\tau }{\partial x_{i}\partial x_{j}}(P_{0})+o(\mu_{1,k_{n}}^{\frac{N(p+1)}{q+1}+1}),
    \end{align*}
    the last equals sign follows from Lemma \ref{ble6}.
    \end{proof}
\begin{lemma}\label{ble6}
    We have that
    \begin{align*}
        &\frac{\partial^{2} \tau}{\partial x_{i}\partial x_{j}}(x)
        \\= &-(p+1)\bigg(\int_{\Omega}\frac{\partial^{2} H}{\partial x_{i}\partial x_{j}}(x,y) G^{p}(x,y)dy-\int_{\Omega}\frac{\partial H}{\partial x_{i}}(x,y) pG^{p-1}(x,y)\frac{\partial G}{\partial x_{j}}(x,y) dy
        \\&- \int_{\Omega}\frac{\partial H}{\partial x_{j}}(x,y) pG^{p-1}(x,y)\frac{\partial S}{\partial x_{j}}(x,y) - \int_{\Omega}\frac{\partial H}{\partial y_{i}}(x,y) pG^{p-1}(x,y)\frac{\partial S}{\partial x_{j}}(x,y)\bigg).
    \end{align*}
\end{lemma}
\begin{proof}
   Since
    \begin{align*}
       H (x,z) =& \int_{\Omega}G(x,y)G^{p}(z,y)dy -\int_{\mathbb{R}^{N}}S(x,y)S^{p}(z,y)dy
       \\=&-\int_{\Omega^{c}}S(x,y)S^{p}(z,y)dy - \int_{\Omega}H(x,y)G^{p}(z,y)dy+\int_{\Omega}S(x,y)(G^{p}(z,y) - S^{p}(z,y))dy,
    \end{align*}
    We have
    \begin{align*}
        \tau (x) = -\int_{\Omega^{c}}S^{p+1}(x,y)dy - \int_{\Omega}H(x,y)G^{p}(x,y)dy+\int_{\Omega}S(x,y)(G^{p}(x,y) - S^{p}(x,y))dy.
    \end{align*}
    So
    \begin{align*}
        \frac{\partial \tau}{\partial x_{i}} =&-\int_{\Omega^{c}}(p+1)S^{p}(x,y)\frac{\partial S}{\partial x_{i}}(x,y)dy - \int_{\Omega}\frac{\partial H}{\partial x_{i}}(x,y)G^{p}(x,y)dy
        \\&-\int_{\Omega}H(x,y)pG^{p-1}(x,y)\frac{\partial G}{\partial x_{i}}(x,y)dy+\int_{\Omega}\frac{\partial S}{\partial x_{i}}(G^{p}(x,y) - S^{p}(x,y) )dy
        \\&+\int_{\Omega}\Gamma(x,y)(pG^{p-1}(x,y)\frac{\partial G}{\partial x_{i}}(x,y) - p S^{p-1}(x,y)\frac{\partial S}{\partial x_{i}}(x,y) )dy.
    \end{align*}
    Since
    \begin{align*}
       &\int_{\Omega}\frac{\partial S}{\partial x_{i}}(G^{p}(x,y) - S^{p}(x,y) )dy
       \\=&\int_{\Omega}S^{p+1}(x,y)\nu_{i}ds + \int_{\Omega}\Gamma(x,y)(pG^{p-1}(x,y)\frac{\partial G}{\partial y_{i}}(x,y) - p S^{p-1}(x,y)\frac{\partial S}{\partial y_{i}}(x,y) )dy
    \end{align*}
    and
    \begin{align*}
        (p+1)\int_{\Omega^{c}}(p+1)S^{p}(x,y)\frac{\partial S}{\partial x_{i}}(x,y)dy = \int_{\Omega}S^{p+1}(x,y)\nu_{i}ds,
    \end{align*}
    we get that
    \begin{align*}
        \frac{\partial \tau}{\partial x_{i}} = &-\int_{\Omega}\frac{\partial H}{\partial x_{i}}(x,y)G^{p}(x,y)dy- \int_{\Omega}H(x,y)pG^{p-1}(x,y)\frac{\partial G}{\partial x_{i}}(x,y)dy
        \\&-\int_{\Omega}S(x,y)pG^{p-1}(x,y)(\frac{\partial H}{\partial x_{i}}H(x,y) + \frac{\partial H}{\partial y_{i}}(x,y)   )dy.
    \end{align*}
    Since
    \begin{align*}
        &\int_{\Omega}H(x,y)pG^{p-1}(x,y)\frac{\partial S}{\partial x_{i}}(x,y)dy
        \\=&-\int_{\Omega}H(x,y)pG^{p-1}(x,y)\frac{\partial G}{\partial y_{i}}(x,y)dy - \int_{\Omega}H(x,y)pG^{p-1}(x,y)\frac{\partial H}{\partial y_{i}}(x,y)dy
        \\=&\int_{\Omega}\frac{\partial H}{\partial y_{i}}(x,y)G^{p}(x,y) - - \int_{\Omega}H(x,y)pG^{p-1}(x,y)\frac{\partial H}{\partial y_{i}}(x,y)dy,
    \end{align*}
    we get that
    \begin{align*}
        \frac{\partial \tau}{\partial x_{i}} = -(p+1)(\int_{\Omega}\frac{\partial H}{\partial x_{i}}(x,y)G^{p}(x,y)dy+\int_{\Omega}\frac{\partial H}{\partial y_{i}}(x,y)G^{p}(x,y)dy  ).
    \end{align*}
    Thus
    \begin{align*}
        \frac{1}{p+1}\frac{\partial^2 \tau}{\partial x_{j}\partial x_{i}} = &-\int_{\Omega}\frac{\partial^2 H}{\partial x_{j}\partial x_{i}}(x,y)G^{p}(x,y)dy--\int_{\Omega}\frac{\partial H}{\partial x_{i}}(x,y)pG^{p-1}(x,y)\frac{\partial G}{\partial x_{j}}(x,y)dy
        \\&-\int_{\Omega}\frac{\partial^2 H}{\partial x_{j}\partial y_{i}}(x,y)G^{p}(x,y)dy-\int_{\Omega}\frac{\partial H}{\partial y_{i}}(x,y)pG^{p-1}(x,y)\frac{\partial G}{\partial x_{j}}(x,y)dy.
    \end{align*}
    Since
    \begin{align*}
        &-\int_{\Omega}\frac{\partial^2 H}{\partial x_{j}\partial y_{i}}(x,y)G^{p}(x,y)dy
        \\=&\int_{\Omega}\frac{\partial H}{\partial x_{j}}(x,y)pG^{p-1}(x,y)\frac{\partial G}{\partial y_{i}}(x,y)dy,
    \end{align*}
    we get that
    \begin{align*}
        &\frac{\partial^{2} \tau}{\partial x_{i}\partial x_{j}}(x)
        \\= &-(p+1)\bigg(\int_{\Omega}\frac{\partial^{2} H}{\partial x_{i}\partial x_{j}}(x,y) G^{p}(x,y)dy-\int_{\Omega}\frac{\partial H}{\partial x_{i}}(x,y) pG^{p-1}(x,y)\frac{\partial G}{\partial x_{j}}(x,y) dy
        \\&- \int_{\Omega}\frac{\partial H}{\partial x_{j}}(x,y) pG^{p-1}(x,y)\frac{\partial S}{\partial x_{j}}(x,y) - \int_{\Omega}\frac{\partial H}{\partial y_{i}}(x,y) pG^{p-1}(x,y)\frac{\partial S}{\partial x_{j}}(x,y)\bigg).
    \end{align*}
\end{proof}
\end{appendix}

\section*{Data Availability Statements}
All data generated or analysed during this study are included in this article.

\section*{Acknowledgments}
The authors are grateful to the anonymous referees for their careful reading and
valuable comments and suggestions that improved the presentation of the paper.

\end{document}